\documentclass[11pt,a4paper]{article}
\usepackage{mathrsfs, amsmath, amsthm, amssymb, bm, mathtools, thm-restate}
\usepackage{graphicx, xcolor, tikz}% packages for figures
\usetikzlibrary{calc, shapes}
\usepackage{caption, subcaption}
\usepackage{array}
\usepackage{cases}
\usetikzlibrary{calc}
\usetikzlibrary{decorations.markings}
\usetikzlibrary{arrows}
\usetikzlibrary{patterns}
\usetikzlibrary{shapes,snakes}
\pgfdeclarelayer{foreground}
\pgfsetlayers{main,foreground}
\usepackage{pgfmath}
\usepackage{caption}
\usepackage{thm-restate}

\makeatletter
\makeatletter
\def\th@plain{%
	\upshape %\itshape % body font
}
\makeatother

\makeatletter
\renewenvironment{proof}[1][\proofname]{\par
	\pushQED{\qed}%
	\normalfont \topsep6\p@\@plus6\p@\relax
	\trivlist
	\item[\hskip\labelsep
	\bfseries
	#1\@addpunct{.}]\ignorespaces
}{%
\popQED\endtrivlist\@endpefalse
}
\makeatother

\newtheorem{theorem}{Theorem}
\newtheorem{lem}{Lemma}
\newtheorem{cor}{Corollary}
\theoremstyle{definition}
\newtheorem{definition}{Definition}
\newtheorem{proposition}{Proposition}

\newtheorem{example}{Example}

\newtheorem{clm}{Claim}

\usepackage[left=25mm, top=25mm, bottom=25mm, right=25mm]{geometry}
\usepackage[T1]{fontenc}
\usepackage[inline]{enumitem}
\usepackage[square, numbers, sort&compress]{natbib}
\usepackage[pdftex,%
bookmarks=true,%
bookmarksnumbered=true, %% 把章节的标号放入书签中，缺省为false
bookmarksopen=true, %% 打开书签树
plainpages=false,%
pdfpagelabels,%
colorlinks=true, %% 彩色文字链接打开，去掉方框
linkcolor=blue, %% 目录链接的颜色为蓝色
citecolor=blue,%
anchorcolor=green,
urlcolor=black,
breaklinks=true,
hyperindex=true]{hyperref}%% 不能和cite宏包同时使用

\def\int{\mathrm{int}}

\def \int {\rm int}

\def \red {\textcolor{red} }
\def \bl {\textcolor{blue} }

\definecolor{PowderBlue}{RGB}{134,205,235}
\definecolor{DodgerBlue}{RGB}{30,144,255}
\definecolor{SkyBlue}{RGB}{135,206,235}
\definecolor{LightSkyBlue}{RGB}{135,206,250}
\definecolor{DeepSkyBlue}{RGB}{0,191,255}
\definecolor{MediumBlue}{RGB}{0,0,205}
\definecolor{RoyalBlue}{RGB}{65,105,225}
\definecolor{MidnightBlue}{RGB}{25,25,112}
\definecolor{NavyBlue}{RGB}{0,0,128}
\definecolor{LightBlue}{RGB}{173,216,230}

\usepackage[colorinlistoftodos,prependcaption,textsize=scriptsize, textwidth=20mm,]{todonotes}

\usepackage[normalem]{ulem}
\begin{document}
\title{Indicated list colouring game on graphs}
\author{Yangyan Gu \thanks{School of Mathematical Sciences, Zhejiang Normal University, Email: yangyan@zjnu.edu.cn, }
\and 
Yiting Jiang \thanks{Institute of Mathematics, School of Mathematical Sciences, Nanjing Normal University, Email: ytjiang@njnu.edu.cn, Grant number: NSFC 12301442, BK20230373, 23KJB110018.}
\and 
Huan Zhou\thanks{ Universit\'{e} Paris Cit\'{e}, CNRS, IRIF, F-75006, Paris, France; School of Mathematical Sciences, Zhejiang Normal University, Email: huanzhou@zjnu.edu.cn, Grant unmber:ANR‐18‐IdEx‐0001. }
\and Jialu Zhu\thanks{School of Mathematical Sciences, Zhejiang Normal University, Email: jialuzhu@zjnu.edu.cn, }         \and
			Xuding Zhu\thanks{School of Mathematical Sciences, Zhejiang Normal University, Email: xdzhu@zjnu.edu.cn, Grant numbers:   NSFC 12371359. }}
\date{\today}

\maketitle

\begin{abstract}
    Given a graph $G$ and a list assignment $L$ for $G$, the indicated $L$-colouring game on $G$  is played by two players: Ann and Ben. In each round, Ann chooses an uncoloured vertex $v$, and Ben colours $v$ with a colour from $L(v)$ that is not used by its coloured neighbours. If all vertices are coloured, then Ann wins the game. Otherwise after a finite number of rounds, there remains an uncoloured vertex $v$ such that all colours in $L(v)$ have been used by its coloured neighbours,  Ben wins. We say $G$ is indicated $L$-colourable if Ann has a winning strategy for the indicated $L$-colouring game on $G$. For a mapping $g: V(G) \to \mathbb{N}$, 
    we say $G$ is indicated $g$-choosable if  $G$ is indicated $L$-colourable for every list assignment $L$ with $|L(v)| \ge g(v)$ for each vertex $v$, and  $G$ is indicated degree-choosable if $G$ is indicated $g$-choosable for  $g(v) =d_G(v)$  (the degree of $v$).   
    This paper proves that 
    a graph $G$ is not indicated degree-choosable if and only if $G$ is an expanded Gallai-tree - a graph whose maximal connected induced subgraphs with no clique-cut are complete graphs or blow-ups of odd cycles, along with a technical condition (see Definition \ref{def-egt}).   
    This leads to a linear-time algorithm that determines if a graph is indicated degree-choosable.     
    A connected graph $G$ is called an IC-Brooks graph if its indicated chromatic number equals $\Delta(G)+1$.  Every IC-Brooks graph is a regular expanded Gallai-tree. We show that if $r \le 3$, then every $r$-regular expanded Gallai-tree is an IC-Brooks graph. For $r \ge 4$, there are $r$-regular expanded Gallai-trees  that are not IC-Brooks graphs. We give a characterization of   IC-Brooks graphs, and present a linear-time algorithm that determines if a given graph of bounded maximum degree is an IC-Brooks graph.  
\end{abstract}

\section{Introduction}

Assume $G$ is a graph and $L$ is a list assignment that assigns to each vertex $v$ of $G$ a set $L(v)$ of permissible colours.
An {\em	$L$-colouring} of $G$ is a proper   colouring $\phi$  of $G$ with $\phi(v) \in L(v)$ for each vertex $v$. We say that $G$ is {\em $L$-colourable} if there exists an $L$-colouring of $G$, and $G$ is {\em $k$-choosable} if $G$ is $L$-colourable for any list assignment $L$ of $G$ with $|L(v)| \ge k$ for each vertex $v$. More generally, for a function $g: V(G) \to \mathbb{N}$, we say $G$ is   {\em $g$-choosable} if $G$ is $L$-colourable for every list assignment  $L$ with $|L(v)| \ge g(v)$ for all $v \in V(G)$. The {\em choice number} $ch(G)$ of $G$ is the minimum $k$ for which $G$ is $k$-choosable.
		
List colouring of graphs is a natural generalization of classical graph colouring, introduced independently by Erd\H{o}s-Rubin-Taylor \cite{ERT} and Vizing \cite{Vizing76} in 1970's, and has been studied extensively in the literature. 

In this paper, we study a game version of list colouring problem. Given a list assignment $L$ of $G$, 
the {\em indicated $L$-colouring game} on $G$ is played by two players: Ann and Ben. Initially, each vertex is uncoloured.  In each
round, Ann chooses an uncoloured vertex $v$ and Ben assigns to $v$ a colour $c \in L(v)$   that has not been assigned to any of its coloured neighbors. Ann wins the game if all vertices of $G$ are coloured.  Otherwise, after some round, there remains an uncoloured vertex $v$ such that all colours in $L(v)$ have been assigned to neighbours of $v$, and rendering $v$ uncolourable, and  Ben wins.   We say $G$ is {\em indicated $L$-colourable} if Ann has a winning strategy for the indicated $L$-colouring game.

If $L(v)=\{1,2,\ldots, k\}$ for all vertices $v$, then the indicated $L$-colouring game is called an {\em indicated $k$-colouring game } and $G$ is {\em indicated $k$-colourable} if Ann has a winning strategy for the indicated $k$-colouring game.  The {\em indicated chromatic number} of $G$ is defined as 
$$\chi_i(G) = \min\{k: G \text{ is indicated $k$-colourable}\}.$$

The concept of indicated $k$-colouring game and the parameter $\chi_i(G)$  was introduced by Grytczuk and first studied by Grzesik \cite{Grzesik}. 
The list indicated colouring game was introduced in \cite{KKKO}. 

\begin{definition}
    \label{indicated-choicenumber}
   Assume $G$ is a graph and  $g: V(G) \to \mathbb{N}$ is a mapping.   We say $G$ is {\em indicated $g$-choosable } if $G$ is indicated $L$-colourable for any list assignment $L$ of $G$ for which $|L(v)|\ge g(v)$ for each vertex $v$. If $g(v)=k$ for every vertex $v$, then indicated $g$-choosable is also called {\em indicated $k$-choosable}.
The {\em indicated choice number} of $G$, denoted by $ch_i(G)$, is the least integer $k$ such that $G$ is indicated $k$-choosable.
\end{definition}

\begin{definition}
    \label{def-degree}
    A {\em degree-list assignment} of $G$ is a list assignment $L$ with $|L(v)| \ge d_G(v)$ for each $v$. If $|L(v)| = d_G(v)$ for all $v$,   $L$  is a {\em tight} degree-list assignment. 
A graph $G$   is {\em indicated degree-choosable} if it is indicated $g$-choosable for  $g(v)  =d_G(v)$. 
\end{definition}

Game version of many graph theory problems have been studied in the literature. By introducing an adversary in the process of getting a solution to a graph theory problem, it may help to reveal critical structural properties that are obstacles or essential to the existence of a solution. 

A few graph colouring games are studied  in the literature. 

One graph colouring game was invented first by Bram \cite{Gardener}, and re-invented by Bodlaender \cite{Bod}. Given a graph $G$ and a set $C$ of colours. Two players, Alice and Bob, take turns to colour uncoloured vertices of $G$ properly with colours from $C$. Alice's goal is to produce a proper colouring of the whole graph, and Bob's goal is the opposite. The {\em game chromatic number} of $G$ is the minimum number of colours needed for Alice to have a winning strategy in the game. A benchmark problem is to determine the maximum game chromatic number of planar graphs. The current known upper and lower bounds are $17$ and $7$, respectively \cite{KT, Zhu}.

Another graph colouring game is an online version of the list colouring of graphs, introduced by Schauz \cite{Schauz}. Given a graph with each vertex $v$ be given a set of $g(v)$ tokens, the {\em $g$-painting game} is played by two players: Lister and Painter. In each round, Lister chooses a subset $M$ of uncoloured vertices and removes one token from each   vertex in $M$, and Painter colours an independent set $I$ of $G$ that is a subset of $M$.  Painter wins if all vertices are coloured and Lister wins if some uncoloured vertex has no token left. We say $G$ is $g$-paintable if Painter has a winning strategy for the $g$-painting game.  One challenging open problem is to determine the minimum number of vertices in a $k$-choosable but not $k$-paintable graph \cite{KKLZ,KPZ, Zhu2009}.

The Ramsey online game, introduced independently by Beck \cite{Beck}  and   Kurek and Ruci\'{n}ski  \cite{KR}, is not directly related to vertex colouring of graphs. However, its definition is similar to that of indicated colouring of graphs, and this game, together with the game chromatic number of a graph, inspired Grytczuk to invent the indicated $k$-colouring game.  The Ramsey online game has two players: Builder and Painter: Given a set of vertices. In each round, Builder presents an uncoloured edge, and Painter colours the edge by red or blue. Builder's goal is to force a monochromatic graph with certain property in a fewest number of steps (possibly with the restriction that the constructed graph belongs to a restricted family of graphs), and Painter's goal is the opposite. The game was studied a lot in the literature \cite{Conlon2009,CFS,FKRRT,GHK,KK,MSS}, and  results and ideas in the study of the Ramsey online game were used to improve the upper bound for some classical Ramsey numbers \cite{CFS}.  

Compared to the graph colouring games above, indicated list colouring game considers more basic steps in the original procedure of list colouring. The action of colouring one vertex is split into two smaller steps: select one vertex, then colour it. Instead of optimizing the choice of colours, one  optimizes at selecting vertices only. An adversary chooses the   colour for the chosen vertex.  The question in concern is how much this restriction affects the colourability. 

A trivial upper bound $\chi(G) \le \Delta(G)+1$ for the chromatic number still holds for the indicated chromatic number: $\chi_i(G) \le \Delta(G)+1$.
A classical result in chromatic graph theory is Brooks' Theorem, which states that this upper bound is tight for a connected graph $G$ if and only if $G$ is a complete graph or an odd cycle.  
This result was extended by Borodin \cite{Borodin} and 
 Erd\H{o}s,  Rubin and Taylor \cite{ERT}, who 
 proved that a connected graph $G$ is not degree-choosable if and only if $G$ is a Gallai-tree, i.e.,  each block of $G$ is either a complete graph or an odd cycle.

In this paper, we are interested in the corresponding results in the setting of indicated colouring game and indicated list colouring game. Which graphs are indicated degree-choosable? What is the Brooks theorem with respect to the indicated chromatic number?

We call a graph $G$   an  {\em IC-Brooks graph}  if $\chi_i(G)= \Delta(G)+1$. We shall characterize  indicated degree-choosable graphs and IC-Brooks graphs. The characterization of not degree-choosable graphs introduced the graph class of Gallai-trees, which arises in some other context as well. For the characterization of  not indicated degree-choosable graphs, we define a graph class of expanded Gallai-trees. It is a more complex class of graphs, however,  graphs in this class are still simple enough to be recognized in linear-time. The characterization of IC-Brooks graphs is a little complicated, however, it leads to a linear-time algorithm that determines if a given graph of bounded maximum degree is an IC-Brooks graph.

\begin{definition}
    \label{def-badlist}
    Assume $G$ is a connected graph,  $L$ is a degree-list assignment of $G$.  If $G$ is not indicated $L$-colourable, then we say $(G,L)$ is   {\em infeasible}. Otherwise, we say $(G,L)$ is {\em feasible}.
\end{definition}

Note that if $L$ is a list assignment of $G$, $v$ is a vertex with $|L(v)| > d_G(v)$, then $G$ is indicated $L$-colourable if and only if $G-v$ is indicated $L$-colourable. So if $(G,L)$ is infeasible, then $L$ is a tight degree-list assignment of $G$. 

We prove that all infeasible pairs $(G,L)$ can be constructed from $(K_1, L_{\emptyset})$ by three simple operations, where $L_{\emptyset}(v)=\emptyset$ for the unique vertex $v$ of $K_1$. 

Using this result, we give a characterization of not indicated degree-choosable graphs. To state this characterization, we need some definitions.

For a graph $G$ and a vertex $v$ of $G$,  $N_G(v)$ denotes the set of neighbours of $v$ in $G$ and $N_G[v]=N_G(v) \cup \{v\}$ is the close neighbourhood of $v$. For an induced subgraph $H$ of $G$ and $v \in V(H)$, $N_H(v) =N_G(v) \cap V(H)$ and $N_H[v]= N_G[v] \cap V(H)$. For a subset $X$ of $V(G)$, $N_G(X) = \cup_{v \in X}N_G(v)$ and $N_G[X] = \cup_{v \in X}N_G[v]$.   For a positive integer $n$, $[n]$ denotes the set $\{1,2,\ldots, n\}$.

\begin{definition}
    \label{def-expansion}
    Assume $G$ is a graph and $p: V(G) \to \{1,2,\ldots\}$ is a mapping that assigns to each vertex  of $G$ a positive integer. We denote by $G[p]$ the graph obtained from $G$ by blowing up each vertex $v$ into a clique $K(v)$ of order $p(v)$. To be precise, $V(G[p]) = \cup_{v \in V(G)} \{v\} \times [p(v)]$ and for $i \in [p(u)], j \in [p(v)]$, $(u,i)(v,j) \in E(G[p])$ if and only if either $u=v$ and $i \ne j$ or $uv \in E(G)$. The graph $G[p]$ is called a  {\em blow-up} of $G$.
\end{definition}

\begin{definition}
    \label{def-cliquecut}  
A {\em clique cut} in a graph $G$ is a clique $K$ such that $G-V(K)$ is disconnected.  A clique-cut $K$ is {\em elementary} if for any $v \in V(K)$, there are two vertices $x,y \in V(G)-V(K)$ that are disconnected in $G-V(K)$ and connected in $G-(V(K)-\{v\})$.  An {\em expanded block} of $G$  is a maximal connected induced subgraph $H$ of $G$ such that $H$ has no clique-cut. In particular, if $G$ has no clique-cut, then $G$ itself is an expanded block.
\end{definition}

\begin{definition}
    \label{def-rootclique}
    Assume $G$ is a graph and $H=C[p]$ is an expanded block of $G$ which is a blow up of an odd cycle $C$ of length at least $5$, where each vertex $v$ of $C$ is blow up into a clique $K(v)$ of size $p(v)$. If $uv \in E(C)$ and $K=K(u) \cup K(v)$ is an elementary clique-cut of $G$, then $K$ is called a {\em root clique} of $H$. 
\end{definition}

%%%%%%%%%%%%%

\begin{definition}
    \label{def-egt}
     A connected graph $G$ is an expanded Gallai-tree if each expanded block of $G$ is either a complete graph or a blow-up of an odd cycle, and each expanded block that is a blow-up of an odd cycle of length at least 5 has at most one root-clique.   
\end{definition}

For an expanded Gallai-tree $G$, we denote by $\mathcal{B}_G$ the set of expanded blocks that have a root clique.
For $H \in \mathcal{B}_G$, let $\pi_G(H)$ be the root-clique of $H$ in $G$. 
For each clique-cut $K$ of $G$, let 
$$\mathcal{B}_{G,K} = \{H \in \mathcal{B}_G: \pi_G(H)=K \}.$$

It follows from the definition  that a blow-up of a Gallai-tree is an expanded Gallai-tree. However, not every expanded Gallai-tree is a blow-up of a Gallai-tree.   See Fig. \ref{fig-5} for an example.

The following result characterizes graphs that are not indicated degree-choosable.

\begin{theorem}
    \label{thm-indicateddegreechoosable}
    A connected graph $G$ is not indicated degree-choosable if and only if $G$ is an expanded Gallai-tree.
\end{theorem}

By Theorem \ref{thm-indicateddegreechoosable}, an IC-Brooks graph is a regular expanded Gallai-tree. Theorem \ref{thm-charactIC} in Section 8 gives   a characterization of IC-Brooks graphs (some technical definitions are needed in the characterization). Using this characterization, we show that for $r \le 3$, every $r$-regular expanded Gallai-tree is indeed an IC-Brooks graph. For $r \ge 4$, we give examples of  $r$-regular expanded Gallai-trees that are not IC-Brooks graphs. The characterization (Theorem \ref{thm-charactIC}) may seem not very simple. However, it   leads to   a linear-time algorithm that   determines if an $r$-regular expanded Gallai-tree is an IC-Brooks graph. We show that although not every regular expanded Gallai-tree is an IC-Brooks graph,   every expanded Gallai-tree is an induced subgraph of an  IC-Brooks graph. This shows that the structure of IC-Brooks graphs are complicated and we do not expect a very clean and simple characterization. 

Our results have some other consequences.
It was conjectured in
 \cite{KKKO} that all IC-Brooks graphs with $\Delta(G)=r$ contains $K_{r+1}^-$ (complete graph $K_{r+1}$ minus one edge) as a subgraph, and proved that the conjecture holds for graphs $G$ with $\Delta(G) \le 3$ \cite{KKKO}. It follows from our results that the conjecture fails when $\Delta(G) \ge 4$.
 
It was proved in \cite{KKKO} that  every IC-Brooks graph has  $\omega(G) \ge (\Delta(G)+1)/2$. Our results imply an improvement of this  bound:   $\omega(G) \ge \frac 23 (\Delta(G)+1)$. This bound is tight. There are  IC-Brooks graphs with   $\omega(G) = \frac 23 (\Delta(G)+1)$ . 

The relation between the chromatic number and indicated chromatic number of graphs was studied in \cite{Grzesik}. It was proved in \cite{Grzesik} that for any positive integer $n$, there is a graph $G$ with
$\chi(G) \ge n$ and $\chi_i(G) = \frac 43 \chi(G)$. Our results imply that there are  graphs $G$ with arbitrary large chromatic number for which $\chi_i(G) > ( \frac 32 - \epsilon) \chi(G)$ for any $\epsilon > 0$. 
  It remains an open question whether $\chi_i(G) \le C \chi(G)$ for some constant $C$.

\section{Equivalent definition and some basic properties}

This section proves some basic lemmas concerning indicated $L$-colourability of graphs. Lemma \ref{lem-remaining}, whose proof is trivial and omitted, can be viewed as an alternate definition of indicated $L$-colourability, and will be the working definition in the remainder of the paper.

\begin{definition}
    \label{def-partial}
    Assume $G$ is a graph, $L$ is a list assignment of $G$, $X$ is a subset of $V(G)$, and $\phi$ is an $L$-colouring of $G[X]$. Then $L^{\phi}$ is the list assignment of $G-X$ defined as
    \[
    L^{\phi}(v) = L(v) - \phi(N_G(v) \cap X),
    \]
    for each vertex $v \in V(G)-X$. Here $\phi(A)= \{\phi(x): x \in A\}$. If $X=\{v\}$ is a single vertex, $c \in L(v)$ and  $\phi(v)=c$, then we denote $L^{\phi}$ by $L^{v \to c}$. 
\end{definition}

\begin{lem}
    \label{lem-remaining}
    Assume $G$ is a graph and $L$ is a list assignment of $G$. If $G$ has a single vertex $v$, then $G$ is  indicated $L$-colourable if and only if   $L(v) \ne \emptyset$. Assume $|V(G)| \ge 2$.
    Then the following are equivalent:
    \begin{enumerate}
        \item $G$ is not indicated $L$-colourable.
        \item For each vertex $   v \in V(G)$, there is a colour $ c \in L(v)$ such that   $G-v$ is not indicated $L^{v\to c}$-colourable.  
        \item For any subset $X$ of $G$, there is an $L$-colouring $\phi$ of $G[X]$ such that $G-X$ is not indicated $L^{\phi}$-colourable.
    \end{enumerate}  
\end{lem}

\begin{cor}
    \label{cor-leaf}
    Assume $G$ is a graph, $L$ is a list assignment of $G$ and $v$ is a vertex of $G$ such that $L(v) = \{c\}$ consists of a single colour. Then $G$ is not indicated $L$-colourable if and only if $G-v$ is not indicated $L^{v \to c}$-colourable. 
\end{cor}

\begin{lem}
    \label{lem-con}
    Assume $G$ is connected and $L$ is a degree-list assignment of $G$. If $(G,L)$ is infeasible, then $L$ is tight.
\end{lem}
\begin{proof}
    Assume to the contrary that $|L(v)| > d_G(v)$. Order the vertices of $G$ as $v_1,v_2,\ldots, v_n$ so that $v_n=v$ and each vertex $v_i$ with $i < n$ has a neighbour $v_j$ with $j > i$. It is obvious that choosing vertices in this order is a winning strategy for Ann.
\end{proof}

\begin{lem}
    \label{lem-color}
    Assume $(G,L)$ is an infeasible pair and $v$ is a vertex of $G$. If $G-v$ is connected and not indicated $L^{v \to c}$-colourable, then $c \in L(x)$ for all $x \in N_G(v)$.
\end{lem}
\begin{proof}
    It is easy to see that $L^{v \to c}$ is a degree-list assignment of $G-v$. If $c \notin L(x)$ for some $x \in N_G(v)$, then $L^{v \to c}$ is not tight. By Lemma \ref{lem-con},  $G-v$ is indicated $L^{ v \to c}$-colourable, a contradiction.
\end{proof}

\begin{lem}
\label{lem-color1}
    Assume $(G,L)$ is an infeasible pair and $v$ is a vertex of $G$. Then $L(v)\subseteq \cup_{u\in N(v)} L(u)$.
\end{lem}

\begin{proof}
    Assume to the contrary that $c\in L(v)\backslash \cup_{u\in N(v)} L(u)$. For any $L$-colouring $\phi$ of $G-v$, $c\in L^\phi (v)$. Then $G$ is indicated $L$-colourable, a contradiction.
\end{proof}

\section{Three graph operations }

Two  vertices $x$ and $y$ of a graph $G$ are {\em adjacent twins} if $N_G[x] = N_G[y]$. 

\begin{definition}
    Assume $G$ is a graph, $L$ is a degree-list assignment of $G$ and $v$ is a vertex of $G$.  
    \begin{enumerate}
        \item For  a colour $c \notin L(x)$ for any $x \in N_G[v]$, 
        $D_{v,c}(G,L) = (G',L')$ is defined as follows:  
     
      $G'$ is obtained from $G$ by adding     a vertex $v'$ and adding  edges $v'x$ for $x \in N_G[v]$, and $L'$ is the list assignment   of $G'$ defined as 
       \[
       L'(x) = \begin{cases} L(x) \cup \{c\}, &\text{ if $x \in N_G[v]$}, \cr 
       L(v) \cup \{c\}, &\text{ if $x = v'$}, \cr 
       L(x), &\text{ otherwise}. 
       \end{cases}
       \]

     We call the pair $(G',L')$ obtained from $(G,L)$ by {\em duplicating $v$} (see Fig. \ref{fig-1}). 
        \item Assume  $d_G(v)=2$. $T_v(G,L)=(G',L')$ is the pair defined as follows:
        
         $G'$ is obtained from $G$   by replacing  $v$ with a path $[v_1,v_2,v_3]$ and with each of $v_1$ and $v_3$ adjacent to one neighbor of $v$, and $L'$ is a list assignment of $G'$ defined as $L'(x)=L(x)$ for $x \in V(G)-\{v\}$ and $L'(v_i)=L(v)$ for $i=1,2,3$. 

         We call the pair $(G',L')$ obtained from $(G,L)$ by {\em tripling the vertex $v$} (see Fig. \ref{fig-2}).
    \end{enumerate}   
\end{definition}

%If $(G',L')$ is obtained from $(G,L)$ by duplicating $v$, then we say $(G,L)$ is obtained from $(G',L')$ by {\em contracting a pair of adjacent twins}.
%If $(G',L')$ is obtained from $(G,L)$ by tripling $v$, then we say $(G,L)$ is obtained from $(G',L')$ by {\em contracting the path $v_1v_2v_3$}.

\begin{definition}
    \label{def-vertexsum}
    Assume for $i=1,2$, $G_i$ is a graph and $L_i$ is a list assignment of $G_i$ and $v_i$ is a vertex of $G_i$ such that $L_1(v_1) \cap L_2(v_2) = \emptyset$. Let $G$ be obtained from the disjoint union of $G_1$ and $G_2$ by identifying $v_1$ and $v_2$ into a single vertex $v^*$. Let $L$ be the list assignment of $G$ defined as $L(v) = L_i(v)$ if $v \in V(G_i) - \{v_i\}$ and $L(v^*) =L_1(v_1) \cup L_2(v_2)$. Then we say $(G,L)$ is the {\em vertex sum } of $(G_1,L_1)$ and $(G_2, L_2)$ with respect to vertices $v_1, v_2$, and write 
    $(G,L) = (G_1,L_1) \oplus_{(v_1,v_2)} (G_2, L_2)$ (see Fig. \ref{fig-3}). 
\end{definition}

\begin{lem}
    \label{lem-main1}
    If $(G,L)$ is   infeasible, $v \in V(G)$ 
    and $c \notin \cup_{x \in N_G[v]}L(x)$, then 
    $D_{(v,c)}(G,L)$ is  infeasible. 
\end{lem}
\begin{proof}
   Let  $(G', L')=D_{(v,c)}(G,L)$. It is obvious that $L'$ is a degree-list assignment of $G'$.
    We prove by induction on the number of vertices of $G$ that $G'$ is not indicated $L'$-colourable. 
    By Lemma \ref{lem-remaining}, it suffices to show that for any vertex $x$ of $G'$, there is a colour $a \in L'(x)$ such that $G'-x$ is not indicated $L'^{x \to a}$-colourable. 
    
    Assume $x$ is a vertex of $G$. If $x \ne v,v'$, then there is a colour $a \in L(x)$ such that 
    $(G-x, L^{x \to a})$ is infeasible. It is easy to verify that $(G'-x, L'^{x \to a}) = D_{(v,c)}(G-x, L^{x \to a})$. By induction hypothesis, $(G'-x, L'^{x \to a})$ is infeasible.
    
    If $x =v$ or $v'$, then $(G'-x, L'^{x \to c}) = (G,L)$. Hence $(G'-x, L^{x \to c})$ is infeasible. 
    Therefore $(G',L')$ is  infeasible.
\end{proof}

\begin{lem}
    \label{lem-main2}
    If $(G,L)$ is   infeasible, $v \in V(G)$ 
    and $d_G(v)=2$, then 
    $T_v(G,L)$ is   infeasible. 
\end{lem}
\begin{proof}
    Let $(G',L')=T_v(G,L)$.  Again,  it is obvious that $L'$ is a degree-list assignment of $G'$, and it suffices to show that for any vertex $x$ of $G'$, there is a colour $c \in L'(x)$ such that $(G'-x,L'^{x \to c})$ is infeasible.

    Assume $x$ is a vertex of $G'$. If $x \not\in N_{G'}[v_1] \cup N_{G'}[v_3]$, then let $a \in L(x)$ be a colour such that $(G-x,L^{x\to a})$ is infeasible.
    Then $(G'-x, L'^{x \to a})=T_v(G-x,L^{x \to a})$. Hence by induction hypothesis, $(G'-x, L'^{x \to a})$ is infeasible.
    If $x \ne v_2$ is adjacent to $v_1$, then let $a \in L(x)$ be a colour such that $(G-x, L^{x\to a})$ is infeasible. By Lemma \ref{lem-color}, $a \in L'(v_1)=L(v)$. By applying Corollary \ref{cor-leaf} to $v_1, v_2$ successively with respect to $(G-x, L^{x\to a})$, 
    we obtain a pair $(G'-\{x,v_1,v_2\}, L'') = (G-x, L^{x \to a})$, which is infeasible. Hence $(G'-x, L'^{x \to a})$ is infeasible. 

   The case $x \ne v_2$ is adjacent to $v_3$ is symmetric.

   It remains to show that for $x \in \{v_1,v_2,v_3\}$, there is a colour $c \in L(x)$ such that
   $(G'-x,L'^{x \to c})$ is infeasible.

   Assume $L(v) = \{c_1,c_2\}$.
    By Lemma \ref{lem-remaining}, we may assume that $(G-v, L^{v \to c_1})$ is infeasible.
    
    If $x=v_2$, then by applying Corollary \ref{cor-leaf} to   $v_1$ and $v_3$ with   respect to  $(G'-v_2, L'^{v_2 \to c_2})$, we obtain a pair $(G'-\{v_1,v_2,v_3\}, L'') = (G-v, L^{v \to c_1})$, which is infeasible.
    Hence $(G'-v_2, L'^{v_2 \to c_2})$ is infeasible.
    
    If $x=v_1$, then  by applying Corollary \ref{cor-leaf} to $v_2, v_3$ with respect to  $(G'-v_1, L'^{v_1 \to c_1})$, and we obtain a pair $(G'-\{v_1,v_2,v_3\}, L'') = (G-v, L^{v \to c_1})$, which is infeasible.  Hence $(G'-v_1, L'^{v_1 \to c_1})$ is infeasible.

    The case $x =v_3$ is symmetric. 

    Therefore $(G',L')$ is infeasible.
\end{proof}

\begin{lem}
    \label{lem-vs} 
    If $(G_i, L_i)$ is infeasible for $i=1,2$,   $v_i \in V(G_i)$ and $L_1(v_1) \cap L_2(v_2) = \emptyset$, then $(G,L)=(G_1, L_1) \oplus_{(v_1,v_2)} (G_2,L_2)$ is infeasible.
\end{lem}
\begin{proof}
   Let $v^*$ be the new vertex identified by $v_1$ and $v_2$. 
    We need to show that for any vertex $v$ of $G$, there is a colour $c \in L(v)$ such that $(G-v, L^{v\to c})$ is infeasible.
    If $v = v^*$, then let $c \in L_1(v_1)$ be a colour such that $(G_1-v_1, L_1^{v_1 \to c})$ is infeasible. Then $(G-v^*, L^{v^* \to c})$ is infeasible, as $(G_1-v_1, L_1^{v_1 \to c})$ is a connected component of $(G-v^*, L^{v^* \to c})$. 

    Assume $v \ne v^*$. Without loss of generality, assume $v \in V(G_1)$. Let $c \in L_1(v)$ be a colour such that $(G_1-v, L_1^{v \to c})$ is infeasible. Then $(G-v, L^{v \to c}) = (G_1-v, L_1^{v \to c}) \oplus_{(v_1,v_2)} (G_2, L_2)$. By induction hypothesis, 
    $(G-v, L^{v \to c})$ is infeasible.
\end{proof}

\section{Constructible infeasible pairs}

Assume $G$ and $G'$ are graphs, $L$ and $L'$ are list assignment of $G$ and $G'$, respectively. If $f:G \to G'$ is an isomorphism and $g: \bigcup_{v \in V(G)}L(v) \to \bigcup_{v \in V(G')}L'(v)$ is a bijection such that 
for each vertex $v$ of $G$, $L'(f(v)) = \{g(c) : c \in L(v)\}$, then we say $(G,L)$ and $(G',L')$ are isomorphic. We do not distinguish $(G,L)$ and $(G',L')$, unless otherwise specified explicitly.

\begin{definition}
    \label{def-w}
    Let $\mathcal{W}$ be the minimum set of pairs $(G,L)$ for which the following hold:
    \begin{enumerate}
        \item $(K_1, L_{\emptyset}) \in \mathcal{W}$, where $v \in V(K_1)$ and $L_{\emptyset}(v) = \emptyset$.    
        \item If $(G,L) \in \mathcal{W}$ and $v \in V(G)$, then $D_v(G,L) \in \mathcal{W}$.
        \item If $(G,L) \in \mathcal{W}$, $v \in V(G)$ and $d_G(v)=2$, then 
        $T_v(G,L) \in \mathcal{W}$.
        \item If for $i=1,2$, $(G_i, L_i) \in \mathcal{W}$, $v_i \in V(G_i)$ and $L_1(v_1) \cap L_2(v_2) = \emptyset$, then $(G_1, L_1) \oplus_{(v_1,v_2)} (G_2, L_2) \in \mathcal{W}$.
    \end{enumerate}
\end{definition} 

In other words, $\mathcal{W}$ is the smallest family of pairs $(G,L)$ that contains $(K_1, L_{\emptyset})$ and is closed under the three operations defined in the previous section. 
From the results in the previous section it follows that all pairs $(G,L) \in \mathcal{W}$ are infeasible. We call pairs $(G,L) \in \mathcal{W}$ {\em constructible infeasible pairs}. 

We shall show that the inverse is also true: all infeasible pairs are contained in $\mathcal{W}$. In this section, we prove some properties  of pairs contained in $\mathcal{W}$.

First, we present some examples of pairs $(G,L) \in \mathcal{W}$. 
We denote by $  \Theta_{n_1,n_2,\ldots,n_k}$ the   {\em generalized theta graph} consisting of
$k$ internally disjoint paths joining two vertices $u$ and $v$, where the $k$ paths are of lengths $n_1,n_2,\ldots, n_k$, respectively. The two vertices $u,v$ are called the {\em end vertices} of the generalized theta graph.

\begin{example}
    \label{example-1}
    The following are some examples of constructible infeasible pairs.  
    \begin{enumerate}
        \item Starting from $(K_1,L_{\emptyset})$, by duplicating $n-1$ times the vertex $v$, we obtain $(K_n,L_{K_n})$, where $L_{K_n}$ assigns to all vertices the same set of $n-1$ colours. 
        %Note that the list assignment $L_{K_n}$ is uniquely determined (up to a renaming of the colours).

\begin{figure}
 \centering 
\begin{tikzpicture}[scale=1.7, every node/.style={scale=0.8}, baseline]

% 图1: (K1, Lphi)
\draw (0,0) circle (1.5pt) node[below right, black] {\textcolor{blue}{$v$}};
\node[below] at (0,-0.4) {$(K_1, L_\phi)$};

% 箭头1
\draw[->] (0.6,0) -- (0.9,0)node[midway,below]{\scriptsize $D_{v,1}$};

% 图2: (K2, Lk2)
\draw (1.4,0)--(2,0);
\filldraw[fill=white, draw=black] (1.4,0) circle (1.5pt)node[below right, black] {\textcolor{blue}{$v$}};
\filldraw[fill=white, draw=blue] (2,0) circle (1.5pt);

\node[red] at (1.4,0.15) {\scriptsize 1};
\node[red] at (2,0.15) {\scriptsize 1};

\node[below] at (1.7,-0.4) {$(K_2, L_{K_2})$};

% 箭头2
\draw[->] (2.5,0) -- (2.8,0)node[midway,below]{\scriptsize $D_{v,2}$};

% 图3: (K3, Lk3)

\draw (3.3,0) -- (3.8,0.3);
\draw (3.3,0) -- (3.8,-0.3);
\draw (3.8,-0.3) -- (3.8,0.3);

\filldraw[fill=white,draw=black](3.3,0) circle (1.5pt) ;
\filldraw[fill=white,draw=blue](3.8,0.3) circle (1.5pt);
\filldraw[fill=white,draw=black](3.8,-0.3) circle (1.5pt);

\node at (3.3,0.15) {\scriptsize 1\red{2}};
\node at (3.95,0.3) {\scriptsize 1\red{2}};
\node at (3.95,-0.3) {\scriptsize 1\red{2}};
\node at (3.3,-0.15) {\textcolor{blue}{$v$}};

\node[below] at (3.6,-0.4) {$(K_3, L_{K_3})$};

% 箭头3
\draw[->] (4.4,0) -- (4.7,0)node[midway,below]{\scriptsize $D_{v,3}$};

% 图4: (K4, Lk4)

\draw (5.3,0) -- (5.75,0)--(6.1,0.4)--(6.1,-0.4)--(5.75,0);
\draw (6.1,0.4) -- (5.3,0)--(6.1,-0.4);

\filldraw[fill=white,draw=black](5.3,0) circle (1.5pt);
\filldraw[fill=white,draw=blue](5.75,0) circle (1.5pt);
\filldraw[fill=white,draw=black](6.1,0.4) circle (1.5pt);
\filldraw[fill=white,draw=black](6.1,-0.4) circle (1.5pt);

\node at (5.3,0.15) {\scriptsize 12\red{3}};
\node at (5.95,0) {\scriptsize 12\red{3}};
\node at (6.3,0.4) {\scriptsize 12\red{3}};
\node at (6.3,-0.4) {\scriptsize 12\red{3}};
\node at (5.3,-0.15) {\textcolor{blue}{$v$}};

\node[below] at (5.8,-0.4) {$(K_4, L_{K_4})$};

% 箭头5
\draw[->] (6.8,0) -- (7.1,0);

% 图5: (Kn, Lkn)
\node[blue] at (7.6,0) {$\cdots$};
\node[below] at (7.6,-0.4) {$(K_n, L_{K_n})$};

\end{tikzpicture}
   
\caption{Examples of constructible pairs:  duplicating operation.}
 \label{fig-1}
\end{figure}
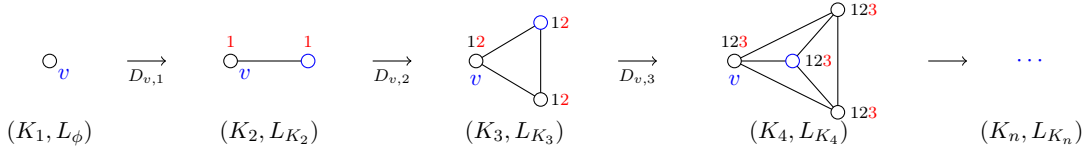

        \item Starting from $(K_3, L_{K_3})$, by repeatedly applying the tripling operation, we obtain $(C_{2n+1}, L_{C_{2n+1}})$, where $C_{2n+1}$ is an odd cycle, and $L_{2n+1}$ assigns to all vertices the same set of 2 colors. 
        %Again the list assignment $L_{C_{2n+1}}$ is uniquely determined.

   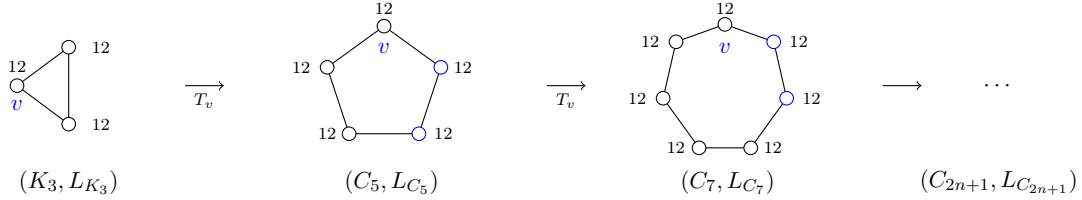
\begin{figure} 
   \centering
\begin{tikzpicture}[scale=1.7, every node/.style={scale=0.8}, baseline]

% 图1: (K3, L)

\draw (0,0) -- (0.4,0.3)--(0.4,-0.3)--(0,0);

\filldraw[fill=white,draw=black](0,0) circle (1.5pt) ;
\filldraw[fill=white,draw=black](0.4,0.3) circle (1.5pt);
\filldraw[fill=white,draw=black](0.4,-0.3) circle (1.5pt);

\node at (0,0.15) {\scriptsize 12};
\node at (0.65,0.3) {\scriptsize 12};
\node at (0.65,-0.3) {\scriptsize 12};
\node at (0,-0.15) {\textcolor{blue}{$v$}};

\node at (0.4,-0.75) {$(K_3, L_{K_3})$};

% 箭头1
\draw[->] (1.3,0) -- (1.6,0)node[midway,below]{\scriptsize $T_v$};

% 图2: (C5, LC5)

 \coordinate (P0) at (2.84, 0.464);
    \coordinate (P1) at (3.28, 0.143);
    \coordinate (P2) at (3.11, -0.375);
    \coordinate (P3) at (2.57, -0.375);
    \coordinate (P4) at (2.4, 0.143);

    % 五边形绘制
    \draw (P0) -- (P1) -- (P2) -- (P3) -- (P4) -- cycle;

   \filldraw[fill=white,draw=black](P0) circle (1.5pt) ;
\filldraw[fill=white,draw=black](P3) circle (1.5pt);
 \filldraw[fill=white,draw=black](P4) circle (1.5pt) ;
\filldraw[fill=white,draw=blue](P1) circle (1.5pt);
\filldraw[fill=white,draw=blue](P2) circle (1.5pt);

\node at (2.84, 0.6) {\scriptsize 12};
\node at (3.45, 0.143) {\scriptsize 12};
\node at (3.3, -0.375) {\scriptsize 12};
\node at (2.4, -0.375){\scriptsize 12};
\node at (2.2, 0.143) {\scriptsize 12};
\node[below] at  (2.84, 0.4) {\textcolor{blue}{$v$}};

\node at (2.9,-0.75) {$(C_5, L_{C_5})$};

\draw[->] (4.1,0) -- (4.4,0)node[midway,below]{\scriptsize $T_v$};

\coordinate (P0) at (5.478, 0.478);
    \coordinate (P1) at (5.1, 0.342);
    \coordinate (P2) at (5, -0.098);
    \coordinate (P3) at (5.278, -0.485);
    \coordinate (P4) at (5.678, -0.485);
    \coordinate (P5) at (5.956, -0.098);
    \coordinate (P6) at (5.856, 0.342);

    % 绘制七边形
    \draw (P0) -- (P1) -- (P2) -- (P3) -- (P4) -- (P5) -- (P6) -- cycle;

\filldraw[fill=white,draw=black](P0) circle (1.5pt) ;
\filldraw[fill=white,draw=black](P3) circle (1.5pt);
 \filldraw[fill=white,draw=black](P4) circle (1.5pt) ;
\filldraw[fill=white,draw=blue](P5) circle (1.5pt);
\filldraw[fill=white,draw=blue](P6) circle (1.5pt);
\filldraw[fill=white,draw=black](P2) circle (1.5pt);
 \filldraw[fill=white,draw=black](P1) circle (1.5pt) ;

\node at (5.478, 0.6) {\scriptsize 12};
\node at (4.9, 0.342) {\scriptsize 12};
\node at (4.8, -0.098) {\scriptsize 12};
\node at (5.1, -0.485) {\scriptsize 12};
\node at (5.85, -0.485) {\scriptsize 12};
\node at (6.15, -0.098){\scriptsize 12};
\node at (6.05, 0.342) {\scriptsize 12};

\node[below] at  (5.478, 0.4) {\textcolor{blue}{$v$}};

\node at (5.5,-0.75) {$(C_7, L_{C_7})$};

\draw[->] (6.7,0) -- (7,0);

\node at (7.6,0) {$\cdots$};
\node at (7.6,-0.75) {$(C_{2n+1}, L_{C_{2n+1}})$};

\end{tikzpicture}

\caption{Examples of constructible pairs: tripling operation.}
\label{fig-2}
\end{figure}

        \item Using copies of complete graphs and odd cycles, applying vertex sum operations, we obtain $(G,L)$, where $G$ is any Gallai-tree, and for each block $B$ of $G$ that is $r$-regular, $C_B$ is a set of $r$ colours, and  $L(v) = \bigcup_{  B \text{ is a block of $G$ containing $v$}   } C_B.$ 

   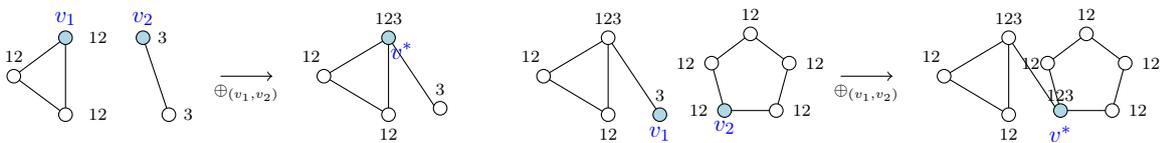
\begin{figure}
\begin{tikzpicture}[scale=1.7, every node/.style={scale=0.8}, baseline]

% 图1: (K3, L)

\draw (0,0) -- (0.4,0.3)--(0.4,-0.3)--(0,0);

\filldraw[fill=white,draw=black](0,0) circle (1.5pt) ;
\filldraw[fill=LightBlue,draw=black](0.4,0.3) circle (1.5pt);
\filldraw[fill=white,draw=black](0.4,-0.3) circle (1.5pt);

\node at (0,0.15) {\scriptsize 12};
\node at (0.65,0.3) {\scriptsize 12};
\node at (0.65,-0.3) {\scriptsize 12};
\node at (0.4,0.45) {\textcolor{blue}{$v_1$}};

\draw (1,0.3) -- (1.2,-0.3);

\filldraw[fill=LightBlue,draw=black](1,0.3) circle (1.5pt) ;
\filldraw[fill=white,draw=black](1.2,-0.3) circle (1.5pt);
\node at (1.15,0.3) {\scriptsize 3};
\node at (1.35,-0.3) {\scriptsize 3};
\node at (1,0.45) {\textcolor{blue}{$v_2$}};

% 箭头1
\draw[->] (1.6,0) -- (2,0)node[midway,below]{\scriptsize $\oplus_{(v_1,v_2)}$};

\draw (2.4,0) -- (2.9,-0.3)--(2.9,0.3)--(2.4,0);
\draw (2.9,0.3) -- (3.3,-0.3);

\filldraw[fill=white,draw=black](2.4,0) circle (1.5pt) ;
\filldraw[fill=white,draw=black](2.9,-0.3) circle (1.5pt);
\filldraw[fill=LightBlue,draw=black](2.9,0.3) circle (1.5pt);
\filldraw[fill=white,draw=black](3.3,-0.25) circle (1.5pt);

\node at (2.9,0.45) {\scriptsize 123};
\node at (2.9,-0.45) {\scriptsize 12};
\node at (2.3,0.15) {\scriptsize 12};
\node at (3,0.2) {\textcolor{blue}{$v^*$}};
\node at (3.3,-0.1) {\scriptsize 3};

%======图2

\draw (4.1,0) -- (4.6,-0.3)--(4.6,0.3)--(4.1,0);
\draw (4.6,0.3) -- (5,-0.3);

\filldraw[fill=white,draw=black](4.1,0) circle (1.5pt) ;
\filldraw[fill=white,draw=black](4.6,-0.3) circle (1.5pt);
\filldraw[fill=LightBlue,draw=black](5,-0.3) circle (1.5pt);
\filldraw[fill=white,draw=black](4.6,0.3) circle (1.5pt);

\node at (4.6,0.45) {\scriptsize 123};
\node at (4.6,-0.45) {\scriptsize 12};
\node at (4,0.15) {\scriptsize 12};
\node at (5,-0.45) {\textcolor{blue}{$v_1$}};
\node at (5,-0.15) {\scriptsize 3};

\coordinate (P0) at (5.7, 0.3299);
\coordinate (P1) at (5.4, 0.1019);
\coordinate (P2) at (5.5, -0.2678);
\coordinate (P3) at (5.9, -0.2678);
\coordinate (P4) at (6, 0.1019);

    % 五边形绘制
    \draw (P0) -- (P1) -- (P2) -- (P3) -- (P4) -- cycle;

   \filldraw[fill=white,draw=black](P0) circle (1.5pt) ;
\filldraw[fill=white,draw=black](P4) circle (1.5pt);
 \filldraw[fill=LightBlue,draw=black](P2) circle (1.5pt) ;
\filldraw[fill=white,draw=black](P1) circle (1.5pt);
\filldraw[fill=white,draw=black](P3) circle (1.5pt);

\node at (5.7, 0.48) {\scriptsize 12};
\node at (5.2, 0.1019) {\scriptsize 12};
\node at (5.3, -0.2678) {\scriptsize 12};
\node at (6.1, -0.2678){\scriptsize 12};
\node at (6.2, 0.1019) {\scriptsize 12};
\node at  (5.5, -0.4) {\textcolor{blue}{$v_2$}};

% 箭头2
\draw[->] (6.4,0) -- (6.8,0)node[midway,below]{\scriptsize $\oplus_{(v_1,v_2)}$};

\draw (7.2,0) -- (7.7,-0.3)--(7.7,0.3)--(7.2,0);
\draw (7.7,0.3) -- (8.1,-0.3);

\filldraw[fill=white,draw=black](7.2,0) circle (1.5pt) ;
\filldraw[fill=white,draw=black](7.7,-0.3) circle (1.5pt);
\filldraw[fill=white,draw=black](7.7,0.3) circle (1.5pt);

\node at (7.7,0.45) {\scriptsize 123};
\node at (7.7,-0.45) {\scriptsize 12};
\node at (7.1,0.15) {\scriptsize 12};
\node at (8.1,-0.45) {\textcolor{blue}{$v^*$}};
\node at (8.1,-0.15) {\scriptsize 123};

\coordinate (P0) at (8.3, 0.3299);
\coordinate (P1) at (8, 0.1019);
\coordinate (P2) at (8.1, -0.2678);
\coordinate (P3) at (8.5, -0.2678);
\coordinate (P4) at (8.6, 0.1019);

    % 五边形绘制
    \draw (P0) -- (P1) -- (P2) -- (P3) -- (P4) -- cycle;

   \filldraw[fill=white,draw=black](P0) circle (1.5pt) ;
\filldraw[fill=white,draw=black](P4) circle (1.5pt);
 \filldraw[fill=LightBlue,draw=black](P2) circle (1.5pt) ;
\filldraw[fill=white,draw=black](P1) circle (1.5pt);
\filldraw[fill=white,draw=black](P3) circle (1.5pt);

\node at (8.3, 0.48) {\scriptsize 12};
\node at (7.87, 0.1019) {\scriptsize 12};
%\node at (5.3, -0.2678) {\scriptsize 12};
\node at (8.7, -0.2678){\scriptsize 12};
\node at (8.8, 0.1019) {\scriptsize 12};

\end{tikzpicture}

\caption{Examples of constructible pairs: vertex-sum operation.}
\label{fig-3}
\end{figure}

        \item Assume $K_{1,n}$ have vertices $v, u_1,u_2,\ldots, u_n$, where $v$ is the center, and $u_i$ are the leaf vertices. Starting from $(K_{1,n}, L_{K_{1,n}})$, Duplicating the center $v$, we  obtain $(\Theta_{1,2,\ldots, 2}, L_{\Theta_{1,2,\ldots, 2}})$. Then by 
        applying a sequence of tripling operation to the degree 2 vertices of $\Theta_{1,2,\ldots, 2}$, we obtain $(\Theta_{1, 2k_1, \ldots, 2k_n}, L_{\Theta_{1,2k_1,\ldots, 2k_n}})$.

   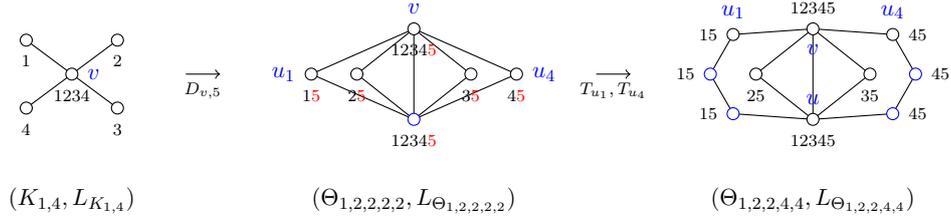
\begin{figure}
       \centering  

\begin{tikzpicture}[scale=1.5, every node/.style={scale=0.8}, baseline]

% 图1: (K1,4, L)

\coordinate (P0) at (0.4, 0);
    \coordinate (P1) at (0, 0.3);
    \coordinate (P2) at (0.8,0.3);
    \coordinate (P3) at (0.8, -0.3);
    \coordinate (P4) at (0, -0.3);

    \coordinate (P5) at (2.5, 0);
    \coordinate (P6) at (2.9, 0);
    \coordinate (P7) at (3.4, 0.4);
    \coordinate (P8) at (3.4, -0.4);
    \coordinate (P9) at (3.9, 0);
     \coordinate (P10) at (4.3, 0);

  \draw (P1) -- (P0) -- (P3);
  \draw(P4) -- (P0) -- (P2);

\draw (P5)--(P7)--(P6)--(P8)--(P5);
\draw (P7)--(P9)--(P8)--(P10)--(P7)--(P8);

% 箭头1
\draw[->] (1.4,0) -- (1.7,0)node[midway,below]{\scriptsize $D_{v,5}$};

% 使用 foreach 循环生成点
\foreach \i in {0, 1, 2, 3, 4, 5, 6, 7, 9, 10} {
    \filldraw[fill=white,draw=black] (P\i) circle (1.5pt);}

% 单独处理 P7
\filldraw[fill=white,draw=blue] (P8) circle (1.5pt);

\foreach \i in {1, 2, 3, 4} {
    \node[below=1mm] at (P\i) {\scriptsize \i};}

\node[below=1mm] at (P0) {\scriptsize 1234};
\node[right=1mm] at (P0) {\bl{$v$}};

\foreach \i in {7, 8} {
    \node[below=1mm] at (P\i) {\scriptsize 1234\red{5}};}
\node[above=1mm] at (P7) {\bl{$v$}};
\node[below=1mm] at (P5) {\scriptsize 1\red{5}};
\node[below=1mm] at (P6) {\scriptsize 2\red{5}};
\node[below=1mm] at (P9) {\scriptsize 3\red{5}};
\node[below=1mm] at (P10) {\scriptsize 4\red{5}};
\node[left=1mm] at (P5) {\bl{$u_1$}};
\node[right=1mm] at (P10) {\bl{$u_4$}};
% 箭头1
\draw[->] (5,0) -- (5.3,0)node[midway,below]{\scriptsize $T_{u_1},T_{u_4}$};

\coordinate (V0) at (6.2, 0.35);
    \coordinate (V1) at (6, 0);
    \coordinate (V2) at (6.2,-0.35);
    \coordinate (V3) at (6.4, 0);
    \coordinate (V4) at (6.9, 0.4);

    \coordinate (V5) at (6.9,-0.4);
    \coordinate (V6) at (7.4, 0);
    \coordinate (V7) at (7.6, 0.35);
    \coordinate (V8) at (7.8, 0);
    \coordinate (V9) at (7.6, -0.35);

  \draw (V0) -- (V1) -- (V2)--(V5)--(V3)--(V4)--(V6)--(V5)--(V9)--(V8)--(V7)--(V4)--(V5);
 \draw (V0) -- (V4);

% 使用 foreach 循环生成点
\foreach \i in {0, 3, 4, 5, 6, 7} {
    \filldraw[fill=white,draw=black] (V\i) circle (1.5pt);}

\foreach \i in {1, 2, 8, 9} {
    \filldraw[fill=white,draw=blue] (V\i) circle (1.5pt);}

\foreach \i in {0, 1, 2} {
    \node[left=1mm] at (V\i) {\scriptsize 15};}
    
\foreach \i in {7, 8, 9} {
    \node[right=1mm] at (V\i) {\scriptsize 45};}
    \node[below=1mm] at (V5) {\scriptsize 12345};
    \node[above=1mm] at (V4) {\scriptsize 12345};
 \node[below=1mm] at (V3) {\scriptsize 25};
  \node[below=1mm] at (V6) {\scriptsize 35};
    
\node[below=1mm] at (V4) {\bl{$v$}};
\node[above=1mm] at (V5) {\bl{$u$}};
\node[above=1mm] at (V0) {\bl{$u_1$}};
\node[above=1mm] at (V7) {\bl{$u_4$}};

\node[below=1.4cm] at (P0) {$(K_{1,4}, L_{K_{1,4}})$};
\node[below=0.8cm] at (P8) {$(\Theta_{1,2,2,2,2}, L_{\Theta_{1,2,2,2,2}})$};
\node[below=0.8cm] at (V5) {$(\Theta_{1,2,2,4,4}, L_{\Theta_{1,2,2,4,4}})$};
\end{tikzpicture}
\caption{Construction of theta graphs}
\label{fig-4}
\end{figure}

        \item Starting with $(\Theta_{1,4,4},L_{\Theta_{1,4,4}})$, and duplicating the   vertices $x_2$ and $y_2$, we obtain the pair $(G,L) \in \mathcal{W}$ as illustrated in Fig. \ref{fig-5}.

\begin{figure}
    \centering
\begin{tikzpicture}[scale=1.5, every node/.style={scale=0.8}, baseline]

% 图1: (K1,4, L)

\coordinate (P0) at (0, 0.4);
    \coordinate (P1) at (0, 0);
    \coordinate (P2) at (0,-0.4);
    \coordinate (P3) at (0.6, 0.4);
    \coordinate (P4) at (0.6, -0.4);

    \coordinate (P5) at (1.2, 0.4);
    \coordinate (P6) at (1.2, 0);
    \coordinate (P7) at (1.2, -0.4);

  \draw (P0) -- (P1) -- (P2) -- (P4) -- (P3) -- (P0);
  \draw(P3) -- (P5) -- (P6) -- (P7) -- (P4);

% 使用 foreach 循环生成点
\foreach \i in {0, 1, 2, 3, 4, 5, 6, 7} {
    \filldraw[fill=white,draw=black] (P\i) circle (1.5pt);}

\foreach \i in {0, 1, 2} {
    \node[left=1mm] at (P\i) {\scriptsize 13};}
  \foreach \i in {5, 6, 7} {
    \node[right=1mm] at (P\i) {\scriptsize 23};}  
  \foreach \i in {3, 4} {
    \node[below=1mm] at (P\i) {\scriptsize 123};}  

    \node[right=1mm] at (P1){\bl{$x_2$}};
     \node[left=1mm] at (P6){\bl{$y_2$}};

% 箭头1
\draw[->] (2.3,0) -- (2.7,0)node[midway,below]{\scriptsize $D_{x_2,2},D_{y_2,1}$};

\coordinate (V0) at (4.1, 0.4);
    \coordinate (V1) at (3.8, 0);
    \coordinate (V2) at (4.4,0);
    \coordinate (V3) at (4.1,-0.4);
    \coordinate (V4) at (4.7, 0.4);
    \coordinate (V5) at (4.7,-0.4);
    \coordinate (V6) at (5.3, 0.4);
    \coordinate (V7) at (5, 0);
    \coordinate (V8) at (5.6, 0);
    \coordinate (V9) at (5.3, -0.4);

  \draw (V0) -- (V1) -- (V2)--(V0)--(V4)--(V5)--(V3)--(V1);
  
  \draw(V4)--(V6)--(V8)--(V7)--(V6);
 \draw (V3) -- (V2);
 \draw (V7) -- (V9)--(V8);
  \draw (V5) -- (V9);

% 使用 foreach 循环生成点
\foreach \i in {0, 1, 3, 4, 5, 6, 8, 9} {
    \filldraw[fill=white,draw=black] (V\i) circle (1.5pt);}

\foreach \i in {2, 7} {
    \filldraw[fill=white,draw=blue] (V\i) circle (1.5pt);}

\node[above=1mm] at (V0) {\scriptsize 13\red{2}};
\node[left=1mm] at (V1) {\scriptsize 13\red{2}};
\node[below] at (V2) {\scriptsize 13\red{2}};
\node[below=1mm] at (V3) {\scriptsize 13\red{2}};
    
  \foreach \i in {4, 5} {
    \node[below=1mm] at (V\i) {\scriptsize 123};}  

\node[above=1mm] at (V6) {\scriptsize 23\red{1}};
\node[below] at (V7) {\scriptsize 23\red{1}};
\node[right=1mm] at (V8) {\scriptsize 23\red{1}};
\node[below=1mm] at (V9) {\scriptsize 23\red{1}};

    \node[below] at (V1){\bl{$x_2$}};
     \node[below] at (V8){\bl{$y_2$}};

\end{tikzpicture}
 \caption{Construction of a cubic IC-Brooks graph}
    \label{fig-5}
\end{figure}
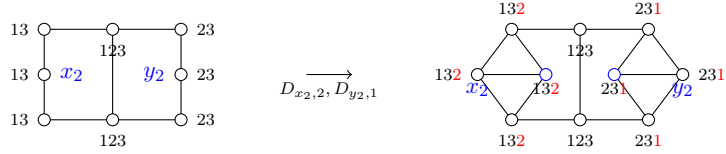

        \item Starting from $(C_{6n+3}, L_{C_{6n+3}})$, by repeatedly duplicating each of the vertices of $C_{6n+3}$ into a clique of size $m$, we obtain an infeasible pair  $(G,L) \in \mathcal{W}$  as illustrated in Fig. \ref{fig-6}.

\begin{figure}
    \centering

\begin{tikzpicture}[scale=1.5, every node/.style={scale=0.8}, baseline] % 调整 scale 以控制图大小
    % --- 图 1: 只有外圆 ---
    \begin{scope}[xshift=0cm] % 向左平移图 1
        % 圆的半径
        \def\radius{0.8}
        % 顶点总数
        \def\numVertices{9}
        
        % 画圆
        \draw (0, 0) circle (\radius);
        
        % 绘制顶点和标签
        \foreach \i in {1,...,\numVertices} {
            % 计算角度
            \pgfmathsetmacro{\angle}{360/\numVertices * (\i - 1)}
            % 计算顶点坐标
            \pgfmathsetmacro{\x}{\radius * cos(\angle)}
            \pgfmathsetmacro{\y}{\radius * sin(\angle)}
            % 画顶点
            \filldraw[fill=white,draw=black] (\x, \y) circle (1.2pt);
            % 添加顶点标签
            \node at ({1.2*\x}, {1.2*\y}) {\scriptsize 12};
            \ifnum\i=3
             \node at ({0.8*\x}, {0.8*\y}) {\bl{$v$}};
             \else
             \node[white] at ({0.8*\x}, {0.8*\y}) {$v$};
             \fi 
        }
    \end{scope}

% 箭头1
\draw[->] (1.7,0) -- (2,0)node[midway,below]{\scriptsize $D_{v,3}$};
% 箭头1
\draw[->] (5.4,0) -- (5.7,0)node[midway,below]{\scriptsize $...$};

    % --- 图 2: 外圆加一个顶点 ---
    \begin{scope}[xshift=3.7cm] % 向右平移图 2
        % 圆的半径
        \def\outerRadius{0.8} % 外圆半径
        \def\innerRadius{0.5} % 内圆半径
        % 顶点总数
        \def\numVertices{9}
        
        % 画外圆
        \draw (0, 0) circle (\outerRadius);
        % 画内圆
        \draw[white] (0, 0) circle (\innerRadius);
        
        % 绘制连接的曲线与边
        
        \foreach \i in {1,...,\numVertices} {
            % 计算角度
            \pgfmathsetmacro{\angle}{360/\numVertices * (\i - 1)}
            % 计算外圆顶点坐标
            \pgfmathsetmacro{\outerX}{\outerRadius * cos(\angle)}
            \pgfmathsetmacro{\outerY}{\outerRadius * sin(\angle)}
            % 计算内圆顶点坐标
            \pgfmathsetmacro{\innerX}{\innerRadius * cos(\angle)}
            \pgfmathsetmacro{\innerY}{\innerRadius * sin(\angle)}

            % 连边（内圆与对应外圆顶点）
            \ifnum\i=3  % 指定顶点 3 作为改变颜色的例子
        \draw (\innerX, \innerY) -- (\outerX, \outerY); % 红色边
    \else
        \draw[white] (\innerX, \innerY) -- (\outerX, \outerY); % 其他保持默认颜色
    \fi

            % 连边（内圆顶点与外圆邻居顶点）
            \pgfmathsetmacro{\prevAngle}{360/\numVertices * (\i - 2)} % 前一个顶点角度
            \pgfmathsetmacro{\nextAngle}{360/\numVertices * (\i)}     % 后一个顶点角度
            % 计算邻居顶点坐标
            \pgfmathsetmacro{\prevX}{\outerRadius * cos(\prevAngle)}
            \pgfmathsetmacro{\prevY}{\outerRadius * sin(\prevAngle)}
            \pgfmathsetmacro{\nextX}{\outerRadius * cos(\nextAngle)}
            \pgfmathsetmacro{\nextY}{\outerRadius * sin(\nextAngle)}

            % 定义曲线控制点参数
            \pgfmathsetmacro{\controlFactor}{0.55}

            % 连边（内圆到外圆的后一个邻居） - 曲线
    \ifnum\i=3  % 指定顶点 5 的连接作为改变颜色的例子
        \pgfmathsetmacro{\controlX}{\controlFactor * (\innerX + \nextX)}
        \pgfmathsetmacro{\controlY}{\controlFactor * (\innerY + \nextY)}
        \draw (\innerX, \innerY) .. controls (\controlX, \controlY) .. (\nextX, \nextY); % 蓝色曲线
    \else
        \pgfmathsetmacro{\controlX}{\controlFactor * (\innerX + \nextX)}
        \pgfmathsetmacro{\controlY}{\controlFactor * (\innerY + \nextY)}
        \draw[white] (\innerX, \innerY) .. controls (\controlX, \controlY) .. (\nextX, \nextY); % 默认颜色曲线
    \fi

            % 连边（内圆到外圆的前一个邻居） - 曲线
        \ifnum\i=3    \pgfmathsetmacro{\controlX}{\controlFactor * (\innerX + \prevX)}
            \pgfmathsetmacro{\controlY}{\controlFactor * (\innerY + \prevY)}
            
            \draw (\innerX, \innerY) .. controls (\controlX, \controlY) .. (\prevX, \prevY);
\else
        \pgfmathsetmacro{\controlX}{\controlFactor * (\innerX + \prevX)}
            \pgfmathsetmacro{\controlY}{\controlFactor * (\innerY + \prevY)}
        \draw[white] (\innerX, \innerY) .. controls (\controlX, \controlY) .. (\prevX, \prevY); % 默认颜色曲线
    \fi
            
        }

   % 在顶层绘制顶点
\begin{pgfonlayer}{foreground}
\foreach \i in {1,...,\numVertices} {
    % 计算角度
    \pgfmathsetmacro{\angle}{360/\numVertices * (\i - 1)}
    % 外圆顶点坐标
    \pgfmathsetmacro{\outerX}{\outerRadius * cos(\angle)}
    \pgfmathsetmacro{\outerY}{\outerRadius * sin(\angle)}
    % 内圆顶点坐标
    \pgfmathsetmacro{\innerX}{\innerRadius * cos(\angle)}
    \pgfmathsetmacro{\innerY}{\innerRadius * sin(\angle)}

    % 绘制内圆顶点
    \ifnum\i=3
        \filldraw[fill=white, draw=black] (\innerX, \innerY) circle (1.2pt);
         \node at ({0.8 * \innerX}, {0.8 * \innerY}) {\scriptsize 12\red{3}}; % 外圆顶点标签
    \else
        \filldraw[fill=white, draw=white] (\innerX, \innerY) circle (1.2pt); % 内圆顶点
    \fi
    
    % 绘制外圆顶点标签
 \ifnum\i=3 
        \node at ({1.2 * \outerX}, {1.2 * \outerY}) {\scriptsize 12\red{3}}; % 外圆顶点标签
    \else
    \ifnum\i=4
        \node at ({1.2 * \outerX}, {1.2 * \outerY}) {\scriptsize 12\red{3}}; % 外圆顶点标签
          \else
        \ifnum\i=2
        \node at ({1.2 * \outerX}, {1.2 * \outerY}) {\scriptsize 12\red{3}}; % 外圆顶点标签
        \else
        \node at ({1.2 * \outerX}, {1.2 * \outerY}) {\scriptsize 12};
    \fi
    \fi
    \fi 

    % 绘制外圆顶点
    \filldraw[fill=white, draw=black] (\outerX, \outerY) circle (1.2pt); % 外圆顶点
}
\end{pgfonlayer}

    \end{scope}

    % --- 图 2: 外圆和内圆的连接 ---
    \begin{scope}[xshift=7.4cm] % 向右平移图 2
        % 圆的半径
        \def\outerRadius{0.9} % 外圆半径
        \def\innerRadius{0.6} % 内圆半径
        % 顶点总数
        \def\numVertices{9}
        
        % 画外圆
        \draw (0, 0) circle (\outerRadius);
        % 画内圆
        \draw (0, 0) circle (\innerRadius);
        
        % 绘制连接的曲线与边
        \foreach \i in {1,...,\numVertices} {
            % 计算角度
            \pgfmathsetmacro{\angle}{360/\numVertices * (\i - 1)}
            % 计算外圆顶点坐标
            \pgfmathsetmacro{\outerX}{\outerRadius * cos(\angle)}
            \pgfmathsetmacro{\outerY}{\outerRadius * sin(\angle)}
            % 计算内圆顶点坐标
            \pgfmathsetmacro{\innerX}{\innerRadius * cos(\angle)}
            \pgfmathsetmacro{\innerY}{\innerRadius * sin(\angle)}

            % 连边（内圆与对应外圆顶点）
            \draw (\innerX, \innerY) -- (\outerX, \outerY);
            
            % 连边（内圆顶点与外圆邻居顶点）
            \pgfmathsetmacro{\prevAngle}{360/\numVertices * (\i - 2)} % 前一个顶点角度
            \pgfmathsetmacro{\nextAngle}{360/\numVertices * (\i)}     % 后一个顶点角度
            % 计算邻居顶点坐标
            \pgfmathsetmacro{\prevX}{\outerRadius * cos(\prevAngle)}
            \pgfmathsetmacro{\prevY}{\outerRadius * sin(\prevAngle)}
            \pgfmathsetmacro{\nextX}{\outerRadius * cos(\nextAngle)}
            \pgfmathsetmacro{\nextY}{\outerRadius * sin(\nextAngle)}

            % 定义曲线控制点参数
            \pgfmathsetmacro{\controlFactor}{0.55}

            % 连边（内圆到外圆的后一个邻居） - 曲线
            \pgfmathsetmacro{\controlX}{\controlFactor * (\innerX + \nextX)}
            \pgfmathsetmacro{\controlY}{\controlFactor * (\innerY + \nextY)}
            \draw (\innerX, \innerY) .. controls (\controlX, \controlY) .. (\nextX, \nextY);

            % 连边（内圆到外圆的前一个邻居） - 曲线
            \pgfmathsetmacro{\controlX}{\controlFactor * (\innerX + \prevX)}
            \pgfmathsetmacro{\controlY}{\controlFactor * (\innerY + \prevY)}
            \draw (\innerX, \innerY) .. controls (\controlX, \controlY) .. (\prevX, \prevY);
        }

        % 在顶层绘制顶点
        \begin{pgfonlayer}{foreground}
        \foreach \i in {1,...,\numVertices} {
            % 计算角度
            \pgfmathsetmacro{\angle}{360/\numVertices * (\i - 1)}
            % 外圆顶点坐标
            \pgfmathsetmacro{\outerX}{\outerRadius * cos(\angle)}
            \pgfmathsetmacro{\outerY}{\outerRadius * sin(\angle)}
            % 内圆顶点坐标
            \pgfmathsetmacro{\innerX}{\innerRadius * cos(\angle)}
            \pgfmathsetmacro{\innerY}{\innerRadius * sin(\angle)}

            % 绘制顶点
            \filldraw[fill=white,draw=black] (\innerX, \innerY) circle (1.2pt); % 内圆顶点
            \filldraw[fill=white,draw=black] (\outerX, \outerY) circle (1.2pt); % 外圆顶点
        }
        \end{pgfonlayer}
    \end{scope}
    
\end{tikzpicture}
 
    \caption{Example of constructible pairs: blow-up of an odd cycle}
    \label{fig-6}
\end{figure}
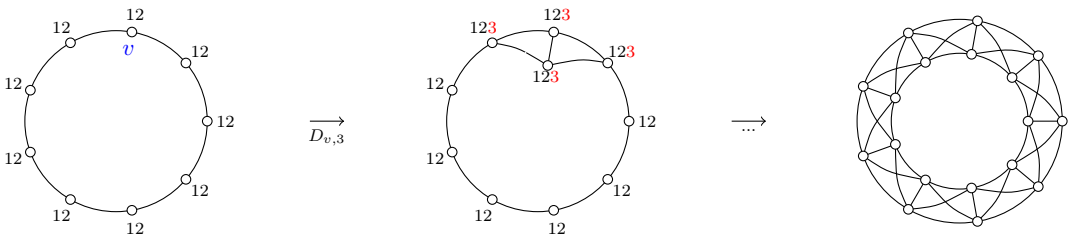

    \end{enumerate}
    
\end{example}

It was conjectured    in \cite{KKKO} that if $k \ge 3$ and $G$ is a $k$-regular graph which does not contain $K_{k+1}^-$ as an induced subgraph, then $\chi_i(G) \le k$. The conjecture was verified for $k=3$. However, the graph $G$ constructed in (6) of Example \ref{example-1} disproves this conjecture for $k \ge 4$. The graph $G$ in that example is $k$-regular for $k=3m-1$. By choosing appropriate colours in the duplication process, the list $L$ can be chosen so that $L(v)$ be the same set of $k$ colours. So, the graph $G$ is not indicated $k$-colourable, and hence $\chi_i(G) = 3m=k+1$. Note that the clique size of $G$ is $2m < k$, provided that $m \ge 2$.

In \cite{Grzesik} it was proved that there are graphs $G$ of an arbitrarily large chromatic number for which $\chi_i(G) = \frac 43 \chi(G)$. The graphs $G$ in (6) of Example \ref{example-1}, with $m=3n+1$, have $\chi(G) =  2m+1$ and $\chi_i(G) = 3m$.
So there are graphs $G$ of an arbitrary large chromatic number for which $\chi_i(G) \approx  \frac 32 \chi(G)$.
In \cite{Grzesik} it was conjectured that $\chi_i(G) \le 2 \chi(G)$ for every graph $G$. This conjecture remains open.

\begin{definition}
    \label{def-twintriple}
    Assume $G$ is a graph and $L$ is a degree-list assignment of $G$. 
    \begin{itemize}
        \item If $u,v$ are two vertices of $G$ with $N_G[u] = N_G[v]$ and $L(u)=L(v)$, then we say $\{u,v\}$ is a pair of adjacent twins of $(G,L)$.
        \item If $[v_1,v_2,v_3]$ is a path in $G$ with $d_G(v_1)=d_G(v_2)=d_G(v_3)=2$ and $L(v_1)=L(v_2)=L(v_3)$, then we say $(v_1, v_2, v_3)$ is a {\em triple of 2-vertices} in $(G,L)$.
        \end{itemize}  
\end{definition}

\begin{definition}
    \label{def-rootedcycle}
    Assume $G$ is a graph,   $C_1, C_2, \ldots, C_k$ are odd cycles of length at least 5 that with a common edge $uv$, $d_G(w) = 2$ for $w \in \cup_{i=1}^k V(C_i) - \{u,v\}$, and $N_G(u) - \cup_{i=1}^k V(C_i) = N_G(v) - \cup_{i=1}^kV(C_i)$. Then we say $u,v$ is a pair of adjacent pseudo-twins of index $k$. 
    We say $u,v$ is a pair of adjacent pseudo-twins if they are adjacent pseudo-twins of index $k$ for some $k \ge 0$.
\end{definition}

For example, the vertices $u, v$ of the graph $(\Theta_{1,2,2,4,4}, L_{\Theta_{1,2,2,4,4}})$ in Fig. \ref{fig-4} are adjacent pseudo-twins of index $2$. Note that   $u,v$ is  a pair of adjacent pseudo-twins on index $0$ means that   $u,v$ is a pair of adjacent twins.

\begin{lem}
    \label{lem-p1}
    If $(G,L) \in \mathcal{W}$ and $G$ is 2-connected, then $G$ contains a pair of adjacent pseudo-twins. 
\end{lem}
\begin{proof}
    Since $G$ is 2-connected and $|V(G)| \ge 3$,  no vertex-sum operation is applied after the last duplication operation. 
Assume the last duplication operation creates a pair of adjacent twins $\{u,v\}$. All the later operations applied in the process of constructing $(G,L)$ (if any) are tripling operations. If $\{u,v\}$ is a pair of adjacent pseudo-twins in $(G,L)$, then after any tripling operation,  $\{u,v\}$ remains  a pair of adjacent pseudo-twins. If $u,v$ have a common neighbour $w$ which has degree $2$, then the operation $T_w(G,L)$ increases by 1 the index of the pseudo-twins.
\end{proof}

\section{Some useful indicated degree-choosable  }

We shall prove in the next section that if $(G,L)$ is infeasible, then $(G,L) \in \mathcal{W}$.
Assume this is not true. Let $(G,L)$ be a counterexample with $|V(G)|$ minimum.
By Lemma \ref{lem-remaining}, for any subset $X$ of $V(G)$, there is an $L$-colouring $\phi$ of $G-X$ so that $(G[X], L^{\phi})$ is infeasible. As $L^{\phi}$ is a degree-list assignment of $G[X]$, $G[X]$ is not indicated degree-choosable. I.e., any induced subgraph of $G$ is not indicated degree-choosable. In this section, we prove  some particular graphs  are indicated degree-choosable, and hence cannot be an induced subgraph of $G$.

%\begin{obs}
%    If a graph $G$ is not indicated degree-choosable, then $G$ contains no induced subgraph which is indicated degree-choosable. In particular, any non indicated degree-choosable graph contains no induced even cycle. 
%\end{obs}

\begin{lem}
     \label{lem-cycle}
     Assume $G $ is a cycle and $L$ is a degree-list assignment of $G$. Then $(G,L)$ is infeasible if and only if $G$ is an odd cycle and all vertices $v$ of $G$ have the same list. In particular, even cycles are indicated degree-choosable.
 \end{lem}
 \begin{proof}
   
     Assume $C=(v_1,v_2,\ldots, v_n)$ is a cycle, and $L$ is a degree-list assignment of $C$ and $(C,L)$ is infeasible. First we prove that 
     $L(v_i) = L(v_j)$ for all $i, j$.
     
     Assume to the contrary that $L(v_1)=\{c_1,c_2\}$ and $  L(v_2) \ne L(v_1)$.

     By Lemma \ref{lem-color}, we may assume that $C-v_1$ is not indicated $L^{v_1 \to c_2}$-colourable. By repeatedly applying Corollary \ref{cor-leaf} and Lemma \ref{lem-color}, we conclude that 
      there is a sequence of colours $c_3,c_4,\ldots, c_n$ 
     such that for $i=2,3,\ldots, n$,  $L(v_i) = \{c_i, c_{i+1}\}$ (hence $c_i \ne c_{i+1}$) and $c_{n+1}=c_1$. 

    Since $L(v_1) \ne L(v_2)$, we may assume that $c_2$ is the unique colour in $L(v_1) \cap L(v_2)$. 
    Then by Lemma \ref{lem-color},   $G-v_2 $ is not $L^{v_2 \to c_2}$-colourable. By Corollary \ref{cor-leaf}, this implies that $c_2 \in L(v_n)$, and hence $L(v_n) = \{c_1, c_2\}$, i.e., $c_n=c_2$. This in turn implies that $c_1 \in L(v_{n-1})$ and hence $c_{n-1}=c_1$ and $L(v_{n-1})=\{c_1,c_2\}$. Repeat this argument, we conclude that $L(v_i) = \{c_1,c_2\}$ for all $i$, a contradiction. 
 
       Assume $L(v_i)=\{c_1,c_2\}$ for all $v_i$. If $C$ is an odd cycle, then $G$ is not $L$-colourable and hence not indicated $L$-colourable. If $C$ is an even cycle, then it follows easily from Corollary \ref{cor-leaf} that $C$ is  indicated $L$-colourable.  
 \end{proof}

\begin{lem}
    \label{lem-2-2}
    Assume $(G,L)$ is infeasible, $uv \in E(G)$ is contained in a cycle $C$ and $d_G(u)=d_G(v)=2$. Then $L(u) = L(v)$.
\end{lem}
\begin{proof} 
    Let $A = V(G)-V(C)$. Then there is an $L$-colouring $\phi$ of $A$ such that $(G-A, L^{\phi})$ is infeasible. Thus $L(u) = L^{\phi}(u) = L^{\phi}(v)=L(v)$ by Lemma \ref{lem-cycle}.
\end{proof}

\begin{lem}
    \label{lem-twins}
    Assume $(G,L)$ is an infeasible pair, and $v,v'$ are adjacent twins, i.e., $N_G[v]=N_G[v']$. Then  
    $(G,L)$ is obtained from an infeasible $(G',L')$ by duplicating $v$, where $G'=G-v'$.
\end{lem}
\begin{proof}
First we show that $L(v)=L(v')$. Otherwise, the restriction of $L$ to $G-\{v,v'\}$ is a non-tight degree-list assignment of $G-\{v,v'\}$. Hence there is an $L$-colouring of $G-\{v,v'\}$. For any  $L$-colouring $\phi$ of $G-\{v,v'\}$, $L^{\phi}$ is a degree-list assignment of $G[\{v,v'\}]$, which is either non-tight or $L^{\phi}(v) \ne L^{\phi}(v')$. In any case, $G[\{v,v'\}]$ is indicated $L^{\phi}$-colourable, a contradiction.

    Since $G$ is not indicated $L$-colourable, there is a colour $c \in L(v')$ such that $G-v'$ is not indicated $L^{v' \to c}$-colourable.
    Since $L$ is a degree-list assignment of $G$, by Lemma \ref{lem-color}, we know that $c \in L(x)$ for all $x \in N_G[v']$. 
     Therefore $(G,L)= D_{v,c}(G-v', L^{v' \to c})$.
\end{proof}

\begin{cor}
\label{cor-clique}
If $(G,L)$ is infeasible, and $K$ is a clique in $G$ and $N_G[u] = N_G[v]$ for all $u, v \in V(K)$.  Then there is a set $A$ of $|V(K)|-1$ colours such that $A \subseteq L(x)$ for all $x \in N_G(v)$ for $v \in V(K)$.
\end{cor}

\begin{lem}
    \label{lem-triangle}
    Assume that $(\Theta_{1,2k_1, 2k_2},L)$ is infeasible, then $L$ is isomorphic to $L_{\Theta_{1,2k_1,  2k_2}}$ defined in (4) Example \ref{example-1}. Consequently, $(\Theta_{1,2k_1,  2k_2},L) \in \mathcal{W}$.
\end{lem}
\begin{proof}
Assume $\Theta_{1,2k_1,2k_2}$ consists of three paths $P_1,P_2,P_3$ connecting two vertices $u$ and $v$, where for $i=1,2$, $P_i$ has length $2k_i$, and $P_3$ is a single edge $uv$. 

For $i=1,2$, let  
 $C_i$ be the odd cycle induced by the $i$th path and the edge $uv$. 
 By Lemma \ref{lem-remaining}, there exists an $L$-colouring $\phi_i$ of $G-C_i$ such that $(C_i,L^{\phi_i})$ is infeasible. By Lemma \ref{lem-cycle}, $L^{\phi_i}(x) = A_i$ for some 2-colour set $A_i$ for all $x \in V(C_i)$. Note that $A_1 \ne A_2$, for otherwise, let $x$ be the neighbour of $v$ in $P_2$, then $\phi_1(x) \in A_1$ and hence $L^{\phi_1}(v) \ne A_1$, a contradiction.

    As $A_1 \cup A_2 \subseteq L(u)$  and $|L(u)|=3$, we conclude that   $|A_1 \cap A_2|=1$, and $L(u)=L(v) = A_1 \cup A_2$.   Hence $L$ is isomorphic to $L_{\Theta_{1, 2k_1, 2k_2}}$.
\end{proof}

\begin{definition}
    \label{def-thetaplus}
    Assume $k_1,k_2,k_3 \ge 2$, and the theta graph $\theta_{k_1,k_2,k_3}$ consists of paths $P_i=[u,v_{i,1}, v_{i,2}, \ldots v_{i, k_i-1}, w]$ connecting   $u$ and $w$, for $i=1,2,3$.  
A graph $H$ obtained from $\theta_{k_1,k_2,k_3}$ by adding some edges joining vertices of $V(P_1) \backslash  \{u,w, v_{1,1}\}$ and $P_3 \backslash  \{u,w\}$ is called a {\em theta-plus} graph. The two vertices $u,w$ are the {\em end vertices} of the theta-plus graph.
\end{definition}

\begin{figure}[ht]
\centering
\begin{tikzpicture}[scale=1.2, every node/.style={scale=0.8}, baseline]

% ------------------------------------------
% 第一部分: theta_{3,2,3} 图 (theta-plus)
\begin{scope}[shift={(0, 0)}]
    \coordinate (U) at (0.8, 1) {};
    \coordinate (W) at (0.8, -1) {};
    \coordinate (V11) at (0, 0.5) {};
    \coordinate (V12) at (0, -0.5) {};
    \coordinate (V21) at (0.8, 0) {};
    \coordinate (V31) at (1.6, 0.5) {};
    \coordinate (V32) at (1.6, -0.5) {};

    \draw (U) -- (V11) -- (V12) -- (W);
    \draw (U) -- (V31) -- (V32) -- (W);
    \draw (U) -- (V21) -- (W); 
  \draw (V12) .. controls (0.3,0.2) .. (V31);
    \draw (V12) --(V32); 

\foreach \point in {U, W, V11, V12, V21, V31, V32} {
    \filldraw[fill=white, draw=black](\point) circle (1.5pt);
}

\node[above] at (U){\bl{$u$}};
\node[below] at (W){\bl{$w$}};
\node[left] at (V11){\bl{$v_{11}$}};
\node[left] at (V12){\bl{$v_{12}$}};
\node[right] at (V21){\bl{$v_{21}$}};
\node[right] at (V31){\bl{$v_{31}$}};
\node[right] at (V32){\bl{$v_{32}$}};

\node[below=0.7cm] at (W){A theta-plus graph};

\end{scope}

%\hspace{3cm}

% ------------------------------------------
% 第二部分: 两个7点圆图 (两条弦)
\begin{scope}[shift={(3.7, 0)}]
    % 创建7个点
    \foreach \i in {1,...,7} {
        \coordinate (P\i) at ({360/7 * (\i - 1)}:1) {};
    }
    
    % 绘制两条弦
    \draw (P1) -- (P3);
    \draw (P4) -- (P7);
    \draw (P1)--(P2)--(P3)--(P4)--(P5)--(P6)--(P7)--(P1);
    % 绘制圆圈，将每个点用 \filldraw 处理
    \foreach \i in {1,...,7} {
        \filldraw[fill=white, draw=black] (P\i) circle (1.5pt);  % 注意这里是 (P\i)，而非整数
    }
    \node[right=0.2cm,below=2cm] at (P1) {Two double chorded cycles};

\end{scope}

\begin{scope}[shift={(5.8, 0)}]
    % 创建7个点
    \foreach \i in {1,...,7} {
        \coordinate (P\i) at ({360/7 * (\i - 1)}:1) {};
    }
    
    % 绘制两条弦
    \draw (P3) -- (P7);
    \draw (P3) -- (P6);
        \draw (P1)--(P2)--(P3)--(P4)--(P5)--(P6)--(P7)--(P1);
    % 绘制圆圈，将每个点用 \filldraw 处理
    \foreach \i in {1,...,7} {
        \filldraw[fill=white, draw=black] (P\i) circle (1.5pt);  % 注意这里是 (P\i)，而非整数
    }
    
\end{scope}

% ------------------------------------------
% 第三部分: 带圆心的7点图
% 图1：中心点与连续4个顶点相连
\begin{scope}[shift={(8.6, 0)}]
    \coordinate (C) at (0, 0) {};
    \foreach \i in {1,...,7} {
        \coordinate (R\i) at ({360/7 * (\i - 1)}:1) {};
    }
    \draw (C) -- (R1);
    \draw (C) -- (R2);
    \draw (C) -- (R3);
    \draw (C) -- (R4);
        \draw (R1)--(R2)--(R3)--(R4)--(R5)--(R6)--(R7)--(R1);
     % 绘制圆圈，将每个点用 \filldraw 处理
    \foreach \i in {1,...,7} {
        \filldraw[fill=white, draw=black] (R\i) circle (1.5pt);  % 注意这里是 (P\i)，而非整数
    }
      \filldraw[fill=white, draw=black] (C) circle (1.5pt);  
 \node[right=0.1cm,below=2cm] at (R1) {Near odd-wheels};
\end{scope}

% 图2：中心点与间隔顶点相连
\begin{scope}[shift={(10.7, 0)}]
    \coordinate (C2) at (0, 0) {};
    \foreach \i in {1,...,7} {
        \coordinate (S\i) at ({360/7 * (\i - 1)}:1) {};
    }
    \draw (C2) -- (S1);
    \draw (C2) -- (S2);
    \draw (C2) -- (S4);
\draw (S1)--(S2)--(S3)--(S4)--(S5)--(S6)--(S7)--(S1);
     % 绘制圆圈，将每个点用 \filldraw 处理
    \foreach \i in {1,...,7} {
        \filldraw[fill=white, draw=black] (S\i) circle (1.5pt);  % 注意这里是 (P\i)，而非整数
    }

      \filldraw[fill=white, draw=black] (C2) circle (1.5pt);  
\end{scope}

\end{tikzpicture}
\caption{Theta-plus graphs, double chorded cycles and near odd-wheels}
\label{fig-7}
\end{figure}
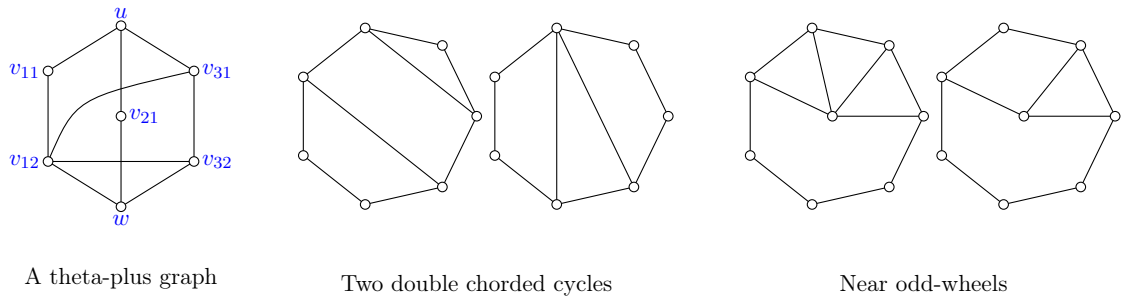

\begin{lem}
\label{lem-thetaplus}
 Every theta-plus graph is indicated degree-choosable.
\end{lem}
\begin{proof}
Assume $H$ is a theta-plus graph, and let $P_i$ ($i=1,2,3$) be the paths defined in Definition \ref{def-thetaplus}.
 It follows from the definition that $d_H(v) = 2$ for $ v \in V(P_2) \cup \{v_{1,1}\} - \{u,w\}$.  Assume to the contrary that $L$ is a degree-list assignment of $H$ and $(H,L)$ is infeasible.  
   
Let $C_1 = P_1 \cup P_2$ and $C_2 = P_2 \cup P_3$.  Note that $C_1,C_2$ are induced odd cycles in $H$.   By Lemma \ref{lem-remaining}, there exists an $L$-colouring $\phi$ of $G-V(C_1)$ such that $(C_1,L^\phi)$ is infeasible. As $d_G(v_{1,1})=d_G(v_{2,1})=2$, $L(v_{1,1})=L^\phi(v_{1,1})=L^\phi(u)=L^\phi(v_{2,1})=L(v_{2,1})$. Similarly, there exists an $L$-colouring $\psi$ of $G-V(C_2)$ such that $(C_2,L^{\psi})$ is infeasible. Then  $  L^\psi(u)=L^\psi(v_{2,1})=L(v_{2,1})$. But $\psi(v_{1,1}) \in L(v_{1,1}) = L(v_{2,1})$. Hence    $\psi(v_{1,1}) \in L^\psi(u)$, a contradiction.
\end{proof}

\begin{definition}
    \label{def-twochordcycle}
    Assume   $H$ is a graph that is obtained from a cycle $C=[v_1v_2\dots v_l]$ by adding   two chords $e_1=v_iv_j$ and $e_2=v_sv_t$, where $i  < j \le s < t$. Then we call $H$ is a {\em   double chorded cycle}.
\end{definition}

\begin{lem}
\label{lem-double chorded cycle}
   Any double chorded cycle is indicated degree-choosable.
\end{lem}

\begin{proof}
    Assume $H$ is a double chorded cycle 
    consisting of  a cycle $C=[v_1v_2\dots v_l]$ with   two chords $e_1=v_1v_r$ and $e_2=v_sv_t$, where $3 \le r \le s \le t-2 \le l-2$.
    
    Assume to the contrary that  $(H,L)$ is an infeasible pair. 

For   $i,j \in \{1,2,\ldots, l\}$, denote by $P^{i,j}$ the path $[v_i,v_{i+1},\dots, v_j]$ in $C$ (where indices are modulo $l$). There are three induced cycles $C_1,C_2,C_3$ in $H$, induced by   $V(P^{1,r})$, $V(P^{r,s})\cup V(P^{t,1})$ and $ V(P^{s,t})$, respectively.

    By Lemma \ref{lem-remaining}, there exists an $L$-colouring $\phi$ of $P^{s+1,t-1}$ such that $(G-P^{s+1,t-1},L^\phi)$ is  infeasible. Note that $G-P^{s+1,t-1}$ is a theta graph. 
    By Lemma \ref{lem-triangle},  we may assume that 
    $L(v)=L^\phi(v)=\{a,b\}$ for each $v\in V(C_1)\backslash\{v_1,v_r\}$,
    $L(v)=L^\phi(v)=\{a,c\}$ for each $v\in V(C_2)\backslash\{v_1,v_r, v_s, v_t\}$, $L(v_1) = L^{\phi}(v_1) = \{a,b,c\}$, $L^\phi(v_r) = \{a,b,c\} \subseteq L(v_r)$,  and $L^\phi(v_t)=L^\phi(v_s) = \{a,c\} \subseteq L(v_t), L(v_s)$.
    
Similarly, there exists an $L$-colouring $\psi$ of $P^{2,r-1}$ such that $(G-P^{2,r-1},L^\psi)$ is  infeasible.  
This implies that there is a set $A$ of two colours, such that $|A \cap \{a,c\}| =1$,
$L(v)=L^\psi(v)=\{a,c\}$ for each $v\in V(C_2)\backslash\{v_1,v_r,v_s,v_t\}$,
    $L(v)=L^\psi(v)=A$ for each $v\in V(C_3)\backslash\{v_s,v_t\}$.

As $|A \cap \{a,c \}| =1$, we may assume  that $A=\{a,d\}$ or $A=\{c,d\}$, where $d$ is either a new colour, or $d=b$. First we show that $A   = \{a,d\}$. Assume to the contrary that $A=\{c,d\}$.
Since $H$ contains no even cycle, $P^{t,1}$ and $P^{r,s}$ have different parities. Assume $P^{t,1}$ has even length. By Lemma \ref{lem-remaining}, there is an $L$-colouring $\rho$ of $P^{t+1,l}$ such that $(H-P^{t+1,l}, L^{\rho})$ is infeasible. As $P^{t,1}$ has even length,  $\rho(v_l) = \rho(v_{t+1}) \in \{a,c\}$. Then either  $L^{\rho}(v_1) \ne \{a,b\}$ or $L^{\rho}(v_t) \ne \{c,d\}$. It follows from Lemma \ref{lem-cycle} that $(H-P^{t+1,l}, L^{\rho})$ is not infeasible, a contradiction. 
    
If $r=s$, then   $L(v_1)=\{a,b,c\}$ and $L(v_r)=\{a,b,c,d\}$ ($a,b,c,d  $ are distinct colours).  
    By Lemma \ref{lem-remaining}, there is an $L$-colouring $\tau$ of $P^{t,l}$ such that $(G-V(P^{t, l}), L^{\tau})$ is infeasible. By Lemma \ref{lem-color},    $\phi(v_t) \in L(v_{t+1}) \cap L(v_{t-1})$. Hence $\tau(v_t) =a$. Since $C_2$ is an odd cycle,  this implies that $\tau(v_l)=a$. But then $L^{\tau}(v_1)=L(v_1)\backslash \{a\}\ne L^{\tau}(v)$ for   $v\in C_1 \backslash \{v_1, v_r\}$, contrary to Lemma \ref{lem-cycle}.
       
Assume $r<s$. Then $d(v_r)=d(v_s)=3$ and $L(v_r)=L(v_1)= \{a,b,c\}$ and $L(v_s) =L(v_t) = \{a,c,d\}$.
    
    By Lemma \ref{lem-remaining}, there is an $L$-colouring $\phi_1$ of $P^{r+1,s}$ such that $G-V(P^{r+1,s})$ is not indicated $L^{\phi_1}$-colourable, and an $L$-colouring ${\phi_2}$ of $P^{t, l}$ such that $G-V(P^{t,l})$ is not indicated $L^{\phi_2}$-colourable. By Lemma \ref{lem-2-2}, \begin{equation}
    \label{eq-1}
        L^{\phi_2} (v_1)=\{a,b\}=L^{\phi_1} (v_r).
    \end{equation}   
    
   This implies that $\phi_1(v_{r+1}) = \phi_2(v_l) = c$.

We claim that ${\phi_1}(v_s)=\phi_2(v_t)=a$.   By Lemma \ref{lem-color}, $\phi_1(v_s)\in L(v_{s+1})=\{a,d\}$. If $r=s-1$, $L^{\phi_1}(v_r)\ne \{a,b\}$, a  contradiction.   If $r<s-1$,   then by Lemma \ref{lem-color}, ${\phi_1}(v_s)\in L(v_{s-1}) = \{a,c\}$. So   $\phi_1(v_s) \in \{a,c\} \cap \{a,b\}$, and hence $\phi_1(v_s)=a$. 
 By symmetry, $\phi_2(v_t)=a$. 

    Since  paths $P^{r,s}$ and $P^{t,1}$ have different parities, and all vertices in $P^{r+1,s-1} \cup P^{t+1,l}$ have the same list $\{a,c\}$, this implies that  $L^{\phi_2} (v_1) \ne \ L^{\phi_1} (v_r)$, in contrary to (\ref{eq-1}).
\end{proof}

Note that in a double chorded cycle, the two chords need to be "non-crossing". For example, if $G$ is obtained from an odd cycle by duplicating one vertex, then it is a cycle with two crossing chords, i.e., two chords $e_1=v_{i_1}v_{j_1}$ and $e_2=v_{i_2}v_{j_2}$ such that $i_1 <i_2 < j_1 < j_2$. But $G$ is not indicated degree-choosable.

\begin{definition}
    \label{def-nearoddwheel}
    Assume $H$ is obtained from an odd cycle $C$ of length at least $5$ by adding a new vertex $x$ adjacent to $t$ vertices on $C$. 
    If   $t\ge 4$ or $t \in \{2,3\}$ and the neighbours of $x$ are not consecutive vertices of $C$, then we call $H$ a near odd-wheel.
\end{definition}

\begin{lem}\label{lem-odd-wheel}
   Every near odd-wheel is indicated degree-choosable.
\end{lem}
\begin{proof}
    Assume $(H,L)$ is infeasible, and $H$ is a near odd-wheel, consisting an odd cycle $C$ and a vertex $x$ adjacent to $t$ vertices $v_1, v_2, \ldots, v_t$ of $C$ and the $t$ vertices occurs in $C$ in this clockwise cyclic order. 
    Let $P_i=C[v_i, v_{i+1}]$ be paths in $C$ from $v_i$ to $v_{i+1}$ in clockwise direction. If $t \ge 4$, then it is easy to check that  the subgraph induced by $\cup_{1\le i\le 3}V(P_i)\cup \{x\}$ is a double chorded cycle, and hence $H$ is indicated degree-choosable.  

    If $t = 2$ or $ 3$, and  $v_i, v_{i+1}$ are not consecutive vertices of $C$, then  $G$ is a theta-plus graph with end vertices $v_i$ and $v_{i+1}$, and hence $H$ is indicated degree-choosable.  
\end{proof}

\begin{cor}
    \label{cor-forbidden}
    If $(G,L)$ is infeasible, then $G$ does not contain any of the following graphs as an induced subgraph (see Fig. \ref{fig-7} for examples of such graphs):
    \begin{enumerate}
        \item An even cycle.
        \item A theta-plus graph.
        \item A double chorded cycle.
        \item A near odd-wheel.
    \end{enumerate}
\end{cor}

\begin{lem}
    \label{lem-c+}
    If $(G,L)$ is infeasible and $G$ is obtained from an odd cycle $C=[v_1v_2\ldots v_{2k+1}]$ by adding a vertex $v$ adjacent to two consecutive or three consecutive vertices of $C$, then $(G,L) \in \mathcal{W}$. 
\end{lem}
\begin{proof}
    Assume $G$ is obtained from $C$ by adding a vertex $v$ adjacent to $v_1,v_2,v_3$, and $(G,L)$ is infeasible. By Lemma \ref{lem-remaining}, there is a colour $c \in L(v)$ such that $(G-v, L^{v \to c})$ is infeasible. By Lemma \ref{lem-cycle}, we may assume that $L^{v \to c}(v_i) = \{a,b\}$ for all $v_i$. By Lemma \ref{lem-color}, $L(v_1)=L(v_2)=L(v_3) = \{a,b,c\}$. By symmetry, we also have $L(v)=\{a,b,c\}$. Hence $(G,L) \in \mathcal{W}$. The case $v$ is adjacent to two consecutive vertices of $C$ is proved similarly.
\end{proof}

\begin{cor}
    \label{cor-c+}
    If $(G,L)$ is infeasible and $G-v$ is an odd cycle for some vertex $v$ with $d_G(v) \ge 2$, then $(G,L) \in \mathcal{W}$.
\end{cor}

\section{All infeasible pairs are constructible}

This section proves the following theorem, which shows that all infeasible pairs are constructible.

\begin{theorem}
    \label{thm-nic}
    A pair $(G,L)$ is infeasible if and only if $(G,L) \in \mathcal{W}$. 
\end{theorem}
\begin{proof}
    We have already shown that if $(G,L) \in \mathcal{W}$, then $(G,L)$ is infeasible. 
    It remains to show that the converse is true.
    
    Assume the converse is not true, and $(G,L)$ is a counterexample with $|V(G)|$ minimum. It follows from Lemma \ref{lem-2-2}, Lemma \ref{lem-twins} that $G$ has no adjacent twins, no triple of 2-vertices. Hence $G$ has no adjacent pseudo-twins. It is also obvious that $G \ne K_2$. Hence $|V(G)| \ge 3$.

 We shall derive a sequence of properties of $G$, that lead to a final contradiction.
    
 \begin{lem}
     \label{lem-no-cut}
     $G$ is 2-connected.
 \end{lem}
    \begin{proof} 
        Assume to the contrary that $G$ has a cut-vertex $v$. Let $G_1,G_2$ be  two connected induced subgraphs   of $G$ such that $V(G_1) \cap V(G_2) = \{v\}$ and $V(G_1) \cup V(G_2) = V(G)$. We choose $v$ so that $G_1$ has no cut-vertex.  %We rename the vertex $v$ in $G_i$ by $v_i$ for $i=1,2$. 

        By Lemma  \ref{lem-remaining}, for $i=1,2$, there is  an $L$-colouring $\phi_i$ of $G-G_i$ such that   $(G_i, L^{\phi_i})$ is infeasible. Let $L_i=L^{\phi_i}$.
        
         By the minimality of $(G,L)$, we know that  $(G_i, L_i) \in \mathcal{W}$. Since $G_1$ has no cut vertex, we conclude that either $G_1$ has a pair of adjacent twins $u_1u_2$ with $L_1(u_1)= L_1(u_2)$, or $G_1$ has an induced path $[u_1,u_2,u_3]$ with  $d_{G_1}(u_i)=2$ and   $L_1(u_i) = A$ for a set $A$ of 2 colours.

        Assume first that $u_1u_2$ is a pair of adjacent twins in $G_1$ with $L_1(u_1)= L_1(u_2)$. Since $G$ has no  adjacent twins,  $v \in \{u_1,u_2\}$, say $v=u_2$.

        By Lemma \ref{lem-remaining}, there is a colour $c \in L(u_1)$ 
        such that $(G-u_1, L^{u_1 \to c})$ is infeasible. Now $v$ is a cut-vertex of $G-u_1$  or $G_1$ contains exactly one edge $u_1u_2$. Hence $(G-u_1, L^{u_1 \to c}) \in \mathcal{W}$.
        Thus there are list assignments $L'_1$ of $G_1-u_1$ and $L'_2$ of $G_2$ such that $(G_1-u_1,L'_1)$ and $(G_2, L'_2)$ are infeasible (and hence in $\mathcal{W}$) and 
        $(G-u_1, L^{u_1 \to c})=(G_1-u_1, L'_1) \oplus_{u_2, v} (G_2, L'_2)$. Note that if $G_1$ contains exactly one edge $u_1u_2$, then $(G_1-u_1, L'_1)=(\{u_2\},L_\emptyset)$ and $L'_2=L_2$, that is, we consider $(G-u_1, L^{u_1 \to c})$ as the vertex sum of $(\{u_2\},L_\emptyset)$ and $(G_2, L_2)$. Then 
         $D_{u_2,c}(G_1-u_1, L'_1) \in \mathcal{W}$ and it is easy to check that $(G,L) = D_{u_2,c}(G_1-u_1, L'_1) \oplus_{u_2, v} (G_2, L'_2)$. Hence $(G,L) \in \mathcal{W}$, a contradiction. 

         Next we assume that $G_1$ has an induced path $[u_1,u_2,u_3]$ with $d_{G_1}(u_i)=2$ and  $L_1(u_i) = A$ for a set $A$ of 2 colours. Since $G$ has no such path, we know that $v \in \{u_1, u_2, u_3\}$. Let $C$ be an induced cycle in $G_1$ containing $u_1u_2u_3$. 
         By Lemma \ref{lem-remaining}, $C$ is an odd cycle and there is an $L$-colouring $\theta$ of $X=V(G_1)-V(C)$ such that 
         $(G-X, L^{\theta}) $ is infeasible. By Lemma \ref{lem-cycle}, all vertices  $u \in V(C)-\{v\}$ have $L^{\theta}(u)=A$.  By Lemma \ref{lem-remaining} again, there is an $L^{\theta}$-colouring $\psi$ of $V(C)-\{v\}$ such that $(G-X-(V(C)-\{v\}), L^{\psi})$ is infeasible. 
         Note that $L^{\psi}(v) =L(v) -A$, and $L^{\psi} = L_2$.
         Therefore $(G,L) = (G_1,L_1) \oplus_{u_i, v} (G_2, L_2)$.
         So $(G,L) \in \mathcal{W}$, a contradiction. This completes the proof of Lemma \ref{lem-no-cut}.
    \end{proof}
 
           %\begin{lem}
 %   \label{lem-2con}
 %       For any vertex $x$ in $G$, if there exist vertices $y$ and $z$ such that $N_G(x)\subseteq N_G(y)$ and $N_G(x)\subseteq N_G(z)$, then $G-x$ is $2$-connected.
 %   \end{lem}

 %   \begin{proof}
 %       Assume that $G-x$ is not $2$-connected and let $B_1, B_2,\dots, B_k$ be the leaf-blocks of $G-x$. Since $G$ is $2$-connected, $N_G(x)\cap B_i\ne\emptyset$ for every $i\in[k]$. Thus $N_G(y)\cap B_i\ne\emptyset$, i. e. $N_{G-x}(y)\cap B_i\ne\emptyset$ for every $i\in[k]$. It implies that $y$ is the only cut-vertex in $G-x$.  For the same argument, $N_{G-x}(z)\cap B_i\ne\emptyset$ for every $i\in[k]$ and $z$ is the only cut-vertex in $G-x$, a contradiction.
 %   \end{proof}

    %Assume $G_1$ and $G_2$ are induced subgraphs of $G$ such that  $V(G_1) \cup V(G_2)=V(G)$ and $V(G_1) \cap V(G_2) = \{v\}$.

%We may assume that $G_1$ has no cut-vertex. Since $G$ has no adjacent twins, if $G_1$ has a pair of adjacent twins, then $v$ is one of the twins. the adjacent twins is $\{u,v\}$ 

%Let $u$ be a non-cut vertex in $G_1-v$, and let $c \in L(u)$ be a colour such that $(G-u, L^{u \to c})$ is infeasible. If $G_1$ is a copy of $K_2$ with vertices $u,v$, then $(G,L)=(G_1, L_1) \oplus (G_2,L_2)$, where $L_1(u)=L_1(v)=\{c\}$ and $L_2 = L^{u \to c}$. Hence $(G,L) \in \mathcal{W}$, a contradiction. By induction hypothesis, $(G-u, L^{u \to c}) \in \mathcal{W}$, and hence ...

\begin{lem}
    \label{lem-2conn-del}
    Assume $G$ is 2-connected and $X$ is a subset of vertices that induces a connected subgraph. If any two vertices in $N_G(X)-X$ are contained in a cycle in $G-X$, then $G-X$ is 2-connected.
\end{lem}
\begin{proof}
    Assume to the contrary that any two vertices in $N_G(X)-X$ are contained in a cycle in $G-X$ and $G-X$ is not 2-connected. Let $v$ be a cut-vertex that separates $G-X$ into 
    two subgraphs $G_1$ and $G_2$ (i.e., $V(G_1) \cap V(G_2) = \{v\}$).
    Since any two vertices in $N_G(X)-X$ are contained in a cycle in $G-X$, we conclude that $N_G(X)-X$ are contained in $V(G_i)$ for some $i=1,2$. Then $v$ is also a cut-vertex of $G$, a contradiction.
\end{proof}

The converse of Lemma \ref{lem-2conn-del} is obviously true: If $G-X$ is 2-connected, then any two vertices in $N_G(X)-X$ are contained in a cycle in $G-X$. But we shall not need this.

An {\em open ear  } in a graph $G$ is  a path connecting two distinct vertices of degree at least 3 and whose interior vertices (if any) are degree 2 vertices.  An open ear is {\em proper} if it contains at least one interior vertex.  It was proved by Whitney \cite{whitney} that if $G$ is 2-connected and $G$ is not a cycle, then $G$ has an open ear whose deletion results in a 2-connected graph. 
 Given an ear $P$, we denote by $I(P)$ the set of interior vertices of $P$ and $Z(P)$ the set of the two end vertices of $P$, and let $V(P)=Z(P) \cup I(P)$. 

\begin{lem}
    \label{lem-ear}
    If $G$ is 2-connected and $G$ is not a cycle, then $G$ has a vertex $v$ such that $G-v$ is 2-connected  or a proper ear $P$ such that  $G-I(P)$ is 2-connected.
\end{lem}
\begin{proof}
   The proof is by  induction on the number of edges in $G$. If $G$ is 2-connected and $|E(G)| \le 4$ then $G$ is a cycle, there is nothing to prove.
    
    Assume $G$ is 2-connected with $|E(G)| \ge 5$, $G$ is not a cycle and the lemma holds for graphs with less edges. By the ear-decomposition theorem, $G$ has either a proper ear $P$  such that   $G-I(P)$ is 2-connected, or $G$ has an edge $e=uv$ such that $G-e$ is 2-connected. In the former case, we are done. In the later case, let $G'=G-e$. If $G'$ is a cycle, then the end vertices of $e$ are connected by two proper ears $P_1,P_2$, and  $G-I(P_1)$ is a cycle, and we are done. Assume $G'$ is not a cycle. 
    By induction hypothesis, $G'$ has a vertex $x$ such that $G'-x$ is 2-connected or a proper ear $P$  such that $G'-I(P)$ is 2-connected. In the former case, $G-x$ is 2-connected and we are done. In the later case, we may assume that $|I(P)| \ge 2$.  If $u,v \notin I(P)$, then $P$ is a proper ear in $G$ and we are done. 
    
    By symmetry, assume $v \in I(P)$.
    If $u \in V(P)$, then the part of $P$ connecting $v$ and $u$ is a proper ear $P'$ of $G$.    Let $C$ be the cycle in $G$  induced by $V(P')$. Let $C'$ be any cycle containing $v$ and a vertex $w \notin V(C)$. Then  either $C'$ or $C \Delta C'$ is a cycle in $G-I(P')$ containing $Z(P')$, where $C \Delta C'$ denotes the symmetric difference of $C$ and $C'$. 
    By Lemma \ref{lem-2conn-del}, $G-I(P')$ is 2-connected and we are done. 

    Assume $u \notin V(P)$.  
    Then $G$ has  a proper ear $P'$  contained in $P$ with $v \in Z(P')$. Assume $Z(P)=
    \{v',v''\}$ and $Z(P')=\{v', v\}$. If $G-I(P')$ is 2-connected, then we are done. 
    Assume that $G-I(P')$ has a cut-vertex $w$.  Since $G'-I(P)$ is 2-connected, $G-I(P) -w $ is a connected subgraph of $G-I(P')-w$ containing $v'$. Let $Q$ be the connected component of $G-I(P')-w$ containing $G-I(P) -w $. 
    If $w \in I(P)$, then $Q$ contains both $v$ and $v''$. Hence $Q$ contains   all vertices of $G-I(P')-w$, a contradiction. Assume $w \notin I(P)$.
    If $u \ne w$, then $u$ is also contained in $Q$ and hence $Q$ contains edge $e=uv$  and therefore all vertices of $G-I(P')-w$, a contradiction. So $u=w \notin I(P)$.
    As $u \notin V(P)$, we conclude that $w\ne v''$. Thus $v''$ is contained in $Q$ and this implies that  $V(P)-I(P') \subseteq Q$ and hence $Q$ contains all vertices of $G-I(P')-w$, a contradiction. 
\end{proof}

For the remainder of the proof of Theorem \ref{thm-nic}, we consider two cases.

\bigskip
\noindent
{\bf Case 1} $G$ has a proper ear $P$   such that $G-I(P)$ is 2-connected.

 Since $G$ has no triple of 2-vertices, $I(P)$ has either one vertex $w_1$ or two adjacent vertices $w_1, w_2$.
 By the minimality of $(G,L)$ and Lemma \ref{lem-p1},  $G'=G-I(P)$   contains a pair of adjacent pseudo-twins $\{u, v\}$ of index $k$ for some $k \ge 0$. It is easy to verify that $G'$ is not a cycle, and hence $d_{G'}(u) = d_{G'}(v) \ge 3$. 

Assume first that $k \ge 1$ and $C=[v,u,v_1,v_2, \ldots, v_{2m+1}]$ has length at least 5, and $d_{G'}(v_i) =2$ for $i=1,2,\ldots, 2m+1$. As $d_{G'}(u) = d_{G'}(v) \ge 3$, edge $uv$ is contained in at least one other odd cycle $C'$ (which might be a triangle).
Since $G$ contains no triple of 2-vertices, some $v_i$ is adjacent to $I(P)$. Note that $|N_G(I(P)) - I(P)|=2 $. Assume $|N_G(I(P)) \cap V(C)| =2 $. If the two
vertices in $N_G(I(P)) \cap V(C)$ are two consecutive vertices of $C$, then the subgraph of $G$ induced by $I(P) \cup V(C) \cup V(C')$ is a double chorded cycle, see Fig. \ref{fig:adjacent pseudo-twins of $G-I(P)$}(a).  If the two
vertices $x,y$ in $N_G(I(P)) \cap V(C)$ are non-consecutive vertices of $C$, then 
the subgraph of $G$ induced by $I(P) \cup V(C)$ is a theta-plus graph with end vertices $x$ and $y$, see Fig. \ref{fig:adjacent pseudo-twins of $G-I(P)$}(b).

 %==========================================  
 \begin{figure}
\centering
\begin{tikzpicture}[scale=1.5, every node/.style={scale=0.8}, baseline] 
    % 定义正五边形的边长
  \def\side{0.4}
     \begin{scope}[xshift=0cm] 
    % 计算正五边形的五个顶点
    \foreach \i in {0,1,2,3,4} {
        \coordinate (P\i) at ({\side * cos(90 + \i*72)}, {\side * sin(90 + \i*72)});
       
    }
    
 \coordinate (W1) at ($(P0)!0.5!(P1) + (-0.3,0.35)$);
  \coordinate (C1) at ($(P2)!0.5!(P3) + (0,-0.4)$);
 \draw (P0) -- (P1) -- (P2) -- (P3) -- (P4) -- cycle;
\draw[red] (P0)--(W1)--(P1);
\draw[blue] (P2)--(C1)--(P3);

 \foreach \i in {0,1,2,3,4}{
 \ifnum\i=2
     \filldraw[fill=white,draw=black] (P\i) circle (1pt);  
            \else
             \ifnum\i=3
    \filldraw[fill=white,draw=black] (P\i) circle (1pt);  
             \else
      \filldraw[fill=black,draw=black] (P\i) circle (1pt);    
             \fi 
             \fi
  % \node at ($(P\i) + (0.3, 0.3)$) {\scriptsize $v_{\i}$};

 }
  \filldraw[fill=black,draw=black] (W1) circle (1pt); 
\filldraw[fill=white,draw=black] (C1) circle (1pt);

\node[above] at (W1){\scriptsize $w_1$};
\node[left] at (P1){\scriptsize $v_1$};
\node[above] at (P0){\scriptsize $v_2$};
\node[right] at (P4){\scriptsize $v_3$};
\node[left] at (P2){\scriptsize $u$};
\node[above=0.28] at (C1){\scriptsize \bl{$C'$}};
\node[right] at (P3){\scriptsize$v$};
\node[below=0.4] at (P0){\scriptsize$C$};
\node[right=0.28] at (W1){\scriptsize \red{$P$}};
 % 在 scope 的右上方添加名字
        \node[below=2] at (P0) {\scriptsize (a) Double chorded cycle}; 
     \end{scope}

       \begin{scope}[xshift=2.6cm, yshift=-0.2cm] 
    % 计算正五边形的五个顶点
    \foreach \i in {0,1,2,3,4} {
        \coordinate (P\i) at ({\side * cos(90 + \i*72)}, {\side * sin(90 + \i*72)});
       
    }
    
 \coordinate (W1) at ($(P0)!0.5!(P1) + (-0.2,0.35)$);
  \coordinate (W2) at ($(P0)!0.5!(P4) + (0.2,0.35)$);
  \coordinate (C1) at ($(P2)!0.5!(P3) + (0,-0.4)$);
 \draw (P0) -- (P1) -- (P2) -- (P3) -- (P4) -- cycle;
\draw[red] (P1)--(W1)--(W2)--(P4);
\draw[white] (P2)--(C1)--(P3);

 \foreach \i in {0,1,2,3,4}{
 \ifnum\i=2
     \filldraw[fill=white,draw=black] (P\i) circle (1pt);  
            \else
             \ifnum\i=3
    \filldraw[fill=white,draw=black] (P\i) circle (1pt);  
             \else
      \filldraw[fill=black,draw=black] (P\i) circle (1pt);    
             \fi 
             \fi
  % \node at ($(P\i) + (0.3, 0.3)$) {\scriptsize $v_{\i}$};

 }
  \filldraw[fill=black,draw=black] (W1) circle (1pt); 
   \filldraw[fill=black,draw=black] (W2) circle (1pt); 
%\filldraw[fill=white,draw=black] (C1) circle (1pt);  

\node[above] at (W1){\scriptsize $w_1$};
\node[above] at (W2){\scriptsize $w_2$};
\node[left] at (P1){\scriptsize $v_1$};
\node[above] at (P0){\scriptsize $v_2$};
\node[right] at (P4){\scriptsize $v_3$};
\node[left] at (P2){\scriptsize $u$};
\node[right] at (P3){\scriptsize$v$};
\node[below=0.4] at (P0){\scriptsize$C$};
\node[above=0.28] at (P0){\scriptsize \red{$P$}};

\node[below=1.7] at (P0){\scriptsize (b) Theta-plus graph};
     \end{scope}

          \begin{scope}[xshift=5.3cm, yshift=-0.2cm] 
    % 计算正五边形的五个顶点
    \foreach \i in {0,1,2,3,4} {
        \coordinate (P\i) at ({\side * cos(90 + \i*72)}, {\side * sin(90 + \i*72)});
       
    }
    
 \coordinate (W1) at ($(P0) + (0,0.35)$);
  \coordinate (W2) at ($(P1) + (-0.2,0.3)$);
  \coordinate (Y) at ($(P2) + (-0.15,-0.15)$);
 \draw (P0) -- (P1) -- (P2) -- (P3) -- (P4) -- cycle;
\draw[red] (P0)--(W1)--(W2);
\draw[blue] (Y)--(P2);
\draw[blue, smooth] (Y) .. controls ($(P1) + (-0.7,0)$) and ($(P2) + (-0.3,0.15)$) .. (W2);

 \foreach \i in {0,1,2,3,4}{
 \ifnum\i=2
     \filldraw[fill=white,draw=black] (P\i) circle (1pt);  
            \else
             \ifnum\i=3
    \filldraw[fill=white,draw=black] (P\i) circle (1pt);  
             \else
      \filldraw[fill=black,draw=black] (P\i) circle (1pt);    
             \fi 
             \fi
  % \node at ($(P\i) + (0.3, 0.3)$) {\scriptsize $v_{\i}$};

 }
  \filldraw[fill=black,draw=black] (W1) circle (1pt); 
   \filldraw[fill=white,draw=black] (W2) circle (1pt); 
\filldraw[fill=white,draw=black] (Y) circle (1pt);

\node[above] at (W1){\scriptsize $w_1$};
%\node[above] at (W2){\scriptsize $w_2$};
\node[left] at (P1){\scriptsize $v_1$};
\node[below] at (P0){\scriptsize $v_2$};
\node[right] at (P4){\scriptsize $v_3$};
\node[left] at (P2){\scriptsize $u$};
\node[right] at (P3){\scriptsize$v$};
\node[below=0.4] at (P0){\scriptsize$C$};
\node at ($(P0)+(0.1, 0.2)$){\scriptsize \red{$P$}};
\node at ($(P1)+(-0.45, -0.2)$){\scriptsize \bl{$P'$}};
\node[below] at (Y){\scriptsize$y$};

\node[below=1.75] at (P0){\scriptsize (c) Theta-plus graph};
     \end{scope}

       \begin{scope}[xshift=8cm] 

        \def\side{0.35}
% 画五边形

 \foreach \i in {0,1,2,3,4} {
        \coordinate (P\i) at ({\side * cos(90 + \i*72)}, {\side * sin(90 + \i*72)});
    }

    % 绘制原始五边形
    \draw (P0) -- (P1) -- (P2) -- (P3) -- (P4) -- cycle;

    % 计算中点 M
    \coordinate (M) at ($(P2)!0.5!(P3)$);

    % 翻折每个顶点
    \foreach \i in {0,1,2,3,4} {
        \coordinate (P\i') at ($2*(M) - (P\i)$);
    }

    % 绘制翻折后的五边形
    \draw[blue] (P0') -- (P1') -- (P2') -- (P3') -- (P4') -- cycle;

 \coordinate (W1) at ($(P0) + (0,0.3)$);
\draw[red] (P0)--(W1);
\draw[red, smooth] (W1) .. controls ($(P4) + (0.5,0.1)$) .. (P1');
\draw[dashed] (P1')--(P2);

 \foreach \i in {0,1,2,3,4}{
 \ifnum\i=2
     \filldraw[fill=white,draw=black] (P\i) circle (0.9pt);  
            \else
             \ifnum\i=3
    \filldraw[fill=white,draw=black] (P\i) circle (0.9pt);  
             \else
      \filldraw[fill=black,draw=black] (P\i) circle (0.9pt);    
             \fi 
             \fi
  % \node at ($(P\i) + (0.3, 0.3)$) {\scriptsize $v_{\i}$};

 }

  \foreach \i in {0,1,4}{
     \filldraw[fill=black,draw=black] (P\i') circle (0.9pt);  

 }
 
  \filldraw[fill=black,draw=black] (W1) circle (0.9pt);

\node[above] at (W1){\scriptsize $w_1$};
\node[left] at (P1){\scriptsize $v_1$};
\node[above] at (P0){\scriptsize $v_2$};
\node[right] at (P4){\scriptsize $v_3$};
\node[left] at (P2){\scriptsize $u$};
\node[above=0.3] at (P0'){\scriptsize \bl{$C'$}};
\node[right] at (P3){\scriptsize$v$};
\node[below=0.4] at (P0){\scriptsize$C$};
\node[right=0.28] at (W1){\scriptsize \red{$P$}};
\node[right] at (P1'){\scriptsize$y$};
 % 在 scope 的右上方添加名字
        \node[below=2] at (P0) {\scriptsize (d) Theta-plus graph ended $v_2$, $u$}; 
     \end{scope}

\end{tikzpicture}
\caption{A pair of pseudo-twins $\{u, v\}$ of index $k$ in $G-I(P)$, where $k\ge 1$ } \label{fig:adjacent pseudo-twins of $G-I(P)$}
\end{figure}
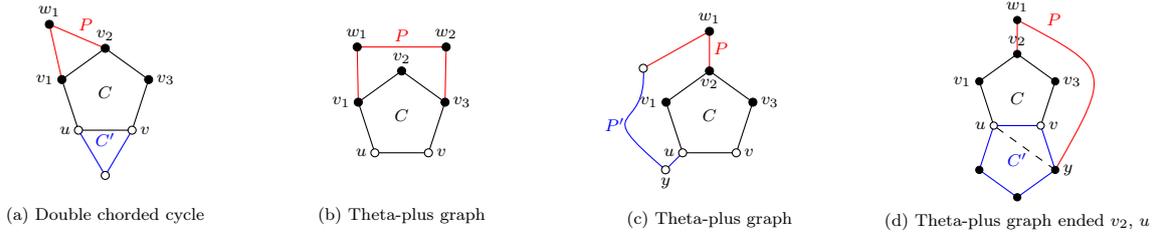

Assume $N_G(I(P)) \cap V(C)$ contains exactly one vertex  $v_i$. Without loss of generality, assume that $i \ne 1$ (hence $v_i$ is not adjacent to $u$).
Note that every vertex of $V(C) - \{v_i,u,v\}$ has degree 2 in $G$ and hence is not adjacent to any vertex not in $C$.
Let $P'$ be a shortest path in $G-v_i$ connecting $I(P)$ to $\{u,v\}$. Let $y$ be the vertex of $P'$ adjacent to $\{u,v\}$. 
If $y$ is adjacent to $u$, then the subgraph of $G$ induced by  $I(P) \cup V(P') \cup V(C)$ is a theta-plus graph with end vertices $v_i$ and $u$, see Fig. \ref{fig:adjacent pseudo-twins of $G-I(P)$}(c). Assume $y$ is adjacent to $v$ but not $u$. By the definition of pseudo-twins, $y$ is contained in another odd cycle $C'$ in $G-I(P)$ of length at least 5 containing edge $uv$. Then the subgraph of $G$ induced by  $I(P) \cup V(P') \cup V(C) \cap V(C')$ contains three induced paths of lengths at least 2 connecting $v_i$ and $u$, two of them are the two paths in $C$, and the third one go through $P'$ and part of $C'$. These three paths induces a   theta-plus graph with end vertices $v_i$ and $u$, see Fig. \ref{fig:adjacent pseudo-twins of $G-I(P)$}(d). 
%If $v_i$ is not adjacent to $v$ or $y$ is adjacent to $u$,   then  the subgraph of $G$ induced by  $I(P) \cup V(P') \cup V(C)$ is a theta-plus graph with end vertices $v_i$ and $v$, or with end vertices $v_i$ and $u$.
%Assume $v_i$ is adjacent to $v$ and $y$ is not adjacent to $u$. If there is a path $P''$ in $G-\{v_i, v\}$ connecting $P$ and $u$, then   $I(P) \cup V(P'') \cup V(C)$ is a theta-plus graph with end vertices  $v_i$ and $u$. Otherwise, $I(P) \cup V(P') \cup V(C) \cup V(C')$ is a double chorded cycle. 

Assume next that $k=0$, i.e., $\{u,v\}$ is a pair of adjacent twins. Let $U$ be the set of common neighbours of $u$ and $v$. As $d_{G'}(u) = d_{G'}(v) \ge 3$, $|U| \ge 2$. Since $G$ contains no adjacent twins and no induced even cycles, exactly one of $u,v$ is adjacent to a vertex of $P$, say $u \sim w_1$. Let $P'$ be a shortest path in $G-u$ connecting $w_1$ and a vertex     $z \in U$.    Let $y$ be the neighbour of $z$ in $P'$.   If $z$ is not adjacent to a vertex $z' \in U$, then either $y$ is adjacent to $z'$ and $\{v, y, z, z'  \}$ induces 4-cycle, see Fig. \ref{fig:adjacent twins of $G-I(P)$}(a), or $y$ is not adjacent to $z'$ and   $\{u, v, z, z'  \}  \cup V(P')$ induces a double chorded cycle, see Fig. \ref{fig:adjacent twins of $G-I(P)$}(b).  Thus $z$ is adjacent to every vertex of $U$.

\begin{figure}
    \centering
\begin{tikzpicture}[scale=1.5, every node/.style={scale=0.8}, baseline] 
%$u,v$ is a pair of adjacent twins 
 
     \begin{scope}[xshift=0cm]     
 \coordinate (u) at (0,0);
  \coordinate (v) at (0.6,0);
    \coordinate (z1) at (0.1,-0.6);
      \coordinate (z) at (0.5,-0.6);
        \coordinate (w1) at (0.2,0.5);
         \coordinate (w2) at (0.9,0.5);
          \coordinate (y) at (0.3,-0.9);
 \draw (u) -- (v) -- (z)  -- (u) -- (z1)--(v);
  \draw (y) -- (z1);
\draw[red] (u)--(w1)--(w2);
\draw[blue] (y)--(z);
\draw[dashed] (z1)--(z);
 \draw[draw=gray] (0.3,-0.6) ellipse [x radius=0.6, y radius=0.15];

\draw[blue, smooth] (y) .. controls ($(v) + (0.7,-0.6)$)  .. (w2);

 \foreach \point/\fillColor in {u/black, v/black, z/white, z1/white, w1/black, w2/white, y/white} {
  \filldraw[fill=\fillColor,draw=black] (\point) circle (1pt);
}

\node[above] at (w1){\scriptsize $w_1$};
\node[left] at (u){\scriptsize $u$};
\node[right] at (v){\scriptsize $v$};
\node[left] at (z1){\scriptsize $z'$};
\node[right] at (z){\scriptsize $z$};
\node[below] at (y){\scriptsize$y$};
\node[left=0.6] at (z1){\scriptsize {$U$}};
\node[right=0.7] at (v){\scriptsize \bl{$P'$}};
\node at (0.5,0.6){\scriptsize \red{$P$}};
 % 在 scope 的右上方添加名字
 
        \node[below=1.8] at (0.4,0) {\scriptsize (a) Induced 4-cycle $[uzyz']$}; 
     \end{scope}

       \begin{scope}[xshift=3.5cm] 
   \coordinate (u) at (0,0);
  \coordinate (v) at (0.6,0);
    \coordinate (z1) at (0.1,-0.6);
      \coordinate (z) at (0.5,-0.6);
        \coordinate (w1) at (0.2,0.5);
         \coordinate (w2) at (0.9,0.5);
          \coordinate (y) at (0.3,-0.9);
 \draw (u) -- (v) -- (z)  -- (u) -- (z1)--(v);
\draw[red] (u)--(w1)--(w2);
\draw[blue] (y)--(z);
\draw[dashed] (y)--(z1)--(z);
 \draw[draw=gray] (0.3,-0.6) ellipse [x radius=0.6, y radius=0.15];

\draw[blue, smooth] (y) .. controls ($(v) + (0.7,-0.6)$)  .. (w2);

 \foreach \point/\fillColor in {u/black, v/black, z/white, z1/white, w1/black, w2/white, y/white} {
  \filldraw[fill=\fillColor,draw=black] (\point) circle (1pt);
}

\node[above] at (w1){\scriptsize $w_1$};
\node[left] at (u){\scriptsize $u$};
\node[right] at (v){\scriptsize $v$};
\node[left] at (z1){\scriptsize $z'$};
\node[right] at (z){\scriptsize $z$};
\node[below] at (y){\scriptsize$y$};
\node[left=0.6] at (z1){\scriptsize {$U$}};
\node[right=0.7] at (v){\scriptsize \bl{$P'$}};
\node at (0.5,0.6){\scriptsize \red{$P$}};
 % 在 scope 的右上方添加名字
 
        \node[below=1.8] at (0.4,0) {\scriptsize (b) Double chorded cycle}; 
     \end{scope}

          \begin{scope}[xshift=7cm] 
   \coordinate (u) at (0,0);
  \coordinate (v) at (0.6,0);
    \coordinate (z1) at (0.1,-0.6);
     \coordinate (z2) at (-0.3,-0.6);
      \coordinate (z) at (0.5,-0.6);
        \coordinate (w1) at (0.2,0.5);
         \coordinate (w2) at (0.9,0.5);
          \coordinate (y) at (0.3,-0.9);
 \draw (z)  -- (u) -- (z1);
 \draw (u) -- (z2);
\draw[red] (u)--(w1)--(w2);
\draw[blue] (y)--(z);
\draw (z1)--(z);
\draw[gray!50,dashed] (u) -- (v) -- (z);
\draw[gray!50,dashed]  (z1)--(v)--(z2);
 \draw[draw=gray] (0.1,-0.6) ellipse [x radius=0.7, y radius=0.15];

\draw[smooth] (z2) .. controls (0.1,-0.7)  .. (z);
\draw[blue, smooth] (y) .. controls ($(v) + (0.7,-0.6)$)  .. (w2);

 \foreach \point/\fillColor in {u/black, v/gray, z/white, z1/white, z2/white, w1/black, w2/white, y/white} {
  \filldraw[fill=\fillColor,draw=black] (\point) circle (1pt);
}

\node[above] at (w1){\scriptsize $w_1$};
\node[left] at (u){\scriptsize $u$};
\node[right] at (v){\scriptsize $v$};
\node[right] at (z){\scriptsize $z$};
\node[below] at (y){\scriptsize$y$};
\node[right=0.7] at (v){\scriptsize \bl{$P'$}};
\node at (0.5,0.6){\scriptsize \red{$P$}};
\node[left=0.5] at (z2){\scriptsize {$U$}};
 % 在 scope 的右上方添加名字
 
        \node[below=1.8] at (0.4,0) {\scriptsize (c) Local structure of $G-v$};
     \end{scope}

\end{tikzpicture}

\caption{A pair of twins $\{u, v\}$ in $G-I(P)$} \label{fig:adjacent twins of $G-I(P)$}
\end{figure}
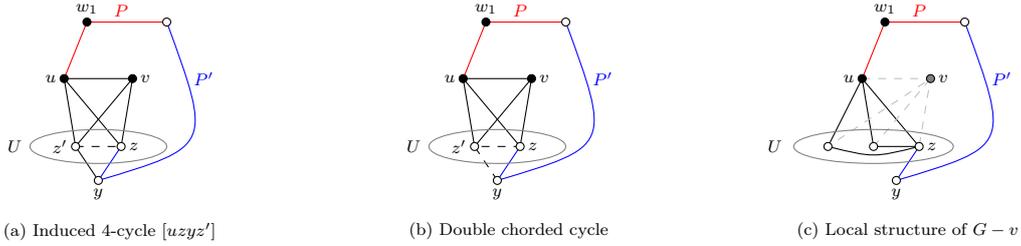
As $N_G(v) \subseteq N_G(u) \cap N_G(z)$, by Lemma \ref{lem-2conn-del}, $G-v$ is 2-connected, see Fig. \ref{fig:adjacent twins of $G-I(P)$}(c). 
If $G-v$ contains a triple of 2-vertices $(v_1,v_2,v_3)$, then since $G$ contains no triple of 2-vertices, there is an index $ i \in \{1,2,3\}$ such that $v_i$ is adjacent to $v$. But $v_i$ is adjacent to both $u$ and $z$ in $G-v$, and $d_{G-v}(u), d_{G-v}(z) \ge 3$, a contradiction. Thus $G-v$ contains no triple of 2-vertices. By Lemma \ref{lem-p1}, $G-v$ has a pair of adjacent twins $t_1,t_2$, with $d_{G-v}(t_1) = d_{G-v}(t_2) \ge 3$.
As $G$ contains no adjacent twins, we may assume that $v$ is adjacent to $t_1$ and not adjacent to $t_2$. If $t_1 =u$, then $t_2=w_1$, as $w_1$ is the only neighbor of $u$ that is not adjacent to $v$. However, $d_G(w_1) = 2$, a contradiction. Thus $t_1 \ne u$, and hence $t_1 \sim u$. This implies that $t_2 \sim u$. Again as  $w_1$ is the only neighbor of $u$ that is not adjacent to $v$, we have $t_2=w_1$,  a contradiction.

\bigskip
\noindent
{\bf Case 2} $G$ has a vertex $w$ such that $G-w$ is 2-connected.

We choose such a vertex $w$ of smallest degree. We may assume that Case 1 does not apply. Hence $d_G(w) \ge 3$. 

  By Lemma \ref{lem-p1},  $G'=G-w$ contains a pair of adjacent pseudo-twins $u,v$ of index $k$ for some $k \ge 0$.  By Corollary \ref{cor-c+}, 
   $G'$ is not a cycle, and hence $d_{G'}(u) = d_{G'}(v) \ge 3$. 

\begin{clm}
    \label{clm-k=0}
    $k=0$.
\end{clm}
\begin{proof}
    Assume to the contrary  that $k \ge 1$. Then $uv$ is contained in a cycle 
    $C=[v,u,v_1,v_2, \ldots, v_{2m+1}]$ of length at least 5, where $d_{G'}(v_i) =2$ for 
    $i=1,2,\ldots, 2m+1$. Similarly, it is obvious that $G-w$ is not a cycle and hence $uv$ is contained in another odd cycle $C'$ (which maybe a triangle).

Since $G$ contains no triple of 2-vertices, some $v_i$ is adjacent to $w$. Since $G$ does not contain a near odd wheel, $w$ has at most three neighbours in $C$, and if $w$ has at least two neighbours in $C$, then the neighbours are consecutive vertices of $C$.

If $w$ is adjacent to three consecutive vertices of $C$, with the middle vertex $v'$ of degree 3 in $G$, then by Lemma \ref{lem-2conn-del}, $G-v'$ is 2-connected. By the choice of $w$, 
$d_G(w) = d_G(v') =3$. Then $\{w,v'\}$ is a pair of adjacent twins, see Fig. \ref{fig:adjacent pseudo-twins of $G-w$}(a), a contradiction. Thus $w$ is not adjacent to three consecutive vertices of $C$ with the middle vertex $v'$ of degree 3 in $G$.

Thus by symmetry, we may assume that $v_1$ is not adjacent to $w$, and hence $d_G(v_1) =2$, and $w$ is adjacent to $ v_{i_0}$ for some $2 \le i_0 \le 3$.
Let $P$ be the proper ear of $G$ that contains $v_1$. Since Case 1 does not apply, some $v' \in V(C)$ is a cut-vertex in  $G-I(P)$. 
Note that $v'$ separates $u$ and $v_{i_0}$ in  $G-I(P)$.  Thus $w$ is not adjacent to $u$, and hence $w$ is adjacent to at most two vertices of $C$
(if $w$ is adjacent to 3 vertices of $C$, then the three vertices are consecutive in $C$ and the middle vertex has degree 3 in $G$). 

Let $C'$ be another induced cycle, maybe triangle, contained the edge $uv$. Since $v'$ separates $u$ and $v_{i_0}$ in $G-I(P)$, then there is no edge between $w$ and $V(C')-\{v\}$.   If $w$ is adjacent to two vertices of $C$, then these two vertices are consecutive in $C$, and hence the set $\{w\} \cup V(C) \cup V(C')$ induces a double chorded cycle, see Fig. \ref{fig:adjacent pseudo-twins of $G-w$}(b). This is a contradiction. 
Thus $w$ is adjacent to exactly one vertex $v_{i_0}$ of $C$. 

 %==========================================
 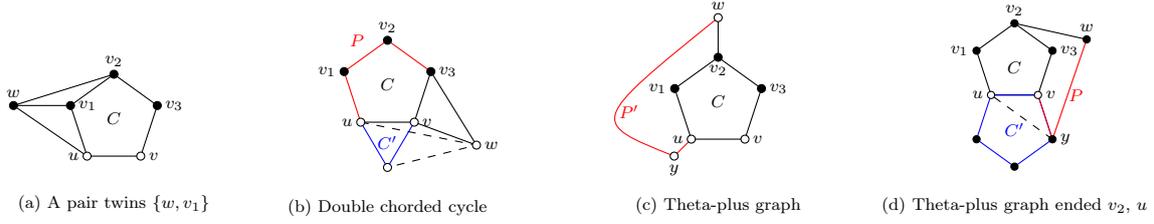
\begin{figure}
     \centering
\begin{tikzpicture}[scale=1.5, every node/.style={scale=0.8}, baseline] 
%$u,v$ is a pair of adjacent psedo-twins of index $k$, $k\ge 1$ 
    % 定义正五边形的边长
  \def\side{0.4}
     \begin{scope}[xshift=0cm, yshift=-0.3cm]
    % 计算正五边形的五个顶点
    \foreach \i in {0,1,2,3,4} {
        \coordinate (P\i) at ({\side * cos(90 + \i*72)}, {\side * sin(90 + \i*72)});
       
    }
    
 \coordinate (W1) at ($(P1) + (-0.5,0)$);
 \draw (P0) -- (P1) -- (P2) -- (P3) -- (P4) -- cycle;
\draw (P0)--(W1)--(P1);
\draw (W1)--(P2);

 \foreach \i in {0,1,2,3,4}{
 \ifnum\i=2
     \filldraw[fill=white,draw=black] (P\i) circle (1pt);  
            \else
             \ifnum\i=3
    \filldraw[fill=white,draw=black] (P\i) circle (1pt);  
             \else
      \filldraw[fill=black,draw=black] (P\i) circle (1pt);    
             \fi 
             \fi
  % \node at ($(P\i) + (0.3, 0.3)$) {\scriptsize $v_{\i}$};

 }
  \filldraw[fill=black,draw=black] (W1) circle (1pt);

\node[above] at (W1){\scriptsize $w$};
\node[right] at (P1){\scriptsize $v_1$};
\node[above] at (P0){\scriptsize $v_2$};
\node[right] at (P4){\scriptsize $v_3$};
\node[left] at (P2){\scriptsize $u$};
\node[right] at (P3){\scriptsize$v$};
\node[below=0.4] at (P0){\scriptsize$C$};
 % 在 scope 的右上方添加名字
        \node[below=1.5] at (P0) {\scriptsize (a) A pair twins $\{w, v_1\}$}; 
     \end{scope}

       \begin{scope}[xshift=2.4cm] 
    \foreach \i in {0,1,2,3,4} {
        \coordinate (P\i) at ({\side * cos(90 + \i*72)}, {\side * sin(90 + \i*72)});
       
    }
    
 \coordinate (W1) at ($(P4) + (0.4,-0.65)$);
  \coordinate (C1) at ($(P2)!0.5!(P3) + (0,-0.4)$);
 \draw (P2) -- (P3) -- (P4)--(W1)--(P3);
\draw[red] (P2)--(P1)--(P0)--(P4);
\draw[blue] (P2)--(C1)--(P3);
\draw[dashed] (P2)--(W1)--(C1);

 \foreach \point/\fillColor in {P1/black, P0/black, P4/black, P2/white, P3/white, W1/white, C1/white} {
  \filldraw[fill=\fillColor,draw=black] (\point) circle (1pt);
}

\node[right] at (W1){\scriptsize $w$};
\node[left] at (P1){\scriptsize $v_1$};
\node[above] at (P0){\scriptsize $v_2$};
\node[right] at (P4){\scriptsize $v_3$};
\node[left] at (P2){\scriptsize $u$};
\node[above=0.13] at (C1){\scriptsize \bl{$C'$}};
\node[right] at (P3){\scriptsize$v$};
\node[below=0.4] at (P0){\scriptsize$C$};
\node[left=0.2] at (P0){\scriptsize \red{$P$}};
 % 在 scope 的右上方添加名字
        \node[below=2] at (P0) {\scriptsize (b) Double chorded cycle}; 
     \end{scope}

            \begin{scope}[xshift=5.3cm, yshift=-0.15cm] 
    % 计算正五边形的五个顶点
    \foreach \i in {0,1,2,3,4} {
        \coordinate (P\i) at ({\side * cos(90 + \i*72)}, {\side * sin(90 + \i*72)});
       
    }
    
 \coordinate (W1) at ($(P0) + (0,0.35)$);
 % \coordinate (W2) at ($(P1) + (-0.2,0.3)$);
  \coordinate (Y) at ($(P2) + (-0.15,-0.15)$);
 \draw (P0) -- (P1) -- (P2) -- (P3) -- (P4) -- cycle;
\draw (P0)--(W1);
\draw[red] (Y)--(P2);
\draw[red, smooth] (Y) .. controls ($(P2) + (-0.9,0.15)$) .. (W1);

 \foreach \i in {0,1,2,3,4}{
 \ifnum\i=2
     \filldraw[fill=white,draw=black] (P\i) circle (1pt);  
            \else
             \ifnum\i=3
    \filldraw[fill=white,draw=black] (P\i) circle (1pt);  
             \else
      \filldraw[fill=black,draw=black] (P\i) circle (1pt);    
             \fi 
             \fi
  % \node at ($(P\i) + (0.3, 0.3)$) {\scriptsize $v_{\i}$};

 }
  \filldraw[fill=white,draw=black] (W1) circle (1pt); 
\filldraw[fill=white,draw=black] (Y) circle (1pt);

\node[above] at (W1){\scriptsize $w$};
\node[left] at (P1){\scriptsize $v_1$};
\node[below] at (P0){\scriptsize $v_2$};
\node[right] at (P4){\scriptsize $v_3$};
\node[left] at (P2){\scriptsize $u$};
\node[right] at (P3){\scriptsize$v$};
\node[below=0.4] at (P0){\scriptsize$C$};
\node at ($(P1)+(-0.4, -0.2)$){\scriptsize \red{$P'$}};
\node[below] at (Y){\scriptsize$y$};

\node[below=1.75] at (P0){\scriptsize (c) Theta-plus graph};
     \end{scope}

       \begin{scope}[xshift=7.9cm, yshift=0.2cm, ] 

        \def\side{0.35}
% 画五边形

 \foreach \i in {0,1,2,3,4} {
        \coordinate (P\i) at ({\side * cos(90 + \i*72)}, {\side * sin(90 + \i*72)});
    }

    % 绘制原始五边形
    \draw (P0) -- (P1) -- (P2) -- (P3) -- (P4) -- cycle;

    % 计算中点 M
    \coordinate (M) at ($(P2)!0.5!(P3)$);

    % 翻折每个顶点
    \foreach \i in {0,1,2,3,4} {
        \coordinate (P\i') at ($2*(M) - (P\i)$);
    }

    % 绘制翻折后的五边形
    \draw[blue] (P0') -- (P1') -- (P2') -- (P3') -- (P4') -- cycle;

\draw[red] (P1') -- (P2');

 \coordinate (W1) at ($(P4) + (0.3,0.1)$);
\draw(P0)--(W1);
\draw[red] (P1')--(W1);
\draw[dashed] (P1')--(P2);

 \foreach \i in {0,1,2,3,4}{
 \ifnum\i=2
     \filldraw[fill=white,draw=black] (P\i) circle (0.9pt);  
            \else
             \ifnum\i=3
    \filldraw[fill=white,draw=black] (P\i) circle (0.9pt);  
             \else
      \filldraw[fill=black,draw=black] (P\i) circle (0.9pt);    
             \fi 
             \fi
  % \node at ($(P\i) + (0.3, 0.3)$) {\scriptsize $v_{\i}$};

 }

  \foreach \i in {0,1,4}{
     \filldraw[fill=black,draw=black] (P\i') circle (0.9pt);  

 }
 
  \filldraw[fill=black,draw=black] (W1) circle (0.9pt);

\node[above] at (W1){\scriptsize $w$};
\node[left] at (P1){\scriptsize $v_1$};
\node[above] at (P0){\scriptsize $v_2$};
\node[right] at (P4){\scriptsize $v_3$};
\node[left] at (P2){\scriptsize $u$};
\node[above=0.3] at (P0'){\scriptsize \bl{$C'$}};
\node[right] at (P3){\scriptsize$v$};
\node[below=0.4] at (P0){\scriptsize$C$};
\node[right=0.3] at (P3){\scriptsize \red{$P$}};
\node[right] at (P1'){\scriptsize$y$};
 % 在 scope 的右上方添加名字
        \node[below=2.2] at (P0) {\scriptsize (d) Theta-plus graph ended $v_2$, $u$}; 
     \end{scope}

\end{tikzpicture}
\caption{A pair of pseudo-twins $\{u, v\}$ of index $k$ in $G-w$, where $k\ge 1$ } \label{fig:adjacent pseudo-twins of $G-w$}
 \end{figure}
%==============================================

Let $P'$ be a shortest path in $G-v_{i_0}$ connecting $w$ to $\{u,v\}$. Let $y$ be the vertex of $P'$ adjacent to $\{u,v\}$. 
If $y$ is adjacent to $u$ (and possibly to $v$ as well), then the subgraph of $G$ induced by  $ V(P') \cup V(C)$ is a theta-plus graph with end vertices $v_{i_0}$ and $u$, see Fig. \ref{fig:adjacent pseudo-twins of $G-w$}(c). Assume $y$ is adjacent to $v$ but not to $u$. By the definition of pseudo-twins, $y$ is contained in another odd cycle $C'$ of length at least $5$ containing edge $uv$, and $wy$ is an edge. By the same argument above, $w$ is adjacent to exactly one vertex $y$ of $C'$. Then the subgraph of $G$ induced by  $V(P') \cup V(C) \cup V(C')$ is a theta-plus graph with end vertices $v_{i_0}$ and $u$, see Fig. \ref{fig:adjacent pseudo-twins of $G-w$}(d).
This completes the proof of Claim \ref{clm-k=0}.
\end{proof}

By Claim \ref{clm-k=0},  $\{u,v\}$ is a pair of adjacent twins. Let $U$ be the set of common neighbours of $u$ and $v$. As $d_{G'}(u) = d_{G'}(v) \ge 3$, $|U| \ge 2$. Since $G$ contains no adjacent twins, exactly one of $u,v$ is adjacent to  $w$, say $u \sim w$. 

Let $P'$ be a shortest path in $G-u$  connecting $w$ to a vertex  $z \in U$.  Let $y$ be the neighbour of $z$ in $P'$. %Note that $y \ne v$, for otherwise $y \in U$, a contradiction.
If $z$ is not adjacent to a vertex $z' \in U$, then either $y$ is  adjacent to $z'$ and hence $\{v,y,z,z\}$ induces a 4-cycle, see Fig. \ref{fig:adjacent twins of $G-w$}(a), or 
$y$ is not adjacent to $z'$ and $\{u,v,z,z'\} \cup V(P')$ induces a double chorded cycle, see Fig. \ref{fig:adjacent twins of $G-w$}(b), a contradiction. Thus we may assume that   $z$ is adjacent to every vertex of $U$. 

\begin{figure}
    \centering
\begin{tikzpicture}[scale=1.5, every node/.style={scale=0.8}, baseline] 
%$u,v$ is a pair of adjacent twins 
 
     \begin{scope}[xshift=0cm]     
 \coordinate (u) at (0,0);
  \coordinate (v) at (0.6,0);
    \coordinate (z1) at (0.1,-0.6);
      \coordinate (z) at (0.5,-0.6);
        \coordinate (w1) at (0.4,0.5);
          \coordinate (y) at (0.3,-0.9);
 \draw (w1)--(u) -- (v) -- (z)  -- (u) -- (z1)--(v);
  \draw (y) -- (z1);
\draw[red] (y)--(z);
\draw[dashed] (z1)--(z);
 \draw[draw=gray] (0.3,-0.6) ellipse [x radius=0.6, y radius=0.15];

\draw[red, smooth] (y) .. controls ($(v) + (0.7,-0.6)$)  .. (w1);

 \foreach \point/\fillColor in {u/black, v/black, z/white, z1/white, w1/white, y/white} {
  \filldraw[fill=\fillColor,draw=black] (\point) circle (1pt);
}

\node[above] at (w1){\scriptsize $w$};
\node[left] at (u){\scriptsize $u$};
\node[right] at (v){\scriptsize $v$};
\node[left] at (z1){\scriptsize $z'$};
\node[right] at (z){\scriptsize $z$};
\node[below] at (y){\scriptsize$y$};
\node[left=0.6] at (z1){\scriptsize {$U$}};
\node[right=0.5] at (v){\scriptsize \red{$P'$}};
 % 在 scope 的右上方添加名字
 
        \node[below=1.8] at (0.4,0) {\scriptsize (a) Induced 4-cycle $[uzyz']$}; 
     \end{scope}

       \begin{scope}[xshift=3.5cm] 
   \coordinate (u) at (0,0);
  \coordinate (v) at (0.6,0);
    \coordinate (z1) at (0.1,-0.6);
      \coordinate (z) at (0.5,-0.6);
        \coordinate (w1) at (0.4,0.5);
          \coordinate (y) at (0.3,-0.9);
 \draw (w1)--(u) -- (v) -- (z)  -- (u) -- (z1)--(v);
\draw[red] (y)--(z);
\draw[dashed] (y)--(z1)--(z);
 \draw[draw=gray] (0.3,-0.6) ellipse [x radius=0.6, y radius=0.15];

\draw[red, smooth] (y) .. controls ($(v) + (0.7,-0.6)$)  .. (w1);

 \foreach \point/\fillColor in {u/black, v/black, z/white, z1/white, w1/white,  y/white} {
  \filldraw[fill=\fillColor,draw=black] (\point) circle (1pt);
}

\node[above] at (w1){\scriptsize $w$};
\node[left] at (u){\scriptsize $u$};
\node[right] at (v){\scriptsize $v$};
\node[left] at (z1){\scriptsize $z'$};
\node[right] at (z){\scriptsize $z$};
\node[below] at (y){\scriptsize$y$};
\node[left=0.6] at (z1){\scriptsize {$U$}};
\node[right=0.4] at (v){\scriptsize \red{$P'$}};
 % 在 scope 的右上方添加名字
 
        \node[below=1.8] at (0.4,0) {\scriptsize (b) Double chorded cycle}; 
     \end{scope}

          \begin{scope}[xshift=7cm] 
   \coordinate (u) at (0,0);
  \coordinate (v) at (0.6,0);
    \coordinate (z1) at (0.1,-0.6);
     \coordinate (z2) at (-0.3,-0.6);
      \coordinate (z) at (0.5,-0.6);
        \coordinate (w1) at (0.4,0.5);
          \coordinate (y) at (0.3,-0.9);
 \draw (z)  -- (u) -- (z1);
 \draw (w1)--(u) -- (z2);
\draw[red] (y)--(z);
\draw (z1)--(z);
\draw[gray!50,dashed] (u) -- (v) -- (z);
\draw[gray!50,dashed]  (z1)--(v)--(z2);
 \draw[draw=gray] (0.1,-0.6) ellipse [x radius=0.7, y radius=0.15];

\draw[smooth] (z2) .. controls (0.1,-0.7)  .. (z);
\draw[red, smooth] (y) .. controls ($(v) + (0.7,-0.6)$)  .. (w1);

 \foreach \point/\fillColor in {u/black, v/gray!50, z/white, z1/white, z2/white, w1/white, y/white} {
  \filldraw[fill=\fillColor,draw=black] (\point) circle (1pt);
}

\node[above] at (w1){\scriptsize $w$};
\node[left] at (u){\scriptsize $u$};
\node[right] at (v){\scriptsize $v$};
\node[right] at (z){\scriptsize $z$};
\node[below] at (y){\scriptsize$y$};
\node[right=0.4] at (v){\scriptsize \red{$P'$}};
\node[left=0.5] at (z2){\scriptsize {$U$}};
 % 在 scope 的右上方添加名字
 
        \node[below=1.8] at (0.4,0) {\scriptsize (c) Local structure of $G-v$};
     \end{scope}

\end{tikzpicture}

\caption{A pair of twins $\{u, v\}$ in $G-w$} \label{fig:adjacent twins of $G-w$}
 \end{figure}
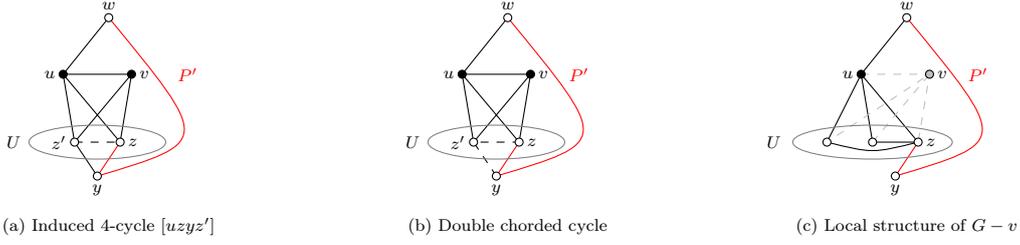
 
As $N_G(v) \subseteq N_G(u) \cap N_G(z)$, by Lemma \ref{lem-2conn-del}, $G-v$ is 2-connected, see Fig. \ref{fig:adjacent twins of $G-w$}(c). 
If $G-v$ contains a triple of 2-vertices $(v_1,v_2,v_3)$, then since $G$ contains no triple of 2-vertices, there is an index $ i \in \{1,2,3\}$ such that $v_i$ is adjacent to $v$. But $v_i$ is adjacent to both $u$ and $z$ in $G-v$, and $d_{G-v}(u), d_{G-v}(z) \ge 3$, a contradiction. Thus $G-v$ contains no triple of 2-vertices. 

By Lemma \ref{lem-p1}, $G-v$ has a pair of adjacent twins $t_1,t_2$, with $d_{G-v}(t_1) = d_{G-v}(t_2) \ge 3$.
As $G$ contains no adjacent twins, we may assume that $v$ is adjacent to $t_1$ and not adjacent to $t_2$. As $w$ is the only neighbor of $u$ that is not adjacent to $v$, we conclude that  $t_2=w$, as either   $t_1 \sim u$, or  $t_1 =u$.

Assume $t_1\in U$. If there is a common neighbour $x$ of $t_1$ and $w$ such that $x\notin U\cup\{u\}$, $\{t_1,w,u,v,x\}$ induces a double chorded cycle, see Fig. \ref{fig:adjacent twins of $G-w$-2}(a), a contradiction. Thus $N_G(w)\in U\cup \{u\}$. If there is a vertex $z'$ in $U$ such that $wz'\notin E(G)$, then $\{z',u,w,t_1,v\}$ induces a double chorded cycle, see Fig. \ref{fig:adjacent twins of $G-w$-2}(b) , a contradiction. 
Hence $N_{G-v}(w)=U\cup\{u\}= N_{G-v}(t_1)=N_G(v)$ for both $t_1\in U$ and $t_1=u$.
By Lemma \ref{lem-2conn-del}, $G-u$ is 2-connected, see Fig. \ref{fig:adjacent twins of $G-w$-2}(c),  and  contains a pair of adjacent pseudo-twins $t'_1,t'_2$ of index $k$ for some $k \ge 0$. It is easy to verify that $G-u$ contains no triple of 2-vertices. Hence $k=0$ and $t'_1,t'_2$ is a pair of adjacent twins in $G-u$. As $G$ contains no adjacent twins, we may assume that $u$ is adjacent to $t'_1$ but not $t'_2$. As $N_{G-u}(t'_2)=N_{G-u}(t'_1)$, we have $t'_2 \in N_G(v) - N_G(u)$, contrary to the fact that $N_{G-w}(u) = N_{G-w}(v)$.

\begin{figure}
    \centering
\begin{tikzpicture}[scale=1.5, every node/.style={scale=0.8}, baseline] 
%$u,v$ is a pair of adjacent twins 

          \begin{scope}[xshift=0cm] 
        \coordinate (u) at (0,0);   
    \coordinate (v) at (0.6,0);
    \coordinate (z1) at (0.3,-0.6);
     \coordinate (z2) at (-0.3,-0.6);
      \coordinate (z) at (0.9,-0.6);
        \coordinate (w1) at (0.4,0.5);
          \coordinate (x) at (-0.6,0.2);

 \draw (x)--(w1)--(u) -- (z2)--(x);
 \draw (u)--(v)--(z2);
 \draw[smooth] (w1) .. controls (-0.4,0)  .. (z2);
 
\draw[gray!50] (v)--(z)--(u) -- (z1) -- (v)--(z1)--(z);
 \draw[draw=gray] (z1) ellipse [x radius=0.9, y radius=0.15];

\draw[smooth,gray!50] (z2) .. controls (0.3,-0.7)  .. (z);

 \foreach \point/\fillColor in {u/black, v/black, z/white, z1/white, z2/white, w1/white, x/white} {
  \filldraw[fill=\fillColor,draw=black] (\point) circle (1pt);
}

\node[above] at (w1){\scriptsize \bl{$w=t_2$}};
\node[below] at (z2){\scriptsize \bl{$t_1$}};
\node[left] at (u){\scriptsize $u$};
\node[right] at (v){\scriptsize $v$};
\node[right] at (z){\scriptsize $z$};
\node[below] at (x){\scriptsize$x$};
\node[left=0.5] at (z2){\scriptsize {$U$}};
 % 在 scope 的右上方添加名字
 
        \node[below=1.8] at (0.2,0) {\scriptsize (a) Double chorded cycle};
     \end{scope}

             \begin{scope}[xshift=3.5cm] 
                \coordinate (u) at (0,0);
  \coordinate (v) at (0.6,0);
    \coordinate (z1) at (0.3,-0.6);
     \coordinate (z2) at (-0.3,-0.6);
      \coordinate (z) at (0.9,-0.6);
        \coordinate (w1) at (0.4,0.5);

 \draw (w1)--(u) -- (z2)--(v)--(u)--(z1)--(v);
 \draw[smooth] (w1) .. controls (-0.4,0)  .. (z2);
 \draw[dashed] (z2)--(z1)--(w1);
\draw[gray!50] (v)--(z)--(u);
 \draw[draw=gray] (z1) ellipse [x radius=0.9, y radius=0.15];

\draw[smooth,gray!50] (z2) .. controls (0.3,-0.7)  .. (z);

 \foreach \point/\fillColor in {u/black, v/black, z/white, z1/white, z2/white, w1/white} {
  \filldraw[fill=\fillColor,draw=black] (\point) circle (1pt);
}

\node[above] at (w1){\scriptsize \bl{$w=t_2$}};
\node[below] at (z2){\scriptsize \bl{$t_1$}};
\node[left] at (u){\scriptsize $u$};
\node[below] at (z1){\scriptsize $z'$};
\node[right] at (v){\scriptsize $v$};
\node[right] at (z){\scriptsize $z$};
\node[below] at (x){\scriptsize$x$};
\node[left=0.5] at (z2){\scriptsize {$U$}};
 % 在 scope 的右上方添加名字
 
        \node[below=1.8] at (0.2,0) {\scriptsize (b) Double chorded cycle};
     \end{scope}  

                  \begin{scope}[xshift=7cm] 
   \coordinate (u) at (0,0);
  \coordinate (v) at (0.6,0);
    \coordinate (z1) at (0.3,-0.6);
     \coordinate (z2) at (-0.3,-0.6);
      \coordinate (z) at (0.9,-0.6);
        \coordinate (w1) at (0.3,0.5);

 \draw (z1)--(z)--(v)--(z1)--(w1);
 \draw (v)--(z2);
 \draw[smooth] (w1) .. controls (-0.4,0)  .. (z2);
  \draw[smooth] (w1) .. controls (1,0)  .. (z);
 \draw[dashed,gray!50] (z2)--(u)--(w1);
 \draw[dashed,gray!50] (z1)--(u)--(z);
\draw[dashed,gray!50] (v)--(u);
 \draw[draw=gray] (z1) ellipse [x radius=0.9, y radius=0.15];

\draw[smooth] (z2) .. controls (0.3,-0.7)  .. (z);

 \foreach \point/\fillColor in {u/gray!50, v/black, z/white, z1/white, z2/white, w1/black} {
  \filldraw[fill=\fillColor,draw=black] (\point) circle (1pt);
}

\node[above] at (w1){\scriptsize $w$};
\node[left] at (u){\scriptsize $u$};
\node[right] at (v){\scriptsize $v$};
\node[right] at (z){\scriptsize $z$};
\node[below] at (x){\scriptsize$x$};
\node[left=0.5] at (z2){\scriptsize {$U$}};
 % 在 scope 的右上方添加名字
 
        \node[below=1.8] at (0.2,0) {\scriptsize (c) Local structure of $G-u$};
     \end{scope}  
\end{tikzpicture}
\caption{A pair of twins $\{t_1, t_2\}$ in $G-v$ or 2-connected graph $G-u$} \label{fig:adjacent twins of $G-w$-2}
 \end{figure}
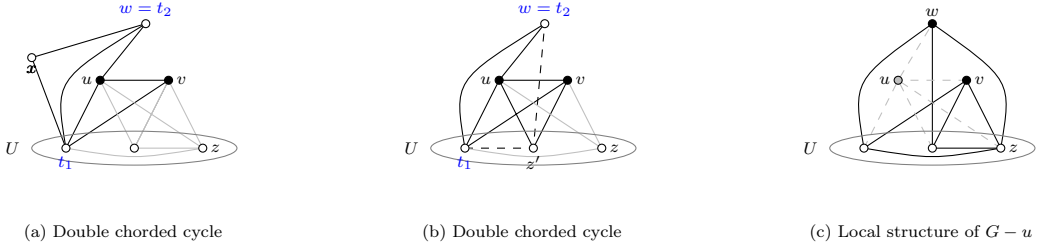

This completes the proof of Theorem \ref{thm-nic}. 
\end{proof}

\section{Expanded Gallai-trees}

The three   operations defined in Section 3 apply to pairs $(G,L)$. We may forget the list assignments, and consider the three operations as graph operations, and denote these operations by $D_v(G), T_v(G), G_1 \oplus_{(v_1,v_2)} G_2$, respectively.

As  a consequence of Theorem \ref{thm-nic}, we have the following corollary. 

\begin{cor}
\label{cor-nidc}
    A graph $G$ is not indicated degree-choosable if and only if one of the following holds:
\begin{enumerate}
\item $G =K_1$.
  \item   $G = G_1 \oplus_{(v_1,v_2)} G_2$, and each $G_i$ is   not indicated degree-choosable.
    \item  $G= D_v(G')$ and $G'$   is not indicated degree-choosable.
    \item $G=T_v(G')$ and $G'$ is not indicated degree-choosable. 
\end{enumerate}
\end{cor}

Recall that a connected graph $G$ is an expanded Gallai-tree if each expanded block of $G$ is a complete graph or a blow-up of an odd cycle of length at least 5, and each expanded block that is a blow up of an odd cycle of length at least 5 has at most one root-clique. 
Now we prove Theorem \ref{thm-indicateddegreechoosable}: A connected graph $G$ is not indicated degree-choosable if and only if $G$ is an expanded Gallai-tree.

\bigskip
{\bf Proof of Theorem \ref{thm-indicateddegreechoosable}}
    Assume $G$ is not indicated degree-choosable. We prove by induction on the number of vertices of $G$ that is $G$ is an expanded Gallai-tree. The case  $|V(G)|=1$ is trivial. Assume  $|V(G)| \ge 2$. If $G=G_1 \oplus_{(v_1,v_2)} G_2$ and $G_1,G_2$ are not indicated degree-choosable, then each expanded block of $G$ is an expanded block of $G_1$ or $G_2$. By induction hypothesis, $G_1,G_2$ are expanded Gallai-trees. Hence each expanded block of $G$ is a complete graph or a blow-up of an odd cycle, and $\mathcal{B}_G =\mathcal{B}_{G_1} \cup \mathcal{B}_{G_2}$, and $\pi_G(H) = \pi_{G_i}(H)$ if $H \in \mathcal{B}_{G_i}$.

    Assume $G$ is $2$-connected.
    
       If $G=D_v(G')$  and $G'$ is not indicated degree-choosable, then by induction hypothesis, $G'$ is an expanded Gallai-tree. If $H$ is an expanded block   of $G$, then either $v \notin V(H)$ and hence $H$ is   an expanded block of $G'$ or $v \in V(H)$, then $H=D_v(H')$, where $H'$ is an expanded block of $G'$. Also a clique-cut $K$ of $G$ is either a clique-cut of $G'$ or $K=D_v(K')$,  where $K'$ is a clique-cut of $G'$. Therefore each expanded block of $G$ is a complete graph or a blow-up of an odd cycle of length at least 5. Moreover, for $H \in \mathcal{B}_G$, either $H \in \mathcal{B}_{G'}$ and $\pi_G(H)=\pi_{G'}(H)$ or $H = D_v(H')$ for some $H' \in \mathcal{B}_{G'}$, and 
       \[
       \pi_G(H) = \begin{cases}
           D_v(\pi_{G'} (H')), &\text{ if   $v \in V(\pi_{G'} (H'))$}, \cr 
           \pi_{G'}(H'), &\text{ otherwise.}
       \end{cases}
       \]

       If $G=T_v(G')$, then $d_{G'}(v)=2$. Since $G$ is 2-connected, so is $G'$ and hence $v$ is contained in an odd cycle $C'$, which is contained in an expanded block $H'$ of $G'$. Then $H=T_v(H')$ is an expanded block of $G$.    So again, each expanded block of $G$ is a complete graph or a blow-up of an odd cycle. 
       If   $v$ is contained in a triangle $C'=[vuw]$ and $uw$ is a cliue-cut of $G'$, then $C=T_v(C') \in \mathcal{B}_G$  with $\pi_G(C)=\{u,w\}$.

Next we prove that if $G$ is an expanded Gallai-tree, then $G$ is not indicated degree-choosable, again by induction on the number of vertices.
 Assume that this is not true and $G$ is a minimum counterexample. Then $G$ is neither a complete graph nor an blow-up of an odd cycle. 
 
If $G$ has a pair of adjacent twins or a triple of 2-vertices, then $G=D_v(G')$ or $G=T_v(G')$. Each clique-cut $K$ of $G$ is either a clique-cut of $G'$ or $v \in V(K)$ and $G'$ has a clique-cut $K'$ such that $K=D_v(K')$. Similarly, each expanded block $H$ of $G$ is either an expanded block of $G'$ or $v \in V(H)$ and there is an expanded block $H'$ of $G'$ such that $H=D_v(H')$ or $H=T_v(H')$. Thus each expanded block of $G'$ is a complete graph or a blow-up of an odd cycle. Moreover,  $\pi_{G'}(H) =\pi_G(H)$ or is obtained from $\pi_G(H)$ by deleting the twin of $v$.  Hence $G'$ is an expanded Gallai-tree. 
By induction hypothesis, $G'$ is not indicated degree-choosable. By Corollary \ref{cor-nidc}, $G$ is not indicated degree-choosable. 

Assume $G$ has no adjacent twins and no triple of 2-vertices. Then $G$ has a cut-vertex and $G=G_1 \oplus_{(v_1,v_2)}G_2$, where each of $G_1,G_2$ is an expanded Gallai-tree. By induction hypothesis, each of $G_1,G_2$ is not indicated degree-choosable, and hence $G$ is not indicated degree-choosable.
\qed

\begin{cor}
    \label{cor-alg}
    There is a linear-time algorithm that determines of a graph $G$ is indicated degree-choosable.
\end{cor}
\begin{proof}
    We may assume that $G$ is a connected graph. We repeatedly apply  the reverse   of the duplicating operation, the tripling operation and vertex-sum operation, until no such revere operations can be applied. If we reach a set of isolated vertices, then $G$ is an expanded Gallai-tree and hence is not indicated degree-choosable. Otherwise, it is degree-choosable. It is obvious that this algorithm terminates in linear-time. 
\end{proof}

\section{Brooks' Theorem for indicated colouring}

\begin{definition}
    \label{def-HK}
    Assume $G$ is an expanded Gallai-tree. 
    Assume $K$ is a clique-cut of $G$ and $H \in \mathcal{B}_{G,K}$. Let $Z_{H,K} = V(H)-V(K)$ and $Q_{H,K}$ be the subgraph of $G$ induced by the union of $K$ and the vertex set of the connected component of $G-V(K)$ containing $Z_{H,K}$. 
    
     If $K$ is a clique-cut of $G$ with $\mathcal{B}_{G,K} \ne \emptyset$, then let $$H_K = G[\cup_{H \in \mathcal{B}_{G,K} }V(H)], Q_K =G[\cup_{H \in \mathcal{B}_{G,K} }Q_{H,K}] .$$
    See Fig. \ref{fig-13} for an example of a clique-cut $K$, and the graphs $H_K$ and $Q_K$.
  %  Let $Z_K = V(H_K)-V(K)$. 
\end{definition}

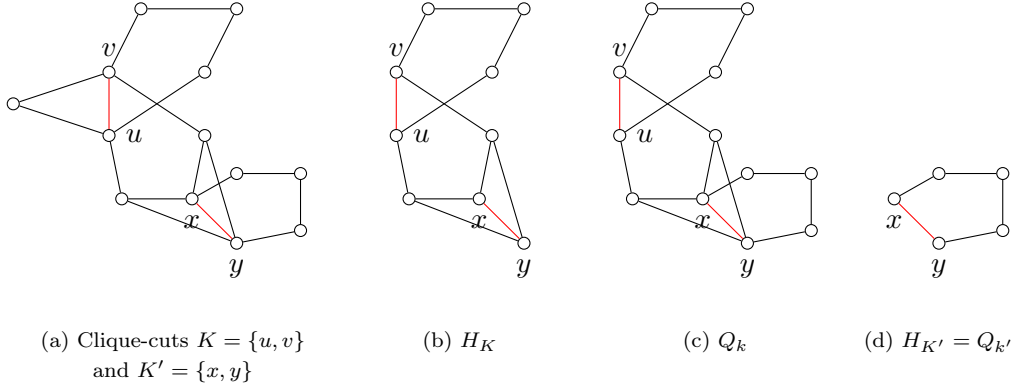
\begin{figure}
    \centering
\begin{tikzpicture}[scale=0.3, 
    dot/.style={circle, fill=white, draw=black, minimum size=4.5pt, inner sep=0pt}, % 默认点的样式
    label position/.style={above, below, left, right},
    scale=2.8, % Adjust scaling if nodes are too close
  ]

  % 默认点
  \node[dot, label=right:{$u$}] (a) at (-1.00, 0.00) {};
  \node[dot, label=above:{$v$}] (b) at (-1.00, 1.00) {};
  \node[dot, label=left:{}] (c) at (-2.50, 0.50) {};
  \node[dot, label=right:{}] (d) at (0.50, 1.00) {};
  \node[dot, label=below:{}] (g) at (-0.80, -1.00) {};
  \node[dot, label=above:{}] (e) at (1.00, 2.00) {};
  \node[dot, label=above:{}] (f) at (-0.50, 2.00) {};
  \node[dot, label=above:{}] (i) at (0.50, 0.00) {};
  \node[dot, label=below:{$x$}] (h) at (0.30, -1.00) {};
  \node[dot, label=below:{$y$}] (j) at (1.00, -1.70) {};
  \node[dot, label=below:{}] (k) at (1.00, -0.60) {};
  \node[dot, label=below:{}] (l) at (2.00, -1.50) {};
  \node[dot, label=above:{}] (m) at (2.00, -0.60) {};

 \node[below=1.2] (A) at (0, -1.50) {\shortstack{\scriptsize (a) Clique-cuts $K=\{u,v\}$ \\ \scriptsize and  $K'=\{x,y\}$}};
  %\node[below=1.2] (A) at (0, -1.50) {(a) Clique-cuts $K=\{u,v\}$ and $K'=\{x,y\}$};

  % 默认点和边
  \draw[red] (a) -- (b);
  \draw (a) -- (c);
  \draw (a) -- (d);
  \draw (a) -- (g);
  \draw (b) -- (c);
  \draw (b) -- (f);
  \draw (b) -- (i);
  \draw (d) -- (e);
  \draw (e) -- (f);
  \draw (g) -- (h);
  \draw (g) -- (j);
  \draw (h) -- (i);
  \draw[red] (h) -- (j);
  \draw (h) -- (k);
  \draw (i) -- (j);
  \draw (j) -- (l);
  \draw (l) -- (m);
  \draw (k) -- (m);

  % Draw $H_K$
  \node[dot, label=right:{$u$}] (u) at (-1.00+4.50, 0.00) {};
  \node[dot, label=above:{$v$}] (v) at (-1.00+4.50, 1.00) {};
  \node[dot, label=right:{}] (v3) at (0.50+4.50, 1.00) {};
  \node[dot, label=below:{}] (v6) at (-0.80+4.50, -1.00) {};
  \node[dot, label=above:{}] (v2) at (1.00+4.50, 2.00) {};
  \node[dot, label=above:{}] (v1) at (-0.50+4.50, 2.00) {};
  \node[dot, label=above:{}] (v4) at (0.50+4.50, 0.00) {};
  \node[dot, label=below:{$x$}] (v5) at (0.30+4.50, -1.00) {};
  \node[dot, label=below:{$y$}] (v7) at (1.00+4.50, -1.70) {};

  % Draw edges
  \draw[red] (u) -- (v);
  \draw (v) -- (v1);
  \draw (v) -- (v4);
  \draw (u) -- (v3);
  \draw (u) -- (v6);
  \draw (v4) -- (v5);
  \draw (v5) -- (v6);
  \draw[red]  (v5) -- (v7);
  \draw (v6) -- (v7);
  \draw (v4) -- (v7);
  \draw (v1) -- (v2); 
  \draw (v2) -- (v3);

  \node[below=1.2] (B) at (4.5, -1.50) {\scriptsize (b)   $H_K$};

  % Draw QK
  \node[dot, label=right:{$u$}] (u0) at (-1.00+8.00, 0.00) {};
  \node[dot, label=above:{$v$}] (v0) at (-1.00+8.00, 1.00) {};
  \node[dot, label=right:{}] (v30) at (0.50+8.00, 1.00) {};
  \node[dot, label=below:{}] (v60) at (-0.80+8.00, -1.00) {};
  \node[dot, label=above:{}] (v20) at (1.00+8.00, 2.00) {};
  \node[dot, label=above:{}] (v10) at (-0.50+8.00, 2.00) {};
  \node[dot, label=above:{}] (v40) at (0.50+8.00, 0.00) {};
  \node[dot, label=below:{$x$}] (v50) at (0.30+8.00, -1.00) {};
  \node[dot, label=below:{$y$}] (v70) at (1.00+8.00, -1.70) {};
  \node[dot, label=below:{}] (v80) at (1.00+8.00, -0.60) {};
  \node[dot, label=below:{}] (v100) at (2.00+8.00, -1.50) {};
  \node[dot, label=right:{}] (v90) at (2.00+8.00, -0.60) {};

  % Draw edges
  \draw[red] (u0) -- (v0);
  \draw (v0) -- (v10);
  \draw (v0) -- (v40);
  \draw (u0) -- (v30);
  \draw (u0) -- (v60);
  \draw (v40) -- (v50);
  \draw (v50) -- (v60);
  \draw[red] (v50) -- (v70);
  \draw (v60) -- (v70);
  \draw (v40) -- (v70);
  \draw (v10) -- (v20); 
  \draw (v20) -- (v30);
  \draw (v50) -- (v80);
  \draw (v80) -- (v90);
  \draw (v90) -- (v100);
  \draw (v70) -- (v100);

  \node[below=1.2] (c) at (8.5, -1.50) {\scriptsize (c) $Q_k$};

   % Draw QK
  \node[dot, label=below:{$x$}] (v50) at (0.30+11.00, -1.00) {};
  \node[dot, label=below:{$y$}] (v70) at (1.00+11.00, -1.70) {};
  \node[dot, label=below:{}] (v80) at (1.00+11.00, -0.60) {};
  \node[dot, label=below:{}] (v100) at (2.00+11.00, -1.50) {};
  \node[dot, label=right:{}] (v90) at (2.00+11.00, -0.60) {};

  % Draw edges
  \draw[red] (v50) -- (v70);
  \draw (v50) -- (v80);
  \draw (v80) -- (v90);
  \draw (v90) -- (v100);
  \draw (v70) -- (v100);

  \node[below=1.2] (d) at (12, -1.50) {\scriptsize (d) $H_{K'}=Q_{k'}$};
\end{tikzpicture}

\caption{ A graph $G$ with clique-cuts $K$ and $K'$,   with subgraphs  $H_K$, $H_{K'}$, $Q_K$ and $Q_{k'}$}
\label{fig-13}
\end{figure}

Note that for any $H \in \mathcal{B}_{G,K}$,  $H \subseteq Q_{H,K}$ and equality $H=Q_{H,K}$ holds  if and only if $H$ does not contain any other clique cut of $G$.

 \begin{theorem}
     \label{thm-charact}
     Assume $G$ is an  expanded Gallai-tree, and with $\Delta(G) \le r$. 
     Then the following are equivalent:  
     \begin{enumerate}
     \item[(1)] $G$ has a degree-list assignment $L$ such that $(G,L)$ is infeasible and $L(v) \subseteq [r]$ for each vertex $v$.
     \item[(2)] For each clique-cut $K$ with $\mathcal{B}_{G,K} \ne \emptyset$,   there is degree-list assignment   $L_K$ for $H_K$ such that
     $(H_K,L_K)$ is infeasible, $L_K(v)\subseteq [r]$ for each vertex $v$, and $L_K(v) \subseteq [r-t_K]$ if $v \in V(K)$,   where $t_K = |\cap_{v \in V(K)} N_G(v)|$. 
     \end{enumerate} 
 \end{theorem}
 \begin{proof}
$ (1) \Rightarrow (2)$: Let $L$ be   a degree-list assignment of $G$ such that  $(G,L)$ is infeasible and $L(v) \subseteq [r]$ for each vertex $v$. Assume $K$ is  a  clique-cut  of $G$  with  $  \mathcal{B}_{G,K} \ne \emptyset$.  By Lemma \ref{lem-remaining},  there is  an $L$-colouring $\phi$ of $ G-V(H_K)$ such that $(H_K, L^{\phi})$ is infeasible. Note that 
for any $x,y \in V(K)$, $N_{G}(x) - V(H_K)=N_{G }(y)-V(H_K) = \cap_{v \in V(K)}N_G(v)$. Let $X= \cap_{v \in V(K)}N_G(v)$. By definition $|X|=t_K$. By permutation of colours, we may assume that $[r]-\phi(X) = [r-t_K]$.   Let $L_K=L^{\phi}$, we obtain the required degree-list assignment of $G$.

 $ (2) \Rightarrow (1)$:
 If $G$ has no clique-cut, or $G$ has only one clique-cut $K$ and $H_K=G$, then 
 the condition and the conclusion are the same, and there is nothing to be proved. Thus we assume that $G$ has a clique-cut $K$ such that $H_K \ne G$.

\bigskip
\noindent
{\bf Case 1} $G$ has a clique-cut $K$ such that  $ \mathcal{B}_{G,K} = \emptyset$. 

Choose such a clique-cut $K$, and let $\mathcal{A}$ be the set of   expanded blocks $H$ of $G$ containing $K$. For  each $H \in \mathcal{A}$, there is a degree-list assignment $L_H$ of $Q_{H,K}$ such that $(Q_{H,K},L_H)$ is infeasible and $L_H(v) \subseteq [r]$ for each vertex $v$. As vertices in $K$ are adjacent twins in $Q_{H,K}$, by Lemma \ref{lem-twins}, there is  a set $C_H$ of $d_{Q_{H,K}}(v)$ colours (for $v \in V(K)$) such that $L_H(x) = C_H$ for all $x \in V(K)$. Let $S$ be a subset of $V(K)$ of size $|V(K)|-1$. By Lemma \ref{lem-remaining}, there is  
a $L_H$-colouring $\psi$ of   $S  $ such that $(Q_{H,K}-S, L_H^{\psi})$ is infeasible. Let  $A_H=\psi(S)$, which is a set of $|V(K)|-1$ colours contained in $C_H$. By a permutation of colours, we may assume that for any two distinct expanded blocks $H,H'$ containing $K$, $C_H \cap C_{H'} = A_H=A_{H'}$. Let $L = \cup_{H \in \mathcal{A} }  L_H$.

Now we show that $(G,L)$ is infeasible. By Lemma \ref{lem-remaining}, it suffices to find, for each vertex $v$ of $G$ a colour $c \in L(v)$ such that $(G-v, L^{v \to c})$ is infeasible. 
\begin{enumerate}
    \item If $v \in V(K)$, then let $c$ be any colour in $A_H$ defined above.
    \item If $v \in V(G)-V(K)$, then $v \in Q_{H,K}$ for some expanded block $H$ containing $K$. Let $c$ be a  colour in $L_H(v) $ such that $(Q_{H,K}-v, L_H^{v \to c})$ is infeasible.
\end{enumerate} 
It follows from induction hypothesis 
that $L^{v \to c}$ is a degree-list assignment of $G-v$   and $(G-v,L^{v \to c})$ is infeasible.

\bigskip
\noindent
{\bf Case 2} 
$  \mathcal{B}_{G,K} \ne \emptyset$ for each clique-cut $K$ of $G$. 

Let $K$ be a clique-cut such that  $Q_K$ is minimal. Assume $H \in \mathcal{B}_{G,K}$.
If $H$ contains another clique-cut $K'$ of $G$,  then for any $H' \in \mathcal{B}_{G,K'}$, $H'-V(K') \subseteq V(Q_{H,K})$ and hence $Q_{H', K'} \subsetneq  Q_{H,K}$. 
Hence $Q_{K'} \subsetneq Q_K$, contrary to the choice of $K$. 

So for any $H \in \mathcal{B}_{G,K}$, $H$ contains no other clique-cut of $G$, $Q_{H,K}=H$ and hence $Q_K = H_K$. 
 
By our assumption, $G-V(H_K) \ne \emptyset$. By the induction hypothesis,
there is a degree-list assignment $L'$ of $G'=G-(V(H_K) - V(K))$ such that $(G',L')$ is infeasible and $L'(v) \subseteq [r]$ for each vertex $v$ of $G'$. Since vertices in $K$ are adjacent twins in $G'$, there is a set $C'$ of $d_{G'}(v)$ colours ($v \in V(K)$) such that $L'(x)=C'$ for all $x \in V(K)$.  

Let $S$ be a subset of $V(K)$ of size $|V(K)|-1$. By Lemma \ref{lem-remaining}, there is  
a $L'$-colouring $\psi$ of   $S  $ such that $(G'-S, L'^{\psi})$ is infeasible. Let  $A=\psi(S)$, which is a set of $|V(K)|-1$ colours contained in $C'$. By a permutation of colours, we may assume that  $[r]-(C'-A) = [r-t_K]$. By assumption, there is a degree-list assignment $L_K$ of $H_K$ such that $(H_K,L_K)$ is infeasible, $L_K(v) \subseteq [r]$ for each vertex $v$ of $H_K$ and $L_K(v) \subseteq [r-t_K]$ for each vertex $v \in V(K)$. 
By Lemma \ref{lem-remaining}, there is a $L_K$-colouring $\tau$ of $S$ such that $(H_K-S, L_K^{\tau})$ is infeasible. By a permutation of colours, we may assume that $\tau(S) = A$. Let $L = L' \cup L_K$. Then $L$ is a degree-list assignment of $G$ with $L(v) \subseteq [r]$ for each vertex $v$ of $G$, and a similar argument as above shows that    $(G,L)$ is infeasible. 

This completes the proof of Theorem \ref{thm-charact}.%It remains to show that $(G,L)$ is infeasible. By Lemma \ref{lem-remaining}, it suffices to show that for any vertex $v$ of $G$, there is a colour $c \ in L(v)$ such that $(G-v, L^{v \to c})$ is infeasible. 
%If $v \in V(K)$, then let $c$ be any colour in $A$. If $v \in Z$, then let $c \in L(v) =L_{H'}(v)$ be the colour such that $(H', L_{H'})$ is infeasible. If $v \in V(G-V(H'))$, then let $c \in L'(v) = L(v)$ be the colour such that $(G-Z-v, L'^{v \to c})$ is infeasible. It follows by induction hypothesis that $(G-v, L^{v \to c})$ is infeasible. Hence $(G,L)$ is infeasible.
 \end{proof}

\begin{lem}
    \label{lem-k}
    Assume $H$ is a graph consisting of $k$ odd cycles $C_1,C_2, \ldots, C_k$ sharing one edge $uv$. Let $p: V(H) \to \{1,2,\ldots \}$ and $G=H[p]$ is a blow-up of $H$ in which each vertex $x$ is blown up into a clique $K(x)$ of order $p(x)$. Let $K =K(u) \cup K(v)$.  Let $r \ge s$ be positive integers such that $d_G(v) \le r$ for each vertex $v$ of $G$ and $d_G(v) \le s$ for $v \in V(K)$. Then the following are equivalent:
    \begin{enumerate}
        \item[(i)] There is a degree-list assignment $L$ of $G$ such that $(G,L)$ is infeasible, $L(v) \subseteq [r]$ for $v \in V(G)$ and   $L(v) \subseteq [s]$ for  $v \in V(K)$.
        \item[(ii)] There is a multifold coloring $f $ of $H^2$ such that $|f(x)|=p(x)-1$  and $f(x)\subseteq [r] - \{i, k+1\}$ if $x \in V(C_i) - N_H[\{u,v\}]$,  and $f(x) \subseteq [s] - \{1,2,\ldots, k+1\}$ if $x \in N_H[\{u,v\}]$.  Here $H^2$ is the graph obtained from $H$ by adding edges connecting pairs of vertices of distance $2$ in $H$.
        \end{enumerate}  
\end{lem}
\begin{proof}
    Let $L'$ be the degree list assignment of $H$ defined as $L'(u)=L'(v) = \{1,2,\ldots, k+1\}$, and $L'(x) = \{i,k+1\}$ for $x \in V(C_i)-\{u,v\}$.
    Up to a permutation of colours, $L'$ is the   unique degree list assignment   of $H$ such that $(H,L')$ is infeasible (see the example graph in Fig. \ref{fig-5}). 
    Thus  the required $L$ exists if and only if  $(G,L)$ can be obtained from $(H, L')$  by a sequence of duplication operations. The order of the duplication is irrelevant. For each vertex $x$ of $H$, it is duplicated $p(x)-1$ times. Hence, a set $Z_x$ of $p(x)-1$ colors need to be chosen. These colours are added to the lists $L'(y)$ for  all $y \in N_H[x]$.  Hence if $x \in N_H[\{u,v\}]$, then  $Z_x \subseteq [s] - L'(u)$. Otherwise, $Z_x \subseteq [r]-L'(x)$. Moreover, if $xy$ is an edge or $x,y$ have a common neighbor $z$, then $Z_x \cap Z_y = \emptyset$. Therefore $f(x)=Z_x$ is a multifold colouring of $H^2$ as stated in (ii).
\end{proof}

Theorem \ref{thm-charactIC} below is a consequence of Theorem \ref{thm-charact} and Lemma \ref{lem-k}, and characterizes IC-Brooks graphs.

\begin{theorem}
     \label{thm-charactIC}
     Assume $G$ is an $r$-regular expanded Gallai-tree. Then $G$ is an IC-Brooks graph if  the following hold: \begin{itemize}
     \item  For each   is clique-cut $K$ with $\mathcal{B}_{G,K} \ne \emptyset$, if  $H_K=H[p]$, where $H$ consists of a family of $k$ odd cycles $C_1,C_2, \ldots, C_k$ sharing an edge $uv$, and $p: V(H) \to \{1,2,\ldots\}$, then  
     there is a multifold colouring $f$ of $H_K^2$ with $|f(x)|=p(x)-1$ and   
     $f(x) \subseteq [r] -\{i, k+1\}$ for $x \in V(C_i) - N_H[\{u,v\}]$
     and   $f(x) \subseteq [r-t_K]-\{1,2,\ldots, k+1\}$ for $x \in N_H[\{u,v\}]$, where     $t_K = |\cap_{v \in V(K)} N_{G}(v)|$. T
     \end{itemize} 
\end{theorem}

Assume $r$ is a fixed integer and $G$ is an $r$-regular expanded Gallai-tree. For each clique-cut $K$ of $G$, if $\mathcal{B}_{G,K} \ne \emptyset$, then $H_K=H[p]$, where $H$ consists of a family of $k$ odd cycles $C_1,C_2, \ldots, C_k$ sharing an edge $uv$, and is a graph with tree-width at most $2r$. Hence there is a linear-time algorithm to determine if the multfold colouring $f$ desribed in Theorem \ref{thm-charactIC} exists. Thus we have the following corollary.

\begin{cor}
    \label{cor-alg}
Let $r$ be a fixed positive integer.    Given an $r$-regular expanded Gallai-tree,   there is a linear-time algorithm to determine if $G$ is an IC-Brooks graph.
\end{cor}

   \section{Regular expanded Gallai-trees and IC-Brooks graphs}
   
In this section, we show that if $r \le 3$, then every $r$-regular expanded Gallai-tree is an IC-Brooks graph. For $r \ge 4$, there are $r$-regular expanded Gallai-trees that are not IC-Brooks graphs. On the other hand, any expanded Gallai-tree is an induced subgraph of an IC-Brooks graphs.

\begin{theorem}
    \label{thm-r-regularforsmallr}
   If $r \le 3$, and $G$ is an expanded Gallai-tree with maximum degree at most $r$, then there is a degree-list assignment $L$ of $G$ such that $(G,L)$ is infeasible and $L(v) \subseteq [r]$. In particular, every $r$-regular expanded Gallai-trees is an IC-Brooks graph.
\end{theorem}
\begin{proof}
    The case $r \le 2$ is trivial. Assume $r=3$. By Theorem \ref{thm-charact} and Lemma \ref{lem-k}, it suffices to show that the required multifold colouring $f$ of $H$ exists, where $H$ consists of at most two odd cycles sharing one edge. If $H$ is an odd cycle, then since $H[p]$ has maximum degree at most 3, we have $p(v) \le 2$ for each vertex $v$, and if $p(v)=2$, then $p(x)=1$ for each $x \in N_{H^2}(v)$. Thus the problem is to find a proper colouring of a graph consisting of isolated vertices with a single colour $3$, which trivially exists. 
    
    If $H$ consists of two odd cycles sharing one edge $uv$, then $p(x)=1$ for all  $x \in N_H[\{u,v\}]$. Similarly as above, the problem is to find a proper colouring of a graph consisting of isolated vertices with a single colour $3$.
\end{proof}

\begin{proposition}
    For $r \ge 4$, the conclusion of Theorem \ref{thm-r-regularforsmallr} does not hold,
and there is an $r$-regular expanded Gallai-tree that is not an IC-Brooks graph.
\end{proposition}
\begin{proof}
We denote by $C_k$ the $k$-cycle with vertices $v_1,v_2, \ldots, v_k$.

Assume $r \ge 4$. If $r = 3k+1$, then  let $G=C_5[p]$, where $p(v_i)=k+1$ for $i=1,3,4$ and $p(v_i)=k$ for $i=2,5$. It is easy to check that the multifold colouring described in Lemma \ref{lem-k} for this graph  does not exist. Note that each vertex in $G$ has degree $r$, except that   vertices in $K(v_1)$  have degree $r-1$.  Take the union of a copy of $K_{2k+1}$  and $2k+1$ copies of  $G$, and add edges connect the $i$th vertex of $K_{2k+1}$ to the vertices in $K(v_1)$ in the $i$th copy of $G$. See Fig. \ref{fig-14}(a).  The resulting graph is an $r$-regular expanded Gallai-tree which is not an IC-Brooks graph.

If $r= 3k+2$, then 
let $G=C_5[k+1]$. Then $G$ is an $r$-regular expanded Gallai-tree. A degree-list assignment $L$ for $G$ for which $(G,L)$ be infeasible is equivalent to a $k$-fold colouring $\phi$ of $K_5$ using colours from $[r]-\{1,2\}$. It is easy to check that such a colouring does not exist. 

If $r = 3k$ and $k \ge 2$, then let $G=C_7[p]$, where 
  $p(v_i)=k$ for $i=1,2,4,5,7$ and $L(v_i)=k+1$ for $i=3,6$. Then each vertex of $G$ has degree $r$, except that the vertices
 in $K(v_1)$ have degree $r-1$. A degree-list assignment of $G$ for which $(G,L)$ is infeasible and $L(x) \subseteq [r]$ is equivalent to an $f$-fold colouring of $C_7^2$ with $|f(x)| = p(x)-1$ and $f(x) \subseteq [r]-\{1,2\}$.   It is again easy to check that there is no such colouring. Take $3$ copies of $G$, add a vertex which is adjacent to $3$   copies of $K(v_1)$. See Fig. \ref{fig-14}(b). The resulting graph is an $r$-regular expanded Gallai-tree which is not an IC-Brooks graph. 
  \end{proof}

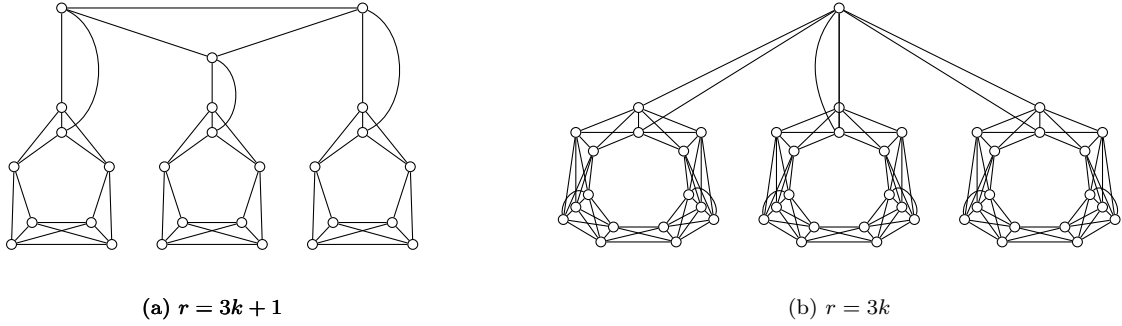
\begin{figure}
    \centering
\begin{tikzpicture}[scale=0.33, 
    dot/.style={circle, fill=white, draw=black, minimum size=1.3mm, inner sep=0pt}
  ]

\begin{scope}[xshift=-11cm]
  % 偏移量
  \def\dx{6}  % 每个 G_i 之间的水平间距

  % ========== 画 G_1, G_2, G_3 ==========
  \foreach \i in {1,2,3} {
    % 计算偏移量
    \pgfmathsetmacro{\shiftx}{(\i-2)*\dx} % 让 G_2 居中，G_1 在左，G_3 在右

    % 五边形的五个顶点
    \node[dot] (v1\i) at (\shiftx+0.00, 2.00) {};
    \node[dot] (v2\i) at (\shiftx+1.90, 0.62) {};
    \node[dot] (v3\i) at (\shiftx+1.18, -1.62) {};
    \node[dot] (v4\i) at (\shiftx-1.18, -1.62) {};
    \node[dot] (v5\i) at (\shiftx-1.90, 0.62) {};

    % 插入特殊点 v_1', v_3', v_4'
    \node[dot] (v1p\i) at (\shiftx+0.00, 3) {}; 
    \node[dot] (v3p\i) at (\shiftx+2, -2.50) {}; 
    \node[dot] (v4p\i) at (\shiftx-2.00, -2.50) {}; 

    % 连接默认点的边（五边形的边）
    \draw (v1\i) -- (v2\i) -- (v3\i) -- (v4\i) -- (v5\i) -- (v1\i);

    % 连接 v_1' 到 v_1, v_5, v_2
    \draw (v1p\i) -- (v1\i);
    \draw (v1p\i) -- (v5\i);
    \draw (v1p\i) -- (v2\i);

    % 连接 v_3' 到 v_3, v_2, v_4
    \draw (v3p\i) -- (v3\i);
    \draw (v3p\i) -- (v2\i);
    \draw (v3p\i) -- (v4\i);
    \draw (v3p\i) -- (v4p\i);

    % 连接 v_4' 到 v_4, v_3, v_5
    \draw (v4p\i) -- (v4\i);
    \draw (v4p\i) -- (v3\i);
    \draw (v4p\i) -- (v5\i);
  }

  % ========== 画 x_1, x_2, x_3 ==========
  % 额外点 x_1, x_2, x_3，放在 G_1, G_2, G_3 的上方
  \node[dot] (x1) at (-6, 7) {};
  \node[dot] (x2) at (0, 5) {};
  \node[dot] (x3) at (6, 7) {};

  % 连接 x_i 之间的边 (x_1, x_2, x_3 形成一个三角形)
  \draw (x1) -- (x2) -- (x3) -- (x1);

  % 连接 x_i 到它们对应的 G_i
  \foreach \i in {1,2,3} {
    \pgfmathsetmacro{\shiftx}{(\i-2)*\dx} 
    \draw (x\i) -- (v1p\i);
    \draw (x\i) to[out=-25,in=25] (v1\i);
    \node[below=1.4]{\scriptsize(a)  $r=3k+1$};
  }
  \end{scope}

  \begin{scope}[xshift=14cm] 

  \node[dot] (x) at (0, 7) {};

    % 设置平移量
  \def\dx{8}  % 每个副本的水平间距
  
  \foreach \i in {1,2,3} {
    % 计算偏移量
    \pgfmathsetmacro{\shiftx}{(\i-2)*\dx} % 让 G_2 居中，G_1 在左，G_3 在右
    
    % 在副本的不同位置绘制图形
    \begin{scope}[xshift=\shiftx cm]
    
       % 默认点（七边形的七个顶点）
  \node[dot] (v1) at (0.00, 2.00) {};
  \node[dot] (v2) at (1.80, 1.25) {};
  \node[dot] (v3) at (2.00, -0.50) {};
  \node[dot] (v4) at (1.00, -1.80) {};
  \node[dot] (v5) at (-1.00, -1.80) {};
  \node[dot] (v6) at (-2.00, -0.50) {};
  \node[dot] (v7) at (-1.80, 1.25) {};

  % 在指定点附近增加副本
  \node[dot] (v1p) at (0.00, 3) {};
  \node[dot] (v2p) at (2.5, 2) {};
  \node[dot] (v4p) at (1.5, -2.4) {};
  \node[dot] (v5p) at (-1.5, -2.40) {};
  \node[dot] (v7p) at (-2.5, 2) {};
  \node[dot] (v3p1) at (2.5, -1) {};
  \node[dot] (v3p2) at (3, -1.5) {};
  \node[dot] (v6p1) at (-2.50, -1) {};
  \node[dot] (v6p2) at (-3, -1.50) {};

  % 连接默认点的边（七边形的边）
  \draw  (v1) -- (v2) -- (v3) -- (v4) -- (v5) -- (v6) -- (v7) -- (v1);

  % 连接副本点到各自的点
  \draw  (v1) -- (v1p);
  \draw  (v2) -- (v2p);
  \draw  (v4) -- (v4p);
  \draw  (v5) -- (v5p);
  \draw  (v7) -- (v7p);
  \draw  (v3) -- (v3p1);
  \draw  (v3)to[out=50,in=70](v3p2);
  \draw  (v3p1) -- (v3p2);
  \draw  (v6) -- (v6p1);
  \draw  (v6) to[out=150,in=100] (v6p2);
  \draw  (v6p1)--(v6p2);

  % 连接 V_i 中每个点到 V_{i+1} 中每个点
  \draw  (v1) -- (v2p);
  \draw  (v1p) -- (v2);
  \draw  (v1p) -- (v2p);
  
  \draw  (v2) -- (v3p1);
  \draw  (v2) -- (v3p2);
  \draw  (v2p) -- (v3);
  \draw  (v2p) -- (v3p1);
  \draw  (v2p) -- (v3p2);
  
  \draw  (v3) -- (v4p);
  \draw  (v3p1) -- (v4);
  \draw  (v3p1) -- (v4p);
  \draw  (v3p2) -- (v4);
  \draw  (v3p2) -- (v4p);
  
  \draw  (v4) -- (v5p);
  \draw  (v4p) -- (v5);
  \draw  (v4p) -- (v5p);
  
  \draw  (v5) -- (v6p1);
   \draw  (v5) -- (v6p2);
   \draw  (v5p) -- (v6);
   \draw  (v5p) -- (v6p1);
  \draw  (v5p) -- (v6p2);
  
  \draw  (v6) -- (v7p);
  \draw  (v6p1) -- (v7);
  \draw  (v6p1) -- (v7p);
  \draw  (v6p2) -- (v7);
  \draw  (v6p2) -- (v7p);
  
  \draw  (v7) -- (v1p);
  \draw  (v7p) -- (v1);
  \draw  (v7p) -- (v1p);
  \draw  (x) -- (v1p);
  \draw  (x) --(v1);
  \end{scope}}

\draw  (x)to[out=230,in=120](-0.18, 2.10) ;

  \node[below=1.4]{\scriptsize (b)  $r=3k$};  \end{scope}
\end{tikzpicture}
\caption{The regular expanded Gallai-tree for $r=3k+1$ and $r=3k$} 
\label{fig-14}
\end{figure}

 \begin{theorem}
    \label{thm-subg}
    If $G$ is an expanded Gallai-tree, then there is an IC-Brooks graph $G'$ such that $G$ is an induced subgraph of $G'$. 
\end{theorem}
\begin{proof}
    Assume $G$ is not indicated degree-choosable, and $L$ is a degree-list assignment of $G$ such that $(G,L)$ is infeasible. Let $r=2k+1$ be an odd integer such that $\cup_{v \in V(G)}L(v) \subseteq [r]$. Let $H'$ be the expanded theta graph consisting of paths $P_1,P_2, P_3, \ldots, P_r$
    connecting two vertices $u$ and $v$, where $P_1$ is a single edge, $P_2=[u, w, v]$ has length 2, and $P_i=[u, w_{i,1}, w_{i,2}, w_{i,3}, v]$ has length $4$ for $i=3,4,\ldots, r$. Let $H$
    be obtained from $H'$ by duplicating $w_{i,2}$ into an $(r-1)$-clique, for $i=3,4,\ldots, r$. Then each vertex of $H$ has degree $r$, except that $w$ has degree $2$. Let $A=\{a,b\}$ be any arbitrary 2-subset of $\{1,2,\ldots, r\}$, and let $L(x)=[r]$ for $x \in V(H)-\{w\}$ and $L(w)=A$. It is easy to verify that $(H,L)$ is infeasible. 

    Take $k$ copies of $H$, and one copy of $K_2$ with $V(K_2)=\{s,t\}$. Identify $s$ and all the $k$ copies of $w$ in the $k$ copies of $H$ into a single vertex. Denote the resulting graph by $Q$. Then each vertex of $Q$ has degree $r$, except that vertex $t$ has degree $1$. Let $B$ be any 1-subset of $[r]$, and let $L'(x)=[r]$ for all vertices $x \in V(Q)-\{t\}$ and $L'(t)=B$. Again it is easy to verify that $(Q,L')$ is infeasible. 

    Now for each vertex $x$ of $G$, take $r-d_G(x)$ copies of $Q$, identify $x$ with the $r-d_G(x)$ copies of $t$ in the copies of $Q$. The resulting graph $G'$ is $r$-regular, and it follows from the discussion above that $G'$ is an IC-Brooks graph. %\xztodo{Better have a figure for this as well.}
\end{proof}

It was proved in \cite{KKKO} that every IC-Brooks graph $G$ has   $\omega(G) \ge \lceil \frac{\Delta(G)}{2} \rceil+1$. The following result shows that $G$ has   $\omega(G) \ge \lceil \frac{2(\Delta(G)+1)}{3} \rceil$, and this upper bound on $\omega(G)$ is sharp.

\begin{theorem}
    \label{thm-omega}
    If $G$ is an $r$-regular expanded Gallai-tree, then $\omega(G) \ge 2(r+1)/3$. In particular, if $G$ is an IC-Brooks graph, then 
    $\omega(G) \ge \frac 23 (\Delta(G)+1)$.
\end{theorem}
\begin{proof}
Assume $G$ is an $r$-regular expanded Gallai-tree. 
Let $H$ be an expanded leaf block of $G$. If $H$ is a complete graph, then we are done. Otherwise, $H=C[p]$ is a blow-up of an odd cycle $C$ of length at least 5. 
Assume that the vertices of $C$ are $v_1,v_2,\ldots, v_{2k+1}$ and $v_3,v_4, \ldots, v_{2k+1}$ are not contained in any other expanded blocks. For $G$ to be $r$-regular, we must have $p(v_i)=p(v_{i+3})$ for $i=2,3,\ldots, 2k-1$, and $r= p(v_2)+p(v_3)+p(v_4)-1$ and $\omega(G) \ge \max\{ p(v_2)+p(v_3), p(v_3)+p(v_4), p(v_4)+p(v_5)\} \ge \frac 23 (r+1)$. 
\end{proof}

By (6) of Example 1, the bound $\omega(G) \ge \frac 23 (\Delta(G)+1)$ for Brooks graph for indicated chromatic number is tight.

% \bibliographystyle{abbrv}
%\bibliography{ref}

%\end{document}

\end{document}